\documentclass[reqno]{amsart}
\usepackage{amssymb,amsmath,amsthm,amstext,amsfonts}
\usepackage[dvips]{graphicx}
\usepackage{psfrag}
\usepackage{color}
\usepackage[colorlinks=true,linkcolor=blue, urlcolor=red, citecolor=blue]{hyperref}	

\makeatletter \@addtoreset{equation}{section} \makeatother

\renewcommand\thetable{\thesection.\@arabic\c@table}

\theoremstyle{plain}
\newtheorem{maintheorem}{Theorem}

\newtheorem{maincorollary}{Corollary}
\newtheorem{theorem}{Theorem }[section]
\newtheorem{proposition}[theorem]{Proposition}
\newtheorem{lemma}[theorem]{Lemma}
\newtheorem{corollary}[theorem]{Corollary}
\newtheorem{claim}{Claim}
\theoremstyle{definition} \theoremstyle{remark}
\newtheorem{remark}[theorem]{Remark}

\newtheorem{definition}[theorem]{Definition}

\newcommand{\field}[1]{\mathbb{#1}}
\newcommand{\real}{\field{R}}

\newcommand{\torus}{\field{T}}

\newcommand{\al} {\alpha}       
        
\newcommand{\ga} {\gamma}    
\newcommand{\de} {\delta}       \newcommand{\De}{\Delta}
\newcommand{\ep} {\vep}
\newcommand{\vep}{\varepsilon}

\newcommand{\te} {\theta}       \newcommand{\Te}{\Theta}

\newcommand{\ka} {\kappa}
\newcommand{\la} {\lambda}      \newcommand{\La}{\Lambda}

\newcommand{\si} {\sigma}

\newcommand{\om} {\omega}       \newcommand{\Om}{\Omega}

\newcommand{\Z}{\mathbb{Z}}
\newcommand{\N}{\mathbb{N}}
\newcommand{\R}{\mathbb{R}}
\newcommand{\supp}{\operatorname{supp}}
\newcommand{\diam}{\operatorname{diam}}
\newcommand{\dist}{\operatorname{dist}}

\newcommand{\ti}{\tilde }

\newcommand{\un}{\underline}
\newcommand{\cI}{{\mathcal I}}
\newcommand{\mL}{{\mathcal L}}
\newcommand{\mM}{{\mathcal M}}

\newcommand{\cP}{\mathcal{P}}
\newcommand{\cL}{\mathcal{L}}
\newcommand{\cE}{\mathcal{E}}
\newcommand{\cM}{\mathcal{M}}
\newcommand{\cB}{\mathcal{B}}

\newcommand{\cG}{\mathcal{G}}

\newcommand{\cQ}{\mathcal{Q}}
\newcommand{\cF}{\mathcal{F}}

\newcommand{\cU}{\mathcal{U}}
\newcommand{\cS}{\mathcal{S}}
\newcommand{\cA}{\mathcal{A}}

\newcommand{\ub}{\underbar}
\newcommand{\mP}{{\mathbb P}}

\begin{document}

\title{Thermodynamic formalism for random non-uniformly expanding maps}
\author{Manuel Stadlbauer, Shintaro Suzuki and Paulo Varandas}

\address{Manuel Stadlbauer, Departamento de Matem\'atica, Universidade Federal do Rio de Janeiro\\
Rio de Janeiro, Brazil}
\email{manuel.stadlbauer@gmail.com}

\address{Shintaro Suzuki, 
Keio Institute of Pure and Applied Sciences (KiPAS), Keio University\\ Yokohama, Japan}
\email{shin-suzuki@math.keio.ac.jp}

\address{Paulo Varandas, Departamento de Matem\'atica, Universidade Federal da Bahia\\
Av. Ademar de Barros s/n, 40170-110 Salvador, Brazil; and Centro de Matem\'atica 
da Universidade do Porto, Rua do Campo Alegre, Porto}
\email{paulo.varandas@ufba.br}

\date{\today}

\maketitle
\begin{abstract}
We develop a quenched thermodynamic formalism for a wide class of random maps with non-uniform
expansion, where no Markov structure, no uniformly bounded degree or 
the existence of some expanding dynamics is required. 
We prove that every measurable and fibered $C^1$-potential at high temperature admits a unique equilibrium state which satisfies a weak Gibbs property, 
and has exponential decay of correlations. 
The arguments combine a functional analytic approach for the decay of correlations (using Birkhoff cone methods)
and Carath\'eodory-type structures to describe the 
relative pressure of not necessary compact invariant sets in random dynamical systems. We establish also 
a variational principle for the relative pressure of random dynamical systems.
\end{abstract}

\tableofcontents

\section{Introduction}

The thermodynamic formalism of smooth dynamical systems was
initiated in the mid seventies by Sinai, Ruelle and Bowen 
both for uniformly hyperbolic diffeomorphisms and flows. They proved
that, restricted to every basic piece of the non-wandering set,
equilibrium states exist and are unique for every H\"older continuous potential.
The basic strategy to prove this remarkable fact was to (semi)conjugate the dynamics to a
subshift of finite type, via a Markov partition. In fact, the construction of equilibrium states and their statistical properties in the uniformly hyperbolic setting follows from a two-step reduction.
Firstly, the hyperbolic map is codified to the symbolic dynamics. Secondly, equilibrium states for the two sided subshift of finite type are obtained by the construction of Gibbs measures associated to the expanding (one sided) subshift (see e.g. \cite{Bo75}). 

\smallskip
The extension of the thermodynamic formalism for random dynamical systems, which are often used to describe physical models evolving with time, has been developed since the nineties. 
Nevertheless, most  generalizations of the classical thermodynamic formalism developed to random dynamical systems deals with distance expanding maps (e.g. random shifts) even when 
 the sample space is non-compact. Indeed,  Kifer \cite{Ki92,Ki08} first proved that
equilibrium states associated to H\"older potentials exist and are unique in the case of smooth expanding dynamics. Moreover, such equilibrium states satisfy a fiber Gibbs property and have fiber exponential decay of correlations. Bogensch\"utz and Gundlach \cite{BG95} established a 
Ruelle theorem for random subshifts of finite type. There are also contributions to the thermodynamic 
formalism of countable Markov shifts, by Denker, Kifer and Stadlbauer~\cite{DKS08}, Stadlbauer~\cite{Stadl,Stadl2}, 
Mayer and Urba{\'n}ski~\cite{MU15}, of random countable iterated function systems with overlaps, by Mihailescu and Urba{\'n}ski \cite{MiU18}, 
and random dynamics of hyperbolic entire and meromorphic functions of finite order satisfying a growth condition at infinity, by Mayer and
Urba{\'n}ski ~\cite{MU18}, among many others. 

\smallskip
Several important difficulties arise when trying to extend this
theory beyond the uniformly hyperbolic setting. First it is worth mentioning that,
while non-uniform hyperbolicity can be defined in terms of non-zero 
Lyapunov exponents for invariant measures, any random map determined by local diffeomorphisms  
whose Lyapunov exponents are all positive for every invariant measure is a random expanding map \cite{CLR}.
Hence, one should not only be able to deal with dynamics having coexistence of hyperbolic and non-hyperbolic
invariant measures, but also to compare their free energies. 
Other difficulties may arise from the existence of critical or singular behavior,  
or discontinuities for the generating dynamics. In what follows we recall some important contributions in this direction. 

\smallskip
Khanin and Kifer~\cite{KhKi96} and Mayer, Skorulski and Urba{\'n}ski ~\cite{MSU08} considered the 
context of maps which expand in average and which are, in the context of random smooth dynamical systems $(f_\om)_\om$ acting on a compact manifold and driven by a noise $\mP$, defined by
$$
\int \log \|Df_\om^{-1}\|_\infty \, d\mathbb P(\om) < 0.
$$
These contributions were landmarks, as they included the first examples of non-uniformly expanding random dynamical systems
(see e.g. ~\cite{KhKi96} for the concept of non-uniform expansion used there) and still the equilibrium states 
have exponential decay of correlations. However, the previous condition still requires a positive $\mathbb P$-measure subset of
dynamics to be uniformly expanding.
Later, Arbieto, Matheus and Oliveira ~\cite{AMO}, used upper-semicontinuity of the entropy function among a certain class of invariant 
measures, having only positive Lyapunov exponents, to construct equilibrium states associated to H\"older potentials at large temperature
for random non-uniformly expanding  maps in a $C^2$-neighborhood of deterministic non-uniformly expanding maps (in the sense of \cite{BCV,VV10}).
In particular, measures of maximal entropy do exist and, under a transitivity assumption, it is unique (cf.~\cite{BO17}).  
Much more recently, Atnip et al \cite{Atnip} constructed equilibrium states for random covering interval maps (a condition defined in terms of the potential and transfer operators acting on BV spaces), covering important examples as random beta-maps and random Liverani-Saussol-Vaienti maps.

\smallskip

Here we develop a thermodynamic formalism for a broad class of non-uniformly expanding random dynamical systems, acting on
a compact smooth manifold of dimension $d\geqslant 1$. We require no Markov structure, all maps generating the random dynamical system may fail to be uniformly expanding, and the (locally well defined) number of preimages may be unbounded. 
Our requirements are of combinatorial type and demand essentially that the average of the weights of regions with possible contraction does not supersede 
the one associated with regions with expanding behavior (cf. \eqref{eq:H3} for the precise statement).
Roughly speaking, we prove that any random transformation for which there is a %
combinatorial expansion in average (which is determined combinatorially in terms of a weighted average of the contracting behavior
by inverse branches) and that any %
smooth potential at high temperature admits a unique quenched equilibrium state, which is a non-uniform $\om$-wise Gibbs state has exponential decay of correlations. 
We refer the reader to the Section~\ref{s.statements} for the precise statements.
Random dynamical systems satisfying the previous assumptions include, as particular examples, 
the classes of random dynamical systems considered in ~\cite{AMO,BaY93, Bo92, DKS08, KhKi96,Ki92, MSU08}
as described in Section~\ref{s.examples}. 
Summarizing, we consider random dynamical systems where:
\begin{itemize}
\item[$\circ$] no Markov assumption is required;
\item[$\circ$] all generating dynamics may be multidimensional and non-hyperbolic;
\item[$\circ$] regions with lack of hyperbolicity may be non-localized (i.e. the dynamics are not
	necessarily perturbations of a deterministic dynamical system);
\item[$\circ$] the degree of the generating dynamics may be unbounded.
\end{itemize}

This work is organized as follows. The precise statement of our results are given in Section~\ref{s.statements}, while 
Section~\ref{s.examples} is devoted to %
applications of our main results.
In Section~\ref{sec:Prelim} we describe basic concepts of random (quenched) thermodynamic formalism and 
non-uniform expansion. 
The proofs of the main results start in Section~\ref{sec:BKM} with the proof of the spectral gap property and the construction of 
invariant densities and conformal measures for the Ruelle-Perron-Frobenius transfer operators.  The latter arises as 
consequences of the strict invariance of a suitable cone on the Banach space of $C^r$ functions 
together with Birkhoff's contraction theorem.  
In Section~\ref{s.expanding} we prove that the previous measures (obtained by spectral methods) enjoy non-uniform
expansion along the orbits of the random dynamical system. 
The proof of the uniqueness of equilibrium states for hyperbolic potentials is derived as a consequence 
of an extension of Ledrappier and Young's
uniqueness of equilibrium states that we extended to the realm of random dynamical systems. 

This article has two appendices which are of independent interest. 
In Appendix~A (Section~\ref{AppendixA}) we present a self-contained account on %
the concept of relative topological pressure for random dynamical systems, and where we
establish a %
variational principle for the relative pressure. 
In Appendix~B (Section~\ref{AppendixB}) we prove that whenever there exists an expanding conformal measure $\nu$ then all expanding equilibrium states giving full weight to the support of $\nu$ are absolutely continuous with respect to the conformal measure.

\section{Setting and statement of the main results}\label{s.statements}

\subsection*{Random dynamical systems}

Let $M$ be a compact metric space with distance $d$ and denote by $\cB$ the Borel $\sigma$-algebra, let $(\Om, \cF, \mP)$ be a Lebesgue space i.e., it is measurably isomorphic to an interval with the completion of the Borel $\sigma$-algebra and the Lebesgue measure on it with countably or finitely many atoms. Consider a $\mP$-preserving, measurable and invertible transformation  $\te$ on $\Om$ and let $X\subset\Om\times M$ be a measurable set with respect to the product $\sigma$-algebra $\cF\times\cB$ such that the fibers $X_\om=\{x\in M:  (\om,x)\in X\}$ of $X$ are compact for each $\om\in\Om$. A continuous bundle random dynamical system $f=(f_\om)_\om$ is generated by continuous maps $f_\om:X_\om\to X_{\te\om}$ such that $(\om,x)\mapsto f_{\om}(x)$ is measurable. Then we define $f^0_\om=id$ 
and
\begin{equation}\label{e.cociclo}
f^n_\om=f_{\theta^{n-1}(\om)}\circ \cdots \circ f_\omega %
	\text{  for  } n\geqslant 1.
\end{equation}
These random iterations induce the random bundle transformation
$$
F: X \to X \quad \text{given by}\quad
F(\om,x)=(\theta(\omega), f_\omega(x)).
$$
Let $L^1_X(\Omega,C(M))$ be the set of families $\phi=(\phi_\om)_\om $ of continuous maps $\phi_\omega: X_\om \to\mathbb R$
such that $(\omega,x)\rightarrow \phi_w(x)$ is measurable and $\|\phi\|_1:=\int_{\Om} |\phi_w|_{\infty}d\mP(w) <+\infty$. 
%
%

Throughout this paper we will always assume that $M$ is a compact and connected $m$-dimensional Riemannian manifold with distance $d$ and that $(\Om,\cF, \mP)$ is
a Lebesgue space. Let $\te$ be an invertible $\mP$-preserving map on $\Om$ and assume that $\te$ is ergodic. 
Let $X\subset \Om\times M$ be a measurable set such that the fibers $X_\om$ are compact and connected. 

\medskip
Assume that $f=(f_\om)_\om$ is a family of $C^1$-local diffeomorphisms 
$f_\om:X_\om\to X_{\theta(\om)}$ satisfying the following conditions:
\begin{enumerate}
\item[(H0a)] The map $(\om, x)\mapsto f_\om(x)$ is a measurable map from $X$ to $X_{\theta\om}$,
\item[(H0b)] There are $\delta_\om>\de_0>0$ such that for every $(\om,x)\in X$ there exists an open neighborhood $U_x$ of $x$  with  
$f_\om|_{U_x}:U_x\to B(f_\om(x), \delta_\om)$  invertible,
\item[(H0c)] $\int \log \max_{x\in X_\om }\|Df_\om(x)^{-1}\|^{-1} \, d\mathbb P <\infty$. 
\end{enumerate}
In particular, every point in $X_{\theta\om}$ has the same finite number of preimages $\deg(f_\om)$ for each $\om\in\Om$.
Moreover, by the third condition and Kingman's subadditive ergodic theorem, even though the norm of the derivatives $Df_\om$ may be unbounded, the smallest Lyapunov exponent associated to an invariant probability of the random dynamical system is finite. 

We also assume the following geometric conditions for the random dynamical system $f$: there 
are random variables $\sigma_\om>1, L_\om  \geqslant0$ and $0\leqslant p_\om, q_\om<\deg(f_\om) $ so that
$p_\om+ q_\om=\deg(f_\om)$ and:
\begin{enumerate}
\item[(H1)] There exists a covering $\cP_\omega^{0}=\{P_1, \dots,P_{p_\om}, \dots, P_{p_\om+q_\om}\}$ of
      $X_\omega$, depending measurably on $\omega$, such that every $f_\om|P_i$ is injective, $\|Df_\om(x)^{-1}\| \leqslant \si_\om^{-1}<1$
      for every $x\in P_{1} \cup \dots \cup P_{p_\om}$, and $\|Df_\om(x)^{-1}\|\leqslant L_\om$ for
       every $x\in X_\omega$,
\item[(H2)] The functions $\log \sigma_\om, \log L_\om, \log p_\om, \log q_\om$ and $\log\deg(f_\om)$ belong to $L^1(\mP)$, 
\item[(H3)] $\log(\sigma_\om^{-1}p_\om+L_\om q_\om)\in L^1(\mP)$ and 
\begin{equation}\label{eq:H3}
\int_\Om \log\Bigl(\frac{\sigma_\om^{-1}p_\om+L_\om q_\om}{\deg(f_\om)}\Bigr)d\mP<0.
\end{equation}
\end{enumerate}

Let us comment on the assumptions on the dynamics. Conditions (H0a) and (H2) are the natural measurability and integrability
requirements, while condition (H0b) says that all maps $f_\om$ are locally invertible and inverse branches have a definite size
(in particular $\diam (X_\om)$ is bounded away from zero). 
The random dynamical system is driven by expanding maps whenever $q_\omega=0$ or $L_\om<1$ for $\mathbb P$-almost every 
$\om$. More generally,
(H1) and (H3) imply that regions of uniform expansion and regions of some contraction may coexist 
for $\mathbb P$-almost every dynamics $f_\omega$ but that, combinatorially, the backward expansion is prevalent in average
(cf. equation \eqref{eq:H3}). In order to 
prove uniqueness of equilibrium states we furthermore assume that
\begin{enumerate}
\item[(H4)] for every $\vep>0$ there exists $\zeta>0$ such that $\|Df_\om^{-1}(x)\| \leqslant e^\vep \|Df_\om^{-1}(y)\|$
for every $y\in B(x,\zeta)\subset X_\om$ and $\mP$-a.e. $\om$,
\end{enumerate}
and that one of the following properties hold:
\begin{enumerate}
\item[(H5)] for every $\vep>0$ there exists $N(\vep)\geqslant 1$ so that 
	$f_\om^{N(\vep)}(B(x,\vep))= X_{\theta^{N(\vep)}\om}$ for every $x \in X_\om$ and $\mathbb P$-almost every $\om$; \; or
\item[(H5')] there exists $C>0$ so that %
	$\mathbb P-\text{ess sup} \|Df_\om(\cdot)^{-1}\|_\infty^{-1}\leqslant C <\infty$.
\end{enumerate}

The equicontinuity assumption (H4) will be used to prove that the random non-uniform hyperbolicity 
ensures the existence of hyperbolic times and, subsequently, that local unstable manifolds do exist for suitable points and
suitable inverse branches of the dynamics (cf. Subsection~\ref{sec:HT}). Hypothesis (H5) and (H5'), even though of completely 
different nature, will be crucial to guarantee that the equilibrium measure satisfies a weak Gibbs property, which then implies its uniqueness. Both the topological exactness and bounded derivative conditions are satisfied by natural
classes of examples, even $C^1$-robustly (see Section~\ref{s.examples}).
\color{black}

\medskip
Now we shall describe the class of potentials considered here. Let  $\phi=(\phi_\om)_\om$ be a potential in $L^1_X(\Om,C^1(M))$ 
(meaning that $\phi_\om : M \to \mathbb R$ is smooth for $\mathbb P$-a.e. $\om$ and $\log \|D\phi_\om\|_\infty\in L^1(\mP)$) so that
 \begin{align*}
 (P) \hspace{2cm}  \int_\Om(\sup\phi_\om-\inf\phi_\om)d\mP +&\int_\Om\log(1+\|D\phi_\om\|_\infty \diam(X_\om))d\mP \\
&<-\int_\Om\log\Bigl(\frac{\sigma_\om^{-1} p_\om+L_\om q_\om}{\deg(f_\om)}\Bigr)d\mP. \nonumber
\end{align*}
\noindent Condition (P) is satisfied by all families of potentials $\phi$ that 
are close to the zero potential in the $L^1_X(\Om,C^1(M))$-topology. 
It says that the mean oscillation of the 
potential $\phi$ is bounded by the average combinatorial expansion given by \eqref{eq:H3}.  In this context $\int \log \deg(f_\om) \,d\mP$ appears naturally  as the topological entropy of the random dynamical system (cf. Remark~\ref{r.entropy}). 
Moreover, we would like to point out that the assertion of  Corollary \ref{cor:decay} provides a criterion for the potentials we may not consider in here. 
For example, as intermittent maps satisfy (H1) to (H3) and it is known that these maps might show an at most polynomial decay of correlations with respect to the invariant probability which is absolutely continuous with respect to Lebesgue measure, it follows that the geometric potential $- \log |Df_\omega|$ cannot be included in our analysis.

\begin{remark}
If
$\phi\in L^1_X(\Om,C^1(M))$ satisfies the integrability conditions 
$$
\log \|D\phi_\om\|_\infty\in L^1(\mP) \quad \text{and} \quad
\log( 1+\|D\phi_\om\|_\infty\diam(X_\omega) ) \in L^1(\mP)
$$ 
then $\frac1\beta \phi$ satisfies (P)  for every large $|\beta|$.
This is known in the physics literature as high temperature regime.
\end{remark}

\subsection*{Statement of the main results}

The first main result here is the following version of Ruelle's theorem for the previous class of random dynamical systems.
If $\mathcal B$ denotes the Borel $\sigma$-algebra on $M$, a random probability measure on $M$ is a measurable function $\mu: \Omega \times M \to [0,1]$ so that: (i) for $\mathbb P$-almost every $\omega\in \Omega$ the function $\mathcal B \ni B \mapsto \mu(\omega,B)$ defines a probability measure, and 
(ii) for every $B\in \mathcal B$ the function $\Omega \ni \omega \mapsto \mu(\omega,B)$ is measurable. For simplicity we use the notation 
$\mu=(\mu_\omega)_{\omega}$
for a random probability.

\begin{maintheorem}\label{thm:Ruelle}
Let $f=(f_\om)_{\om\in\Om}$ be a family of $C^1$-local diffeomorphisms satisfying (H0)-(H3) and let $\phi=(\phi_\om)_{\om\in\Om}$ be a potential function in $L^1_X(\Om, C^1(M))$ satisfying (P). %
Then there exists a triple $(\lambda, h, \nu)$ consisting of a positive random variable $\lambda=(\lambda_\om)_\om$, a positive measurable function $h=(h_\om(\cdot))_\om \in L^1_X(\Om, C^1(M))$ and a probability measure $\nu=(\nu_{\omega})\in \cP(X)$ such that 
\[\cL_\om h_\om=\lambda_\om h_{\te\om} \text{\ and\ } \cL_{\om}^{*}\nu_{\te\om}=\lambda_\omega\nu_{\om}.\]
Moreover, the measure
$\mu=(\mu_\om)_\om$ defined by $\mu_\om=h_\om \nu_\om$ is an $F$-invariant probability measure.%
\end{maintheorem}

A probability $\nu=(\nu_\om)_\om$ as in Theorem~\ref{thm:Ruelle} is called a \emph{conformal measure}.
As a byproduct of our strategy we prove the following fibered exponential decay of correlations for smooth observables.

\begin{maincorollary}\label{cor:decay}
The $F$-invariant probability measure $\mu=(\mu_\om)_\om$ has fibered exponential decay of correlations for 
$C^1$-observables: there exists %
$\tau \in (0,1)$ such that for all $\varphi \in C^1(X_{\theta^n \om})$ and $\psi \in C^1(X_{\om})$ 
there is $K_\om(\varphi,\psi)>0$ so that
\begin{equation*}
\left|\int_{X_{\om}} (\varphi\circ f_\om^n)\psi \, d\mu_{\om} - \int_{X_{\te^n\om}} \varphi d\mu_{\theta^n\om} \, \int_{X_{\om}} \psi d\mu_{\om}
	\right|
	\leqslant K_\om(\varphi,\psi)\cdot\tau^n
	\quad \text{for every}\;  n \geqslant 1.
\end{equation*}
\end{maincorollary}

\medskip
It is natural to ask whether the invariant measure $\mu=(\mu_\om)_\om$, absolutely continuous with respect to the leading eigenmeasures $\nu=(\nu_\om)_\om$ of the transfer operators are indeed equilibrium states, and if these are unique.
An important step to answer both questions is to prove that the conformal measures $\nu=(\nu_\om)_\om$ 
satisfy a weak version of the Gibbs property on dynamic balls $B_\om(x,n,\vep)$
(see Subsection~\ref{subsec:rtf} for the definition).

\begin{maintheorem}\label{thm:weak.Gibbs}
Let $f=(f_\om)_{\om\in\Om}$ be a family of $C^1$-local diffeomorphisms satisfying (H0)-(H4). 
Assume that either (H5) or (H5') holds, and that $\phi=(\phi_\om)_{\om\in\Om}$ 
is a potential function in $L^1_X(\Om, C^1(M))$ satisfying (P). 
The probability $\nu=(\nu_\om)_\om$ given by Theorem~\ref{thm:Ruelle}
satisfies
a \emph{weak Gibbs property}: there exists $\vep_0>0$ such that for every $0<\vep<\vep_0$
there exists a random variable $K(\om) \geqslant 0$ with $\int K(\om) d\mathbb P < \infty$, and for $\nu$-almost every
$(\om,x)$ there exists a strictly increasing  %
sequence of integers $n_k=n_k(\om,x)\geqslant 1$ such that
\begin{equation}\label{eq:wGibbs}
\frac1{e^{\vep K(\theta^n\om) }}
	\leqslant \frac{\nu_\om(B_\om(x,n_k,\vep))}{\exp[\sum_{j=0}^{n_k-1} (\phi(F^j(\om,x)) -\log \la_{\theta^j(\omega)}) ]}
	\leqslant e^{\vep K(\theta^n\om) }
\end{equation}
for every %
$k\geqslant 1$.
\end{maintheorem}
\color{black}

In comparison to the case of deterministic dynamics one need not expect the previous sequences of
integers in \eqref{eq:wGibbs} to have positive density in the integers (see Remark~\ref{ht-density} for a more precise discussion).
Nevertheless, this formulation, together with the fact that the measure $\mu=(\mu_\om)_\om$ is absolutely
continuous with respect to the conformal measures, is sufficient to estimate the entropy of the 
probability $\mu$.

\medskip
In what follows we obtain that the probability $\mu=(\mu_\om)_\om$ is indeed an equilibrium state. In order to do so, 
we need to relate its entropy with other thermodynamic
quantities, namely topological and relative pressures.
At this point one considers separately the contribution to the pressure given by the not necessarily compact subset of
points with infinitely many instants of hyperbolicity, that is hyperbolic times as defined in \eqref{eq:RNUE-times},  associated to 
invariant measures having only positive Lyapunov exponents, and its complement.
More precisely, let $H=H(\alpha)$ ($\alpha>0$) denotes the set of points $(\om,x)\in X$ with infinitely many 
$\frac\alpha2$-hyperbolic times. A potential $\phi\in L^1(\Om, C^1(M))$ is called {hyperbolic} if  
$\pi_\phi(f,H^c) < \pi_\phi(f)$.  
We refer the reader to Subsection~\ref{subsec:rtf}, Subsection~\ref{sec:HT} and Section~\ref{s.temperature} for definitions 
and more details. The existence and uniqueness of equilibrium states among measures having only positive Lyapunov 
exponents is sumarized in the following:

\begin{maintheorem}\label{thm:existence.and.uniqueness}
Under the assumptions of Theorem~\ref{thm:weak.Gibbs},
any hyperbolic potential $\phi=(\phi_\om)_\om \in L^1(\Om, C^1(M))$ satisfying (P) %
has a unique equilibrium state $\mu$ for $f$ with respect to $\phi$, which is a weak Gibbs measure and has
exponential decay of correlations for $C^1$-observables.
\end{maintheorem}

\color{black}

Let us make some remarks on the hyperbolic potential assumption.
As in the classical motivation arising from statistical mechanics, high temperature regimes are 
often associated to uniqueness of equilibrium states. This is due to the fact that, under this condition  
almost all orbits tend to be associated to some hyperbolic behavior. The following result provides an extra
sufficient condition under which all potentials at high temperature are hyperbolic. 

\begin{maintheorem}\label{thm:HP}
Assume that the random dynamical system $f=(f_\om)_\om$ satsfies the assumptions of 
Theorem~\ref{thm:weak.Gibbs} and
\begin{equation}\tag{H6}
\int_\Om \log\Big( \frac{\sigma_\om^{-1}p_\om+L_\om q_\om}{\deg f_\om}\Big) \, d\mathbb P< - \dim M \cdot \int \log L_\om \, d\mathbb P.
\end{equation}
Then there exists $\alpha>0$ so that $\pi_0(f,H(\alpha)^c)<\pi_0(f)$. 
In particular, for every $\phi\in L^1(\Om, C^1(M))$ there exists $T_0>0$ so that 
$\frac1T \phi$ is a hyperbolic potential, for every $|T|>T_0$.
\end{maintheorem}

The situation where one considers a potential of the form $\frac1T \phi$, for some large $|T|$,  is known as  large temperature regime. An immediate consequence of Theorem~\ref{thm:HP} is that such random dynamical systems admit a unique measure of maximal entropy, exhibiting exponential decay of correlations.
The previous results pave the way to the quest for other fine ergodic properties of this large class of 
non-uniformly random dynamical systems, including large deviations, other limit theorems and linear response formulas.
We refer the reader to \cite{BRS,Ha,Ki91} and references therein for accounts on these directions. 

\color{black}

\subsection*{Overview of the arguments}

The strategy explores some ideas boroughed from the study of deterministic non-uniformly expanding maps
combined with new elements from the ergodic theory of random non-uniformly expanding dynamical systems.
However, in opposition to the deterministic case, the construction of the conformal measures at first (via fixed point
theorems) would demand to deal with measurability issues. For that reason we first prove a spectral-gap like property
and construct the conformal measures afterwards, as we now detail.  
   
First we prove that the random Ruelle-Perron-Frobenius operator $\cL_\om$
acts as a contraction on a suitable cone of positive smooth %
functions. The contraction rate, although just measurable, turns out to be sufficient to obtain 
an asymptotic contraction along random orbits. 
From this we deduce the existence of a unique eigenfunction $(h_\om)_\om$ for which the probability measure $\mu=(\mu_\om)_\om$ defined by  $\mu_\om =h_\om \nu_\om$ is $F$-invariant. Moreover, by standard estimates, the spectral gap property ensures 
the exponential decay of correlations. Replacing the original potential by one cohomologous to it, one can obtain a Markov-Feller %
transfer operator for which the unique fixed point $\nu$ has a natural and measurable disintegration $\nu=(\nu_\om)_\om$ 
(see Subsection~\ref{sec:eigen} for more details).
This overcomes the measurability issues related to the construction of conformal measures (see the discussion in \cite{Ki08,MSU08,Stadl2}).

The argument that the probability $\mu$, constructed via the Ruelle operator, is an equilibrium state for $f$ with respect to the potential $\phi$ is far from trivial. 
A first difficulty is to show that this measure enjoys some non-uniform hyperbolicity,
a crucial condition in all known approaches to thermodynamic formalism; this is done in Section~\ref{s.expanding}. 
While the Gibbs property in Theorem~\ref{thm:weak.Gibbs} suggests that 
$$
h_\mu(f) + \int \phi \, d\mu = \int \log \lambda_\om \, d\mathbb P,
$$
the argument leading to existence and uniqueness of equilibrium states has the following route.
\begin{enumerate}
\item[$\circ$] The topological pressure coincides with $\int \log \lambda_\om \, d\mathbb P$\,;
\item[$\circ$] all probability measures having no non-uniform hyperbolicity do not attain the pressure. 
\item[$\circ$] $\mu$ is the unique equilibrium state among the probability measures having some form of non-uniform hyperbolicity.
\end{enumerate}
A second difficulty is that the set of points with non-uniform hyperbolicity forms a non-compact invariant set, for which we could 
not find suitable notions of pressure or a variational principle. 
For that, we introduce a concept of Carath\'eodory-like relative pressure for random dynamical systems 
and prove a variational principle which may be of independent interest; this is done in Appendix A.

We also bound the relative pressure of the set of points with non-uniform hyperbolicity 
(these resemble Hausdorff dimension type of bounds) and allow to prove that invariant measures with large pressure have
necessarily non-uniform expansion along random orbits. 
In other words, assumption (P) on the potential $\phi$ ensures that the pressure is supported on the set of points displaying non-uniform expansion along its orbits. This is crucial to prove that the probability measure $\mu$ obtained in Theorem~\ref{thm:Ruelle} is 
a weak Gibbs measure associated to the pressure and, ultimately, an equilibrium state.

The statement on the uniqueness of equilibrium states is often a hard topic for dynamics
beyond the hyperbolic context. %
For that purpose we extend a classical result by Ledrappier and Young \cite{LY98} to the context of random dynamical systems, 
and deduce that every non-uniformly hyperbolic equilibrium state is necessarily absolutely continuous with respect to the conformal measures.
We found no previous results on this direction.

Finally, taking into account \cite{CV13,VV10}, we believe our results should %
extend to the context of random average distance expanding maps 
on compact metric spaces and H\"older continuous observables. We avoided the extra technicality to emphasize the difficulties which are already present in here.

\section{Examples}\label{s.examples}

In order to illustrate that our assumptions on the random dynamical system are mild, let us now provide examples 
where all the dynamics generating the random dynamics are not expanding.

\subsection{Random Manneville-Pomeau maps}\label{subsection:intermittent}

Let ${\mathbb S}^1=\mathbb R / \mathbb Z$ and let $\theta:{\mathbb S}^1 \to {\mathbb S}^1$ be the rotation
of angle $\alpha\notin \mathbb Q$ on the circle: $\theta(\om)=\om+\alpha$ for every $\om\in {\mathbb S}^1$. Let 
$\gamma: \mathbb S^1\to \mathbb S^1$ be a $C^1$-map, and $\mathbb P$
denote the Lebesgue measure on $\mathbb S^1$.
Consider the random dynamical system described by the skew-product $F: {\mathbb S}^1 \times [0,1] \to {\mathbb S}^1 \times [0,1] $
$$
F(\om,x)=\Big(\,\theta(\omega), T_{\omega}(x-\gamma(\om))+\gamma(\om)\, \text{mod}\, 1\Big),
$$
where $T_\om$ is the map 
$$
T_\om(x) 
     = 
    \begin{cases}
      x(1+(2x)^{\beta(\om)}), \quad \text{if}\, x\in [0,\frac12) \\
    2x-1, \qquad\qquad  \text{otherwise}
    \end{cases}
$$
where $\beta : \mathbb S^1 \to (0,1)$. 
This corresponds to random iteration of maps with indifferent fixed points which may have different contact orders 
(determined by $\beta$) and may be non-localized (determined by $\gamma$).
Note that $p_\om=q_\om=1$, $L_\om=1$ and $\sigma_\om=2$ for every $\om$.
It is not hard to check that hypothesis %
(H0)-(H6) are satisfied. 
In addition, the potential  $\phi=(\phi_\om)_\om \in L^1_X(\mathbb S^1,C^1([0,1]))$ satisfies assumptions (P) whenever
\[\begin{split}
\int_{\mathbb S^1}(\sup\phi_\om-\inf\phi_\om+\log(1+\|D\phi_\om\|_\infty ))\, dLeb(\om) 
<\log \frac43
\end{split}\]
For any such $\phi$ there exists a unique equilibrium state, it is an expanding measure and has 
exponential decay of correlations. 
In particular, taking $\phi\equiv 0$, we conclude that there exists a unique and fully supported measure of maximal entropy for the random dynamical system and it has exponentially fast decay of correlations for smooth observables. 

\subsection{Random perturbations of non-uniformly expanding maps}

Fix $d\geqslant 2$ and let  $g_0:\torus^d\to\torus^d$ be a linear expanding map. Fix some
covering $\cP$ for $g_0$ and some $P_1\in\cP$ containing a fixed (or
periodic) point $p$. Then deform $g_0$ on a small neighborhood of
$p$ inside $P_1$ by a pitchfork bifurcation in such a way that $p$
becomes a saddle for the perturbed local diffeomorphism $g$. 
The perturbation can be performed in such a way that: there
exist constants $\si>1$ and $L>0$, and an open region $\cA\subset \mathbb T^d$
such that 
$L(x)\leq L$ for every $x \in \cA$ and
$L(x)\leq \sigma^{-1}<1$ for all $x\in M\backslash \cA$, and $L$ is
close to $1$ (to be determined), and there exists $k_0 \geq
1$ and a covering $\cP=\{P_1,\dots, P_{k_0}\}$ of $\mathbb T^d$ by domains of
injectivity for $g$ such that $\cA$ can be covered by $q<\deg(g)$
elements of $\cP$. In particular, if $L$ is chosen sufficiently close to $1$ then 
$g$ has a unique measure of maximal entropy, it has non-uniform expansion, has
exponential decay of correlations and varies differentiably with respect to the dynamics (cf.~\cite{BCV,CV13,VV10}).

Let us consider random perturbations of the non-uniformly expanding map $g$.
If the perturbation is suitably chosen (see e.g. \cite[Section 4]{AMO})
it can produce a $C^1$-open set $\cU$ of $C^2$-local diffeomorphisms satisfying 
the assumptions in \cite{AMO}.  
Thus, there exists $\vep>0$ such that the random dynamical system $f=(f_\om)_\om$ generated by maps
$f_\om\in \cU$ has \emph{at least} one equilibrium state for every $\phi\in L^1(\Om,C^1(\mathbb T^d))$
such that $\int \sup\phi_\om \,d\mathbb P - \int \inf\phi_\om \,d\mathbb P$ is small (cf. \cite[Theorem~A]{AMO} for the precise statement). Since the random variables $\sigma_\om$, $L_\om$, $p_\om$, $q_\om$ can be taken as constants
for each $f_\om\in \cU$, hypothesis (H0)-(H6) are satisfied whenever $L$ satisfies
$$
{\sigma^{-1}p+L q} < {\deg(g)}
\quad\text{and}\quad
\frac{\sigma^{-1}p+L q }{\deg g} \, < \frac1{L^d}.
$$
Theorems~\ref{thm:Ruelle} to ~\ref{thm:HP} imply that every $\phi \in L^1(\Om,C^1(\mathbb T^d))$ at high temperature
has a unique equilibrium state, it is an expanding and weak Gibbs probability, and it has exponential decay of correlations.

\subsection{Random maps which expand weakly in average}
We consider now a random dynamical system modeled by the skew-product $F(\omega,x)=(\theta(\omega), f_{\omega_0}(x))$ where $\theta: \{0,1\}^{\mathbb Z} \to 
\{0,1\}^{\mathbb Z}$ is the shift, $f_0: [0,1] \to [0,1]$
is a Manneville-Pomeau map and 
$f_1: [0,1] \to [0,1]$ be a $C^2$-smooth and strictly convex map so that $f_1(0)=0$, $f_1(1)=1$, $L^{-1}=f'(0)\in(0,1)$ and  $f'(1)>1$. 
(smooth dynamics with identical combinatorics can be easily constructed in $\mathbb {\mathbb S}^1$). 
The point $0$ is an attractor whose basin coincides with the interval $[0,1)$.

Let $\mathbb P=\{a,1-a\}^{\mathbb N}$
the the Bernoulli measure on $\{0,1\}^{\mathbb N}$. If $L>1$ is sufficiently close to 1 
then (H3) holds: if $\cP_\om^0=\{[0,\frac12), [\frac12, 1]\}$ for all $\omega=0\omega_1 \omega_2\dots$ 
and $\cP_\om^1=\{[0,1]\}$ for all $\omega=1\omega_1 \omega_2\dots$ then 
$$
\sigma_\om^{-1}
    = 
    \frac12, \quad \text{if}\, \omega_0=0 \\
\qquad \text{and}\qquad
L_\om
    = \begin{cases}
    1, \quad \text{if}\, \omega_0=0 \\
    L, \quad \text{if}\, \omega_0=1
    \end{cases}
$$
and  
$$
p_\om
    = \begin{cases}
    1, \quad \text{if}\, \omega_0=0 \\
    0, \quad \text{if}\, \omega_0=1
    \end{cases}
\qquad \text{and}\qquad
q_\om=1 , \quad\text{for all}\, \om.
$$
Moreover, since
$$
\int_\Om \log\Bigl(\frac{\sigma_\om^{-1}p_\om+L_\om q_\om}{\deg(f_\om)}\Bigr)d\mP
    = a \log \frac34 +(1-a) \log L = \log \Big[  \Big(\frac{3}4\Big)^a L^{1-a} \Big]
$$
hypothesis (H3) holds if and only if $L< \Big(\frac{4}{3}\Big)^{\frac{a}{1-a}}$.
Hypothesis (H6) is satisfied whenever
$$
a \log \frac34 +(1-a) 2 \log L < 0.
$$
The remaining assumptions  (H0), (H1), (H2), (H4) and. (H5') are clear. %
Under these assumptions, any potential $\phi\in L^1(\Om,C^1([0,1]))$ at high temperature has a unique equilibrium state.
In particular, there exists a unique measure of maximal entropy for the random dynamical 
systems, it is expanding, weak Gibbs and has exponential decay of correlations.

\subsection{Multidimensional non-uniformly expanding maps with unbounded degree}

Given an integer $k\geqslant 1$, let $\tilde f_k:\mathbb T^d \to \mathbb T^d$ ($d\geqslant 2$) be a $C^1$ expanding map with
$\deg(\tilde f_k)=k$. Let $f_k$ be a $C^1$-map obtained from a perturbation of $\tilde f_k$ by isotopy in $1\leqslant \ell_k < k$
domains of injectivity (let $\cA_k$ denote the union of those domains of injectivity) in such a way that the following hold:
\begin{enumerate}
\item[(a)] $\|Df_k(x)^{-1}\|^{-1} \geqslant L_k^{-1}>0$ for every $x\in \cA_k$ and $k\geqslant 1$
\item[(b)] $\|Df_k(x)^{-1}\|^{-1} \geqslant \sigma_k >1$ for every $x\notin \cA_k$ and $k\geqslant 1$
\item[(c)] $\lim_{k\to\infty} \frac{\ell_k}k=1$
\end{enumerate}
\color{black}
It is allowed that some of the maps $f_k$ all of its inverse branches have non-contracting behavior (e.g. 
if every domain of invertibility has an indifferent fixed point).  
Conditions (a) and (b) imply that each map $f_k$ has coexisting contracting and expanding behavior.
Condition (c) means that the combinatorial proportion of the number of regions with contraction for the individual 
dynamics $f_k$ tend to
occupy a larger part of the phase space as $k$ increases.

Assume that $\underline a=(a_k)_k \in \ell^1(\mathbb N)$ is a probability vector i.e. $\sum_{k\geqslant 1} a_k=1$.
Let $\mathbb P_{\underline a}$ denote the Bernoulli probability measure on $\Omega:=\mathbb N^{\mathbb Z}$
induced by $\underline a$, and consider the random dynamical 
system modeled by the skew-product $F(\omega,x)=(\theta(\omega), f_{\omega}(x))$ where $\theta: \Omega \to \Omega$ 
is the shift and $f_\omega=f_{\omega_0}$  for every $\omega\in \Om$.
It is easy to check that the assumptions of Theorem~\ref{thm:Ruelle} hold provided that 
\begin{align}\label{eq:integ}
 \sum_{k\geqslant 1} a_k \log \sigma_k <\infty  \qquad \quad & \sum_{k\geqslant 1} a_k \log L_k <\infty \quad  \text{and}&  
 \sum_{k\geqslant 1} a_k \log k <\infty 
\end{align}
and
\begin{align}\label{eq:cominft}
     \sum_{k\geqslant 1} a_k \log\Bigl( (1-\frac{\ell_k}k) \sigma_k^{-1}+\frac{\ell_k}k L_k \Bigr) <0.
\end{align}
The latter integrability conditions hold if  $(a_k)_k$ tends sufficiently fast to zero as $k\to\infty$, and 
relation ~\eqref{eq:cominft} is true for instance whenever 
$$(1-\frac{\ell_k}k) \sigma_k^{-1}+\frac{\ell_k}k L_k < 1$$ 
for all $k\geqslant 1$, meaning that every map is combinatorially expanding. 
Under these conditions, Theorem~\ref{thm:Ruelle} and Corollary~\ref{cor:decay} imply that for
every potential $\phi=(\phi_\om)_{\om}$ satisfying (P)
there exists a random variable $\lambda_\om \geqslant 0$, a conformal measures $\nu=(\nu_{\omega})_\om$
and densities $h=(h_\om)_\om \in L^1_X(\Om, C^1(M))$ such that 
$\cL_\om h_\om=\lambda_\om h_{\te\om} \text{\ and\ } \cL_{\om}^{*}\nu_{\te\om}=\lambda_\omega\nu_{\om}$
for $\mP$-almost every $\om$, and the $F$-invariant probability $\mu=(\mu_\om)_\om$ given by $\mu_\om=h_\om\nu_\om$ has exponential decay 
of correlations for $C^1$-observables.

\subsection{Intermittent maps with unbounded degree}

Given an integer $k\geqslant 1$, let $1\leqslant \ell_k < k$ and let $f_k:\mathbb S^1 \to \mathbb S^1$ ($d\geqslant 2$) be a $C^1$-local diffeomorphism
with $\deg(f_k)=k$ and satisfying
$$
|(f_k)'(x)|>1 \; \text{ for every} \; x\in \mathbb S^1 \setminus\{p^{(k)}_1, p^{(k)}_2, \dots, p^{(k)}_{\ell_k}\}
$$
and 
$$
f_k(p^{(k)}_i) = p^{(k)}_i \quad \text{and}\quad |(f_k)'(p^{(k)}_i)|=1 \; \text{ for every} \; 1\leqslant i \leqslant \ell_k.
$$
This is a special case of the class of maps presented in the previous example, when $d=1$ and $L_k=1$ 
for every $k\geqslant 1$.
Let $\underline a=(a_k)_k \in \ell^1(\mathbb N)$ be a probability vector such that \eqref{eq:integ} and \eqref{eq:cominft}
hold.
If, in addition
\begin{equation}
     \sum_{k\geqslant 1} a_k \log\Bigl( (1-\frac{\ell_k}k) \sigma_k^{-1}+\frac{\ell_k}k L_k \Bigr) < - 
      \sum_{k\geqslant 1} a_k \log L_k, 
\end{equation}
then assumption (H6) is satisfied. Our main results ensure that the random dynamical system has a unique 
measure of maximal entropy and it has exponential decay of correlations.

\subsection{Random non-uniformly expanding fractals}

In this final example we illustrate that the equilibrium states may be supported on random fractal sets with a rich 
geometric structure: these are not conformal nor uniformly hyperbolic and contain continua of points with indifferent 
behavior. 
  
Fix $D\geqslant 2$. For each $k\geqslant 1$, let $f_k : \mathbb T^2 \to \mathbb T^2$ obtained through a Hopf 
bifurcation of an expanding map in $\mathbb T^2$ with degree at most $D$, having a periodic point $P_k$ 
with two complex conjugate eigenvalues $\ti\si e^{i\varpi}$, with
$\ti\si>3$ and $m\varpi \not\in 2\pi \Z$ for every $1\leqslant m\leqslant 4$,
in such a way that:
\begin{enumerate}
\item $f_k$ exhibits a (non-generating) Markov partition;
\item $P_k$ becomes a periodic attractor for $f_k$ and whose basin of attraction is 
an $f_k$-invariant circle whose radial derivative is one, contained in one element $Q_k\subset \mathbb T^2$ of the Markov partition; 
\item $f_k$ has $p_k=\deg(f_k)-1\geqslant 1$ domains of the Markov partition on which $f_k$ is uniformly expanding. 
\end{enumerate}
We refer the reader e.g. to \cite{HV05} for details on such Hopf bifurcations.
Assume that the family $(f_k)_k$ is chosen in such a way that $\sup_{k\geqslant 1} \|Df_k^{-1}\|_\infty^{-1}<\infty$.
Set $\sigma_k^{-1}=\max_{x\in \mathbb T^2\setminus Q_k} \|Df_k(x)^{-1}\|<1$ and 
$L_k=\max_{Q_k} \|Df_k(x)^{-1}\| >1$.

Consider the random dynamical system driven by the shift $\sigma: \mathbb N^{\mathbb Z} \to \mathbb N^{\mathbb Z}$
endowed with a Bernoulli measure $\mathbb P_{\underline a}$. Conditions (H0)-(H4) and (H5') are easily verifiable.
If additionally 
\begin{equation*}
\sum_{k\geqslant 1} a_k\, \log\Big( \frac{\deg f_k-1}{\deg f_k} \cdot \sigma_k^{-1} + \frac1{\deg f_k} \cdot L_k \Big) 
	< - 2 \cdot \sum_{k\geqslant 1} a_k\, \log L_k
\end{equation*}
then condition (H6) holds. Under these assumptions, every potential $\phi=(\phi_\om)_\om\in L^1(\Om,C^1(\mathbb T^2))$
at high temperature has a unique equilibrium state $\mu=\mu_\phi$, which has exponential decay of correlations for smooth observables. Moreover, since $\mu$ is supported in the set of points with infinitely many hyperbolic times, it gives zero weight to the random basins of the periodic attractors.  
\color{black}

\section{Preliminaries}\label{sec:Prelim}

In this section we recall some necessary definitions and prove some auxiliary results
on the thermodynamic formalism of random dynamical systems.
%

\subsection{Random thermodynamic formalism}\label{subsec:rtf}

Let $M$ be a compact metric space with distance $d$ and denote by $\cB$ the Borel $\sigma$-algebra and let $(\Om, \cF, \mP)$ be a Lebesgue space. Let $\te$ be a $\mP$-preserving transformation on $\Om$. Let $X\subset\Om\times M$ be a measurable set with respect to the product $\sigma$-algebra $\cF\times\cB$ such that the fibers $X_\om=\{x\in M : (\om,x)\in X\}$ of $X$ are compact for each $\om\in\Om$. Consider the random dynamical system $f=(f_\om)_\om$ is generated by continuous maps $f_\om:X_\om\to X_{\te\om}$ as in \eqref{e.cociclo} and the skew-product $F: X \to X$ given by
$F(\om,x)=(\theta(\omega), f_\omega(x)).$

Let ${\cP}(X)$ denote the space of probability measures $\mu$ on $X$ 
such that the marginal of $\mu$ on $\Om$ is $\mP$. 
Let $\mM (X, f)\subset \cP(X)$ be the set of $F$-invariant probabilities.
The set $\mM (X, f)$ is always non-empty (cf. Corollary 6.13 in \cite{Cra02}). Moreover, the property
that $\mP$ is the marginal on $\Om$ of a invariant measures can be characterized by the disintegration
$
d\mu(\omega,x)=d\mu_\omega(x)d\mP(\omega),
$
where $\mu_\omega$ are called the \emph{sample measures} of $\mu$ (see e.g.~\cite{Li01}).
The measure $\mu$ is \emph{ergodic} if $(F,\mu)$ is ergodic.  Furthermore, if the $\sigma$-algebra on $\Om$ is countably generated and $\mP$ is ergodic then each invariant measure can be decomposed into its ergodic components by integration.

\subsubsection*{Entropy for random transformations}
The information
of a finite partition $\xi$ in $X$ with respect to some probability measure $\nu$ is $H_{\nu}(\eta):=-\sum_{C\in \eta}\nu(C)\log \nu(C)$, with the notation that $0\log 0=0$.
Following Liu~\cite{Li01}, given a finite Borel partition of $X$ and $\mu\in \mathcal M(X,F)$ define
\begin{equation}\label{e.entropy}
h_{\mu}(
f,\xi)=\lim_{n\to +\infty}\frac{1}{n}\int
H_{\mu_\om} \Big(\bigvee_{k=0}^{n-1}(f^k_\om)^{-1}\xi \Big) d\mP(\om),
\end{equation}
where $\mu_\om$ are the sample measures of $\mu$. The \emph{entropy of $(f,\mu)$} is defined as
$$
h_{\mu}(f):=\sup_{\xi}h_{\mu}(f,\xi)
$$
where the supremum is taken over all finite Borel partitions of $X$.
It is an interesting fact that the supremum can be taken over finite partitions $\Omega\times \xi_0$, 
defined in terms of finite partitions $\xi_0$ on $M$ (cf. \cite{Li01}).
%

\subsubsection*{Topological pressure}

We recall the concept of topological pressure for a potential $\phi \in L^1_X(\Omega,C(M))$.
Let $\vep:\Om\to\mathbb R^+$ be a random variable and set the dynamic ball
\begin{equation}\label{eq.dynball}
B_\om(x,n,\vep)
	=\big\{ y\in X_\om : d( f_\om^j(x), f_\om^j(y) )<\vep (\theta^j(\om)) \text{ for all } 0\leqslant j < n \big\}.
\end{equation}
In the case $\vep$ is constant and $f_\om=f$ for all $\omega$ these coincide with the usual Bowen balls.
Given $\om$, a set $K\subset X_\om$ is \emph{$(\om,n,\vep)$-separated} if for any
$x\neq y\in K$ there exists $0\leqslant j <n$ so that $d(f_\om^j(x),f_\om^j(y))\geqslant\vep$.
If $\phi \in L^1(\Omega,C(M))$, $\vep>0$ and $n\geqslant 1$ set
$$
\pi_\phi(f)(w,n,\vep)=\sup \Big\{ \sum_{x\in K} e^{S_f(\phi)(w,n,x)} \colon  K \textrm{ is a }
(\om,n,\vep)-\textrm{separated set}\Big\},
$$
where $S_f(\phi)(w,n,x)=\sum_{k=0}^{n-1}\phi_{\theta^k(\om)}(f^k_\om(x))$.
Then, $\pi_\cdot(f):L^1_X(\Omega,C(M))\rightarrow \R\cup \{\infty\}$ defined by 
\begin{equation}\label{eq.top.pressure}
\pi_\phi(f)
	=\lim_{\vep\to 0} \limsup_{n\to +\infty} \frac{1}{n}\int_{\Om}\log\pi_\phi(f)(w,n,\vep) d\mP(w)
\end{equation}
is called the \emph{pressure function}, and the topological entropy is defined as $\pi_0(f)$.
We will also use a notion of topological pressure and a variational principle for 
(not necessarily compact) invariant sets, to be established in Appendix~A.

\begin{remark}\label{r.entropy}
If $f_\om$ is a local homeomorphism then 
$
\deg(f_\om^n)
	=\prod_{j=0}^{n} \deg(f_{\theta^j(\om)})
$
and, 
$$ \lim_{n\to\infty }\frac1n\log \deg(f_\om^n)
	= \int \log \deg(f_\om) \,d\mP$$
is a natural candidate to the topological entropy of the random dynamical system
\end{remark}

\subsubsection*{Equilibrium states}

 The variational principle for (quenched) topological pressure relates the topological pressure with fibered entropies as follows:
if $\phi \in L^1_X(\Omega,C(M))$ then 
\begin{equation}
\label{e.pressure}
\pi_\phi(f)=\sup_{\mu \in \cM(X, f)} \left\{h_{\mu}(f)+\int\phi \,d\mu\right\}
\end{equation}
(see e.g.~\cite{Li01}).
In the case that $\mP$ is ergodic we can take the supremum over the set of ergodic measures.
Any probability measure $\mu \in \cM(X,f)$ that attains the supremum~\eqref{e.pressure}
is called an \emph{equilibrium state} for $f$ with respect to $\phi$. 

It often occurs, even under mild hyperbolicity assumptions, that equilibrium states are  absolutely continuous with respect to some reference measure. By some abuse of notation, we say that a (not necessarily invariant) probability measure $\nu=(\nu_\om)_\om\in \cP( X)$ is a  \emph{conformal measure} if for $\mP$-almost every $\om$ there exists a measurable random variable $(\la_\om)_\om$ such that $\la_\om>0$ for $\mP$-a.e. $\om$ and
\begin{equation}\label{eq:conformal}
\nu_{\theta(\om)}(f_\om(A))
	= \int_A \la_\om e^{-\phi_\om} \, d\nu_\om
\end{equation}
for every measurable set $A\subset X_\om$ such that $f_\om|_A$ is injective (we say $J_\om f:=\la_\om e^{-\phi_\om}$ 
is the \emph{Jacobian of $f_\om$}). It is standard to prove that the eigenmeasures $\nu=(\nu_\om)_\om$ 
obtained in Theorem~\ref{thm:Ruelle} satisfy ~\eqref{eq:conformal}.

\subsection{Random non-uniform expansion}\label{subsec:hue}

In this subsection we introduce the concept of non-uniform expansion that we will work with. Given $c>0$,
we say that a (not necessarily invariant) probability measure $\nu \in \cP(X)$ is \emph{$c$-non-uniformly expanding}
 if 
\begin{equation}\label{eq:RNUE}
\limsup\limits_{n\rightarrow +\infty}
		\frac{1}{n}\sum\limits_{j=0}^{n-1} \log \|Df_{\theta^j(\om)} (f^j_\om (x) )^{-1}\|
		\leqslant -2c <0
\end{equation}
for $\nu$-almost every $(\om,x)\in X$.
We say that $n$ is a \textit{$c$-hyperbolic time} for $(\omega,x)$ if
\begin{equation}\label{eq:RNUE-times}
\prod _{j=n-k}^{n-1}  \|Df_{\theta^{j}(\om)} (f^j_\om(x))^{-1}\| \leqslant e^{-ck}.
\end{equation}
for every  $1\leqslant k\leqslant n$.

\begin{remark}
Khanin and Kifer \cite{KhKi96} considered random dynamical systems satisfying the \emph{average expansion condition}
\begin{equation}\label{eq:KK}
\int \log\| Df(\cdot)^{-1}\|_\infty \, d\mathbb P <0.
\end{equation} 
The latter still requires the existence of $\tilde\Omega \subset \Omega$ with $\mathbb P(\tilde\Omega)>0$ so that $f_\omega$ is an expanding map, for every $\omega\in\tilde\Omega$. 
Note that ~\eqref{eq:KK} implies that {\bf every} $F$-invariant probability measure $\mu$ so that $\pi_*\mu=\mathbb P$ is non-uniformly expanding in the sense of \eqref{eq:RNUE} (a similar concept was used by Mayer, Skorulski and Urba{\'n}ski \cite{MSU08}, who considered average distance expanding maps).
In \cite{KhKi96} the absence of invariant measures with non-positive Lyapunov exponents is used crucially 
to ensure the
uniqueness of equilibrium states for H\"older continuous potentials.
\end{remark}

\section{Birkhoff cones and random Ruelle-Perron-Frobenius operator}\label{sec:BKM}

\subsection{Random Ruelle-Perron-Frobenius operator}
Denote by $g=(g_\om)_\om$ the measurable observables in $L^1(\Om,C(X,\mathbb R))$.  %
Given $\om \in \Om$ consider the random Ruelle-Perron-Frobenius operator  
$\cL_\om: C(X_\om) \to C(X_{\theta(\om)})$ is given by
$$
\cL_\om g_\om(x)=\sum_{f_\om(y)=x} e^{\phi_\om(y)} g_\om(y).
$$
Notice that %
the map $\om \mapsto \cL_\om g_\om$ is a measurable function. We denote by $\cL^*_\om: \mathcal M(X_{\theta(\omega)})\to \mathcal M(X_\om)$
the dual of the operator $\cL_\om$ acting on the space of measures in $X_{\theta(\omega)}$.
Our purpose in this section is to prove the following quenched version of the Ruelle-Perron-Frobenius theorem:

\begin{theorem}\label{construction.triple}
Let $f=(f_\om)_{\om\in\Om}$ be a family of local diffeomorphisms satisfying (H0)-(H4) and let $\phi=(\phi_\om)_{\om\in\Om}$ be a potential in $L^1_X(\Om, C^1(M))$ satisfying (P). Then there exists a triple $(\lambda, h, \nu)$ consisting of a 
measurable positive random variable $\lambda=(\lambda_\om)_\om$, of a positive measurable function $h=(h_\om)_\om \in L^1_X(\Om, C^1(M))$ and of a measurable family of probability measures $\nu=(\nu_{\omega})\in \cP(X)$ such that 
\[\cL_\om h_\om=\lambda_\om h_{\te\om} \text{\ and\ } \cL_{\om}^{*}\nu_{\te\om}=\lambda_\omega\nu_{\om}.\]
Moreover, $\mu_\om=h_\om\nu_\om$ defines an $F$-invariant probability measure $\mu$ in $\cM(X, f)$.
\end{theorem}

The proof of this theorem can be divided in three parts. 
Firstly, we construct a positive random variable $\tilde{\lambda}_\om$ and a positive measurable function $\tilde{h}_\om$ satisfying $\cL_\om \tilde{h}_\om=\tilde{\la}_\om \tilde{h}_{\te\om}$ (see Proposition \ref{construction.eigenf.}). This is shown using Birkhoff cone methods for random dynamical systems. This step, inspired by \cite{CV13,Ki08}, requires the construction of suitable Banach spaces
and a kind of Lasota-Yorke inequality for the transfer operators. Indeed, the key point is to introduce an appropriate fiberwise positive cone $\La_\om$ so that $\cL_\om\La_\om\subset\La_{\te\om}$ and to show that the diameter of $\cL_{\te^{-n}\om}^n\La_{\te^{-n}\om}$ with respect to the projective metric for $\La_\om$ decreases exponentially fast, although not uniformly in $\om$, as $n\to\infty$. Secondly, we 
consider a suitable normalized operator $\tilde{\cL_\om}$ and
prove that the sequence $\{\tilde{\cL}_{\te^{-n}\om}^n 1\}_{n=0}^{\infty}$ is a Cauchy sequence in $C(X_\om)$. 
Here $\tilde{\la}_\om$ appears as a positive number related to the normalized constant of the RPF operator.
The third part, which completes the proof of Theorem \ref{construction.triple}, makes use of the so-called Krylov-Bogoliubov procedure for random dynamical systems (see Section 5 in \cite{Ki08}). The key fact is that the measurable triple $(\lambda,h, \nu)$
is obtained using disintegrations of a Markov-Feller operator for cohomologically equivalent potentials determined 
by $\tilde \la_\om$ and $\tilde h_\om$.
%

\subsection{Invariant cones and normalized densities}
For
a given positive random variable $a(\om)$, consider the family of cones of positive continuous functions on $X_\om$ defined by
\begin{equation}\label{eq.cones}
\La_{a(\om)}
	=\big\{g_\om\in C^1(X_\om) : g_\om>0 \text{ and } \|Dg_\om\|_\infty\leqslant a(\om) \inf g_\om \big\},
\end{equation}
where $\|\cdot\|_\infty$ denotes the supremum norm.
These cones are clearly nonempty since they contain all constant functions.
While this particular family of cones is inspired by the work of \cite{CV13}, 
the arguments in there do not apply directly here since our assumptions (namely a fibered average expansion) 
are much weaker. %

We will make use of some auxiliary results, the first of which is the celebrated Birkhoff's contraction theorem for projective maps.
First we recall some concepts.
If $\cB$ is a Banach space, a subset
$\Lambda\subset \cB-\{0\}$ is called a \emph{cone} if $r\cdot v\in\Lambda$ for all
$v\in\Lambda$ and $r\in\real^+$. The cone $\Lambda$ is \emph{closed} if
$\overline\Lambda=\Lambda\cup\{0\}$, and $\Lambda$ is \emph{convex} if $v+w\in\Lambda$ for all
$v,w\in\Lambda$.
Any convex cone $\Lambda$ satisfying $\Lambda\cap (-\Lambda) = \emptyset$ is
\emph{partially ordered} by a relation $\preceq$ on $\cB$ defined by
$
w\preceq v \; \text{ iff } \; v-w\in\Lambda\cup\{0\}.
$
Given a closed and convex cone $\Lambda$ so that
$\Lambda\cap (-\Lambda)=\emptyset$ and two vectors $v,w\in\Lambda$, we define
\begin{equation*}
\Theta (v,w)=\Theta_{\Lambda}(v,w):=\log\frac{B_{\Lambda}(v,w)}{A_{\Lambda}(v,w)},
\end{equation*}
where $A_{\Lambda}(v,w)=\sup\{r\in\real^+: \; r\cdot v\preceq w\}$ and
$B_{\Lambda}(v,w)=\inf\{r\in\real^+: \; w\preceq r\cdot v\}$. The
(pseudo-)metric $\Theta$ is called the \emph{projective metric} of $\Lambda$. 
The relation $v\sim w$ defined by $w=r\cdot v$ for some $r\in\real^+$
is an equivalence relation and $\Theta$ induces a metric 
 and the quotient (or projective) space $\Lambda / \sim$.

\begin{lemma}\label{contraction of metrics}
Let $\Lambda^{(i)}$ be a closed convex cone in a Banach space $\cB_{i}$ for $i=1,2$. 
If $\cL: \cB_1\to\cB_2$ is a linear operator so that  $\cL(\Lambda^{(1)})\subset\Lambda^{(2)}$ and 
$\Delta:=\text{diam}_{\Lambda^{(2)}}(\cL\Lambda^{(1)})<\infty$, then 
\[\Theta_{\Lambda^{(2)}}(\cL v, \cL w)\leqslant (1-e^{-\Delta})\Theta_{\Lambda^{(1)}}(v, w)
\quad \text{for every $v, w \in \Lambda^{(1)}$.} \]
\end{lemma}
\begin{proof}
See e.g. \cite{Liv}.
\end{proof}

The next lemma shows that  the diameter of $\Lambda_{\lambda\kappa}$ with respect to the projective metric for $\Lambda_{\kappa}$ is finite. %

\begin{lemma}\label{evaluation of diam.}
Let $\lambda\in (0,1)$ and $\kappa>0$. Then we have
\[\Theta_{\Lambda_{\kappa}}(\varphi, \psi)\leqslant 2\log\frac{1+\lambda}{1-\lambda}+2\log(1+\lambda\kappa \text{diam}(M))\]
for $\varphi, \psi\in\Lambda_{\la\kappa}$. In particular, 
$\text{diam}_{\Lambda_\kappa}(\Lambda_{\la\kappa})<\infty$
\end{lemma}
\begin{proof}
Analogous to the proof of Proposition 4.3 in \cite{CV13}.
\end{proof}

The first step in the proof of Theorem~\ref{thm:Ruelle} is to construct families of eigenfunctions and eigenvalues for the 
normalized transfer operators. For that purpose, we observe the action of the transfer operator on measurable observables
that are $C^1$-smooth along the fibers.

\begin{proposition}\label{construction.eigenf.}
Under the assumptions of Theorem \ref{construction.triple}, there exist a positive random variable $\tilde{\lambda}_\om$ and a positive measurable function $(\tilde{h}_\omega)_\om \in L^1_X(\Om,C^1(M))$
such that \[\cL_{\om}\tilde{h}_\om=\tilde{\la}_\om \tilde{h}_{\te\om}
\quad \mathbb P-\text{a.e. }\, \omega.\]
\end{proposition}

\begin{proof}
Let $a(\om)$ be an arbitrary positive random variable and $\La_{a(\om)}$ denote the corresponding positive cone given by (\ref{eq.cones}).

For each $(\theta \om,x)\in X$, let $(x_i)_{i=1}^{\deg(f_\om)}$ denote the set of preimages by $f_\omega$ of the point $x\in X_{\te\om}$ and let $f_{\omega,i}^{-1}$ denote the local inverse branch for $f_\om$ defined in an open neighborhood of $x$ and satisfying 
$f_{\omega,i}(x_i)=x$.
By the chain rule, 
\begin{align*}
D(\cL_\om\varphi)(x)=\sum_{i=1}^{\deg(f_\om)}e^{\phi_\om(x_i)}(D\varphi)(x_i)Df_{\omega,i}^{-1}(x) \\ +\sum_{i=1}^{\deg(f_\om)}e^{\phi_\om(x_i)}\varphi(x_i)D\phi_{\om}(x_i)Df_{\om,i}^{-1}(x)
\end{align*}
  
for $\varphi\in\La_{a(\om)}$. So $\|D(\cL_\om\varphi)(x)\|$ is bounded by
\[\sum_{i=1}^{\deg(f_\om)}e^{\phi_\om(x_i)}\|(D\varphi)(x_i)Df_{\om,i}^{-1}(x)\|+
\sum_{i=1}^{\deg(f_\om)}e^{\phi_\om(x_i)}|\varphi(x_i)|\|D\phi_{\om}(x_i)Df_{\om,i}^{-1}(x)\|.\]
By the condition (H1), it holds that $\|Df_{\om,i}^{-1}(\cdot)\| \leqslant \sigma_{\omega}^{-1}$ for
$1\leqslant i \leqslant p_\om$ and $\|Df_{\om,i}^{-1}(\cdot)\|\leqslant L_{\omega}$ for $p_\om < i \leqslant p_\om+q_\om$. Using this property and 
$
\sup\varphi
\leqslant \inf\varphi+\|D\varphi\|_\infty \diam(X_\om) \leqslant \inf\varphi (1+a(\omega) diam(X_\omega))
$ 
for $\varphi\in\La_{a(\om)}$, we get 
\[
\begin{split}
\frac{\|D(\cL_\om\varphi)(x)\|}{\inf\cL_\om\varphi}
&\leqslant
\frac{\sum_{i=1}^{p_\om}\sigma_\om^{-1} 
e^{\phi_\om(x_i)}\|(D\varphi)(x_i)\|+\sum_{i>p_\om}L_\om e^{\phi_\om(x_i)}\|(D\varphi)(x_i)\|}{\deg(f_\om)e^{\inf\phi_\om}\inf\varphi} \\
&+\frac{\sum_{i=1}^{p_\om}\sigma_\om^{-1} 
e^{\phi_\om(x_i)}|\varphi(x_i)|\|(D\phi_\om)(x_i)\|}{\deg(f_\om)e^{\inf\phi_\om}\inf\varphi} \\
&+ \frac{\sum_{i>p_\om}L_\om e^{\phi_\om(x_i)}
|\varphi(x_i)|\|(D\phi_\om)(x_i)\|}{\deg(f_\om)e^{\inf\phi_\om}\inf\varphi} \\
&\leqslant e^{\sup\phi_\om-\inf {\phi_\om}}\Bigl(\frac{\sigma_\om^{-1} p_\omega 
+L_\om q_\om}{\deg(f_\om)}\Bigr)\frac{\|D\varphi\|_\infty}{\inf\varphi} \\
&+ e^{\sup\phi_\om-\inf {\phi_\om}}\Bigl(\frac{\sigma_\om^{-1}p_\omega %
+L_\om q_\om}{\deg(f_\om)}\Bigr)\|D\phi_\om\|_\infty\frac{\sup\varphi}{\inf\varphi} \\
&\leqslant e^{\sup\phi_\om-\inf {\phi_\om}}\Bigl(\frac{\sigma_\om^{-1} p_\omega %
+L_\om q_\om}{\deg(f_\om)}\Bigr)(1+\|D\phi_\omega\|_\infty\diam(X_\om))a(\omega) \\
&+e^{\sup\phi_\om-\inf {\phi_\om} }\Bigl(\frac{\sigma_\om^{-1} p_\omega %
+L_\om q_\om}{\deg(f_\om)}\Bigr)\|D\phi_\om\|_\infty
\end{split}
\]
for every $(\om,x)\in X$.
For notational simplicity, we set 
\[\alpha(\om)=e^{\sup\phi_\om-\inf {\phi_\om}}\Bigl(\frac{\sigma_\om^{-1} p_\omega %
+L_\om q_\om}{\deg(f_\om)}\Bigr)(1+\|D\phi_\omega\|_\infty\diam(X_\om))\]
and
\[\beta(\om)=e^{\sup\phi_\om-\inf {\phi_\om}}\Bigl(\frac{\sigma_\om^{-1}p_\omega %
+L_\om q_\om}{\deg(f_\om)}\Bigr)\|D\phi_\om\|_\infty.
\]

The previous estimates ensure that $\|D(\cL_\om\varphi)\|_\infty/\inf\cL_\om\varphi\leqslant \alpha(\omega)a(\omega)+\beta(\omega)$ for 
every $\varphi\in\La_{a(\om)}$.
Now, the key idea is to construct a function $\ka(\omega)$ which satisfies
$\cL_\om\La_{\ka(\om)}\subset\La_{\ka(\te\om)}$ for $\mP$-a.e. $\om\in\Om$.
Set 
\begin{equation}\label{eq:defkappa}
\ka(\om)=1+\beta(\te^{-1}\om)+\sum_{i=1}^{\infty}(1+\beta(\te^{-(i+1)}\om))\prod_{k=1}^i \alpha(\te^{-i}\om).
\end{equation}
We remark that the condition (P) and the ergodicity of the base dynamical system $(\te, \mP)$ yields the finiteness of the function $\ka(\om)$ for $\mP$-a.e. $\om\in\Om$.
In fact, by the ergodicity of $(\te,\mP)$ and the integrability requirements involving (P), we have $\frac1n \log\beta(\te^{-n}\om)\to0$ as $n\to\infty$ for $\mP$-a.e. $\om\in\Om$. This means that for every positive number $\varepsilon>0$ there is a positive integer $n_1(\omega)$ such that 
\[\beta(\te^{-n}\om)\leqslant e^{\varepsilon n}\]
 for $n\geqslant n_1(\om)$. On the other hand, by the ergodicity of $(\te, \mP)$, %
 there exists a positive integer $n_2(\om)$ such that 
\[
\prod_{i=1}^{n}\alpha(\te^{-i}\om)<\exp\Bigl(\int_\Om\log\alpha(\om) \, d\mathbb P+\varepsilon\Bigr)n
\]
for all $n\geqslant n_2(\om)$. By taking $\varepsilon>0$ sufficiently small so that $\int_\Om\log\alpha(\om)d\mP+3\varepsilon<0$, 
guaranteed by assumption (P), we get that $\ka(\om)<\infty$ for $\mP$-a.e. $\om\in\Om$.
It is not hard to check the following crucial property of the function $\ka(\om)$ defined by \eqref{eq:defkappa}: it is a solution of the twisted %
cohomological equation
\begin{equation}\label{eq:twistedcohm}
\ka(\te\om)=\alpha(\om)\ka(\om)+\beta(\om)+1.
\end{equation}
Hence, taking $a(\om)=\ka(\om)$ we get 
\[\begin{split}
\frac{\|D(\cL_\om\varphi)\|_\infty}{\inf\cL_\om\varphi}
\leqslant \alpha(\om)\ka(\om)+\beta(\om) 
=\ka(\te\om)-1 
<\ka(\te\om)
\end{split}\]
for $\varphi\in\La_{\ka(\om)}$, which shows that $\cL_\om\La_{\ka(\om)}\subset\La_{\ka(\te\om)-1}\subset\La_{\ka(\te\om)}$ for $\mP$-a.e. $\om\in\Om$.
A recursive argument ensures that 
\begin{equation}\label{eq:inclus}
\cL^n_{\te^{-n}\om}\La_{\ka(\te^{-n}\om)}\subset\La_{\ka(\om)}
\quad \text{for $\mP$-a.e. $\om\in\Om$.}
\end{equation}
We need the following estimate on the projective contraction for the composition of transfer operators.

\begin{claim}
There exists $C_0>0$ and $\gamma \in(0,1)$ so that 
for $\mP$-a.e. $\om$ it holds
\begin{equation}\label{decay}
\Te_{\La_{\ka(\om)}}(\cL_{\te^{-(n+m)}\om}^{n+m}\varphi, \cL_{\te^{-n}\om}^{n}\psi)
	\leqslant C_0 \, \gamma^n \; \Te_{\La_{\ka(\te^{-n}\om)}}(\cL_{\te^{-(m+n)}\om}^{m}\varphi, \psi)
\end{equation}
for every $m\geqslant1$, every large $n\geqslant1$ (depending on $\om$) 
and every $\varphi\in\La_{\ka(\te^{-(n+m)})}$, $\psi\in\La_{\ka(\te^{-n}\om)}$.
\end{claim}

\begin{proof}[Proof of Claim~1]
By applying Lemma~\ref{contraction of metrics}, %
for positive integers $n,m\geqslant1$ and $\varphi\in\La_{\ka(\te^{-(n+m)})}$, $\psi\in\La_{\ka(\te^{-n}\om)}$, we have
\begin{align}\nonumber 
&\Te_{\La_{\ka(\om)}}(\cL_{\te^{-(n+m)}\om}^{n+m}\varphi, \cL_{\te^{-n}\om}^{n}\psi) \\
\label{eq:wcont} & \leqslant \Big[ \prod_{i=0}^l (1-e^{-\Delta_{i}(\om)}) \Big]\; 
\Te_{\La_{\ka(\te^{-(n+m)+l}\om)}}(\cL_{\te^{-(n+m)}\om}^{n+m-l}\varphi, \cL_{\te^{-n}\om}^{n-l}\psi)
\end{align}
for $0\leqslant l \leqslant n$, where $\Te_\La$ denotes the projective metric with respect to the positive cone $\La$ and 
\[
\Delta_{i}(\om):= \Delta_{i}(n,m,\om)=\diam_{\substack{\Te_{\La_{\ka(\te^{-(n+m)+i+1}\om)}}}}
\Bigl(\cL_{\te^{-(n+m)+i}\om}^{n+m-i}\La_{\ka(\te^{-(n+m)+i})}\Bigr).
\]
As  $1-e^{-\De_i(\om)}\leqslant 1$, the expression in \eqref{eq:wcont} is always a weak contraction.
We shall estimate an upper bound for the diameters $\Delta_{i}(\om)$, along a positive density sequence of values of $i$.
Set $\eta(\omega):=1-1/\ka(\omega)$, and
note that $\cL_\om\La_{\ka(\om)}\subset\La_{\ka(\te\om)-1}=\La_{\eta(\te\om)\ka(\te\om)}
\subset\La_{\ka(\te\om)}$. %
Lemma \ref{evaluation of diam.} implies that,
\begin{equation}\label{diameter.}
\begin{split}
\diam_{\La_{\ka(\te^{-1}\om)}} & (\cL_{\te^{-1}\om}\La_{\ka(\te^{-1}\om)}) 
 \leqslant \diam_{\La_{\ka(\om)}}(\La_{\eta(\om)\ka(\om)}) \\
&\leqslant 2\log\frac{1+\eta(\om)}{1-\eta(\om)}+2\log(1+\eta(\om)\ka(\om) \text{diam}(M)).
\end{split}
\end{equation}
For a positive integer $N\geqslant 1$, set $A_N=\{\om\in\Om \colon \ka(\om)\leqslant N\}$. Since $\ka(\om)<\infty$ for $\mP$-a.e. $\om\in\Om$, there is a positive integer $N_0$ such that
$\mP(A_{N_0})>0$. Hence, on the one hand, by the ergodic theorem, for $\mP$-a.e. $\om$ there exists an integer $n_3(\om) \geqslant 1$ so that 
\begin{equation}
\frac{1}{n} \# \{1\leqslant i \leqslant n \colon \te^{-i}\om \in A_{N_0}\}  \geqslant\frac{\mP(A_{N_0})}{2}
\end{equation}
for every $n\geqslant n_3(\om)$. %
On the other hand, since %
$$
C_{N_0}:=\max_{\om\in A_{N_0}} \{1-\eta(\om), 1+\eta(\om), \eta(\om)\ka(\om)\}<\infty,
$$
inequality (\ref{diameter.}) ensures that
the $\Te_{\La_{\ka(\om)}}$-diameter of $\cL_{\te^{-1}\om}(\La_{\ka(\te^{-1}\om)})$ is bounded by certain constant $D_{N_0}>0$ 
for every $\om\in A_{N_0}$.
Altogether we conclude that for $\mathbb P$-a.e. $\om$ %
\begin{align*}
\Te_{\La_{\ka(\om)}} & (\cL_{\te^{-(n+m)}\om}^{n+m}\varphi, \cL_{\te^{-n}\om}^{n}\psi) \\
	& \leqslant(1-e^{-D_{N_0}})^{\#\{m\leqslant i \leqslant m+n \colon \te^{-(n+m)+i}\om\in A_{N_0}\}-1}
		 \Te_{\La_{\ka(\te^{-m}\om)}}(\cL_{\te^{-m}\om}^{m}\varphi, \psi) 
\end{align*}
for every large $n\geqslant1$ (depending on $\om$).
This proves the claim, taking $\gamma=(1-e^{-D_{N_0}})^{\frac{\mP(A_{N_0})}{2}} \in (0,1)$ and $C_0=(1-e^{-D_{N_0}})^{-1}$.
\end{proof}

Returning to the proof of the proposition,  let $m\in \cP(X)$ be a probability measure in $X$, let $m=(m_\om)_\om$ 
be its disintegration along the measurable partition $(X_\om)_\om$ and consider the projective cone
\[\La_{\ka(\om)}(m)=\Bigl\{\varphi\in\La_{\ka(\om)} \colon \int_{X_\omega}\varphi \,dm_{\om}=1\Bigr\}.\]
For each $\varphi\in\La_{\ka(\om)}(m)$, the condition $\int_{X_\omega}\varphi \, dm_{\om}=1$ 
implies $\inf\varphi \leqslant 1 \leqslant \sup\varphi$. Together with the inequality $\sup\varphi\leqslant (1+\ka(\om)\diam(X_\om))\inf\varphi$, we get 
\begin{equation}\label{eq:ul-bounds}
R(\om)^{-1}\leqslant\inf\phi\leqslant 1 \leqslant \sup\phi\leqslant R(\om),
\end{equation}
where $R(\om)=1+\ka(\om)\diam(X_\om)>0$. 

The remaining of the proof makes use of the uniform convergence in the cone of positive observables. 
Let $\La^{+}_\om$ denote the cone of strictly positive continuous functions on $X_\om$
endowed with the projective metric
\[\Te_{\La^{+}_\om}(\varphi, \psi)=\log\Bigl(\frac{\sup\frac{\varphi}{\psi}}{\inf\frac{\varphi}{\psi}}\Bigr).\] 
Since $\La_{\ka(\om)}\subset \La^{+}_\om$ then the projective metrics satisfy $\Te_{\La^{+}_\om}(\varphi, \psi)\leqslant\Te_{\La_{\ka(\om)}}(\varphi, \psi)$ for any $\varphi, \psi\in\La_{\ka(\om)}$.
For each $\om$ and positive integer $n\geqslant 1$, we define the normalized Ruelle-Perron-Frobenius operator $\tilde \cL^n_{\te^{-n}\om}: \La_{\ka(\te^{-n}\om)} (m) \to\La_{\ka(\om)}(m)$
by
\begin{equation}\label{aa}
\tilde\cL^{n}_{\te^{-n}\om} \varphi
	=\frac{\cL^{n}_{\te^{-n}\om} \varphi}{\int \cL^{n}_{\te^{-n}\om} \varphi \, dm_{\om} }
\end{equation}
for $\varphi\in\La_{\ka(\te^{-n}\om)}(m)$. For 
notational simplicity, we write $\varphi_n=\tilde\cL^{n}_{\te^{-n}\om} \varphi$ for each $\varphi\in\La_{\ka(\te^{-n}\om)}(m)$, $n\geqslant0$.
Since the latter are projective metrics, (\ref{decay}) ensures that for any $\varphi\in\La_{\ka(\te^{-(n+m)}\om)}$ and $\psi\in\La_{\ka(\te^{-m}\om)}$ there exists $C_{\varphi,\psi}(\om)>0$ such that 
\begin{equation}\label{decay.normalized}
\Te_{\La^{+}_\om}(\varphi_{n+m}, \psi_n)\leqslant\Te_{\La_{\ka(\om)}}(\varphi_{n+m}, \psi_n)\leqslant C_{\varphi,\psi}(\om) \gamma^n.
\end{equation}
If, in addition, $\varphi_{n+m}, \psi_n \in \La_{\ka(\om)}(m)$ then $\inf\frac{\varphi_{n+m}}{\psi_n }\leqslant 1 \leqslant \sup\frac{\varphi_{n+m}}{\psi_n }$. This, together with (\ref{decay.normalized}), implies that 
\[
e^{-C_{\varphi,\psi}(\om)\gamma^n}\leqslant\inf\frac{\varphi_{n+m}}{\psi_n }\leqslant 1 \leqslant \sup\frac{\varphi_{n+m}}{\psi_n }\leqslant e^{C_{\varphi,\psi}(\om)\gamma^n}.
\]
To construct the desired function $\tilde{h}_\om$, take $\varphi=\psi=1$.
Indeed, using $1\in \La_{\ka(\te^{-n}\om)}$ for all $n\geqslant 1$ and \eqref{eq:ul-bounds}, 
\[
\sup|1_{n+m}-1_{m}|\leqslant \sup\Bigl(|1_m|\Bigl |\frac{1_{m+n}}{1_m}-1\Bigr|\Bigr)
\leqslant R(\om)\Bigl|\frac{1_{m+n}}{1_m}-1\Bigr|\leqslant 3R(\om) C_{\varphi,\psi}(\om) \gamma^n
\]
for sufficiently large $n\geqslant1$. This yields that $\{1_n\}_{n=1}^{\infty}$ is a Cauchy sequence in $C^0(X_{\om})$. We set
\[\tilde{h}_{\om}(x)=\lim_{n\to\infty}1_n=\lim_{n\to\infty}\frac{\cL^{n}_{\te^{-n}\om} 1}{\int \cL^{n}_{\te^{-n}\om} 1\, dm_{\om} }.\]
The continuity of $\cL_{\om}$ implies
\[
\cL_{\om}\tilde{h}_\om=\lim_{n\to\infty}\frac {\cL^{n}_{\te^{-n}(\te\om)}1} 
{\int_{X_{\te\om}} \cL^{n}_{\te^{-n}(\te\om)} 1\, dm_{\te\om}}
\frac{\int_{X_{\te\om}} \cL^{n}_{\te^{-n}(\te\om)} 1\, dm_{\te\om}} 
{\int_{X_{\om}} \cL^{n}_{\te^{-n}\om} 1\, dm_{\om}}.
\]
Since 
\[\tilde{h}_{\te\om}=\lim_{n\to\infty}\frac {\cL^{n}_{\te^{-n}(\te\om)}1} {\int_{X_{\te\om}} \cL^{n}_{\te^{-n}(\te\om)} 1\, dm_{\om}},\]
the limit 
\[\tilde{\la}_{\om}=\frac{\int_{X_{\te\om}} \cL^{n}_{\te^{-n}(\te\om)} 1\, dm_{\te\om}} 
{\int_{X_{\om}} \cL^{n}_{\te^{-n}\om} 1\, dm_{\om}}\]
exists and $\cL_\om\tilde{h}_\om=\tilde{\la}_{\om}\tilde{h}_{\te\om}$. 
Finally observe that, taking the limit as $m\to \infty$ in \eqref{decay.normalized},
one has that
$\Te_{\La_{\ka(\om)}}(\tilde h_{\om}, 1_n)\leqslant C_{1,1}(\om) \gamma^n$
 tends to zero as $n\to\infty$. This ensures that $\tilde  h_\om$ is $C^1$-smooth %
and that $(\tilde{h}_\omega)_\om \in L^1_X(\Om,C^1(M))$.
This finishes the proof of the proposition.
\end{proof}

\subsection{Invariant densities}\label{sec:eigen}

The random elements $\tilde \la_\om$ and $\tilde h_\om$
given by the previous subsection are not the exactly ones needed since these arise from
the normalization by an arbitrary measure $m$.
To obtain such a measurable triple $(\la_\om,h_\om, \nu_\om)$ as desired
we now follow the strategy in ~\cite{Ki08}.  

\begin{proof}[Proof of Theorem \ref{construction.triple}]
Let $\tilde{h}_{\om}$, $\tilde{\la}_{\om}$ be given by Proposition \ref{construction.eigenf.} and
set 
$$
g(\om,x)=\tilde{\la}_{\om}^{-1} e^{\phi_\om(x)} \, \frac{\tilde{h}_\om(x)}{ (\tilde{h}_{\te\om}\circ f_\om)(x)}.
$$
Since the functions $\tilde{\la}_{\om}$ and $\tilde{h}_\om$ are positive and $\tilde{h}_\om$ is a $C^1$ function on $X_\om$, the function $g(\om,x)$ is well-defined and belongs to $L^1_{X}(\Om, C^{1}(M))$. %
The (quenched) Ruelle-Perron-Frobenius operator $\mL_{\log g_\om}: C^0(X_\omega)\to C^0(X_{\te\om})$ satisfies $\mL_{\log g_\om}1=1$, %
hence is a Markov-Feller operator. We define a skew-product map on $\cP(X)$ by the formula 
$\mL_{\log g}^{*}\rho=\tilde{\rho}$ where $d\tilde{\rho}(\om,x)=d\mP (\om)d\tilde{\rho}_\om(x)$ and $\tilde{\rho}_\om=(\cL_{\log g_\om})^{*}\rho_{\te\om}$.
This operator is also well-defined since $\mL_{\log g_\om}$ is a Markov-Feller operator. 
For $\rho\in \cP(X)$ and a positive integer $n\geqslant1$, set
\[\rho^{(n)}=\frac{1}{n}\sum_{i=1}^{n}(\cL_{\log g}^{*})^i \rho.\]  
Since the sequence $\{\rho^{(n)}\}$ is tight (see Proposition 4.3 in \cite{Cra02}), it is convergent to a probability $\rho^*\in \mathcal P(X)$ in the narrow topology (see Theorem 4.4 in \cite{Cra02}). From this construction, we can see $\mL^*_{\log g }\rho^*=\rho^*$ and
its disintegrations $\rho_\om$ satisfy $\mL^*_{\log g_\om }\rho_{\te\om}=\rho_{\om}$.
Then, taking $\mu=\rho^{*}$, $h_\om=\tilde{h}_\om a(\om)$ where $a(\om)=\int_{X_\om}(\tilde{h}_\om)^{-1}d\mu_\om>0$, $\nu_{\om}=(h_{\om})^{-1}\mu_\om$ and $\la(\om)=a(\om)\tilde{\la}_{\om}/a(\te\om)$, we get the desired triple $(\la, h, \nu)$.
In fact, by the construction of the triple $(\la, h, \nu)$, we get 
\[
\mL_\om h_\om
=\mL_\om \tilde{h}_{\om}a(\om) 
=\frac{\tilde{\la}_{\om}a(\om)}{a(\te\om)}h_{\te\om} 
=\la(\om)h_\om
\]  
and 
$\int_{X_\omega}h_\om d\nu_\om=\int_{X_{\om}}d\mu_{\om}=1.$
Furthermore, for $\varphi\in C^0(X_\omega)$, we have
\[\begin{split}
\int_{X_\om}\varphi h_\om d(\mL_{\om}^{*}\nu_{\om})
&=\int_{X_{\te\om}}\mL_{\om}(\varphi h_\om)d\nu_{\om} 
=\int_{X_{\te\om}}\la(\om) h(\te\om)\mL_{\log g_\om}\varphi d\nu_{\te\om} \\
&=\int_{X_{\te\om}}\mL_{\log g_\om}\varphi d(\la(\om)\nu_\om) 
=\int_{X_\om}(\varphi h_\om)d(\lambda_\om\nu_\om)
\end{split}\]
and
$\int_{X_\om}d\nu_\om
=\int_{X_\omega}h(\om)^{-1}d\mu_\om 
=\int_{X_\omega}\tilde{h}(\om)^{-1}d\mu_\om\cdot a(\om) 
=1.
$
This concludes the proof of the theorem.
\end{proof}

\subsection{Proof of Corollary~\ref{cor:decay}}

For $\varphi \in C^1(X_{\theta^n \om})$ and $\psi \in C^1(X_{\om})$  set
\begin{equation*}
C_{\om,\varphi,\psi} (n):=\left|\int_{X_{\om}} (\varphi\circ f_\om^n)\psi \, d\mu_{\om} - \int_{X_{\theta^n\om}} \varphi d\mu_{\theta^n\om} \, \int_{X_{\om}} \psi d\mu_{\om}
	\right|
\end{equation*}
There is no restriction in assuming $\int \psi d\mu_{\om}  = 0$. Then, using that $h_{\te^n\om}$ is bounded away 
from zero and infinity,
\begin{equation*}
\begin{split}
C_{\om,\varphi,\psi} (n)
	& =\left| \int_{X_{\om}} (\varphi\circ f_\om^n)\psi h_\om \, d\nu_{\om} \right| 
	 = \left| \int_{X_{\om}} (\varphi\circ f_\om^n)\psi h_\om \, (\lambda_\om^{n})^{-1} d\cL_{\om}^{*n}\nu_{\te^n\om}  \right|  
	\\
	& = \left|\int_{X_{\theta^n\om}} \varphi\; \frac{ (\lambda_\om^{n})^{-1} {\cL}_{\om}^n (\psi h_\om)}{h_{\te^n\om}}
	d\mu_{\te^n\om} \right|
	\leqslant 
	\left\|\frac{ (\lambda_\om^{n})^{-1} {\cL}_{\om}^n (\psi h_\om)}{h_{\te^n\om}}\right\|_{0}   \cdot\|\varphi\|_{C^1}
\end{split}
\end{equation*}
where $\lambda_\om^n:=\prod_{j=0}^{n-1} \lambda_{\te^j\om}$ and  $\|\varphi\|_1=\int |\varphi| d\mu$. If $\psi h_\om$ belongs to the cone field then the right hand side decays exponentially fast in $n$.
Otherwise, just write $\psi h_\om=g_\om$ where $g_\om = g_\om^+ - g_\om^-$  and $g_\om^\pm = \frac{1}{2} (|g_\om| \pm g_\om)+B$
for some large $B>0$ so that $g_\om^\pm$ belongs to the cone field and apply the latter estimates to $g_\om^\pm$.
By linearity, the same estimate holds for $g_\om$ for some constant 
 $K(\varphi,\psi)\geqslant K(\varphi,g_\om^+)+K(\varphi,g_\om^-)$. This concludes the proof of the exponential decay of correlations.

\section{Expanding measures}\label{s.expanding}

In this section we prove that the probability measure $\mu$ given by Theorem~\ref{construction.eigenf.} is expanding
along the fibers, meaning that it has only positive Lyapunov exponents along the fiber direction (cf. Proposition~\ref{prop:expanding}).
 As $\mu$ was constructed analytically, by means of a fixed point approach for the quenched Ruelle-Perron-Frobenius
operators, its geometric properties are far from being immediate.

\subsection{Estimates on points with non-uniform expansion}\label{s.expanding-est}

The starting point is the following estimate on the volume of cylinder sets and its relation with
the Jacobian of the measures $(\nu_\omega)_\omega$.
For every $\om\in \Omega$ and $n\geqslant 1$ set $\cP_\om^{(n)}=\bigvee_{i=0}^{n-1} (f_\om^i)^{-1}(\cP^0_{\theta^i \omega})$
where $\cP_\omega^{0}=\{P_1, \dots,P_{p_\om}, \dots, P_{p_\om+q_\om}\}$ is the covering by domains of
injectivity of $f_\om$, guaranteed by assumption (H1).

\begin{lemma}\label{le:UB-Gibbs}
$
\nu_\om(P) \leqslant \exp \Big\{ \sum_{i=0}^{n-1} (\sup \phi_{\theta^ i\om}-\inf\phi_{\theta^ i\om} -\log \deg (f_{\theta^ i\om}) )\Big\},
$
for $\mathbb P$-almost every $\om$ and every $P\in \cP_\om^{(n)}$.
\end{lemma}

\begin{proof}
Recall that $\cL_\om^* \nu_{\theta\om}=\la_\om \nu_\om$ and  $\cL_\om h_{\om}=\la_\om h_{\theta\om}$ 
for $\mathbb P$-almost every $\om$. This implies that $\deg (f_\om) e^{\inf\phi_\om} \leqslant \la_\om \leqslant \deg (f_\om) e^{\sup\phi_\om}$. Indeed, this is a simple consequence of  the fact that $\deg (f_\om) e^{\inf\phi_\om} \leqslant \cL_\om 1(x) \leqslant \deg (f_\om) e^{\sup\phi_\om}$ for all $x\in X_\om$ and that
\begin{align*}
\la_\omega & = \la_\omega \int _{X_\omega} 1\, d\nu_\om
	= \int _{X_\omega} 1\, d(\cL_\om^* \nu_{\theta\om})= \int _{X_\omega} \cL_\om 1\, d\nu_{\theta\om}.
\end{align*}
Therefore, given $n \geqslant 1$
\begin{equation}\label{eq:lambdas}
\prod_{i=0}^{n-1} \deg (f_{\theta^ i\om}) e^{\inf\phi_{\theta^ i\om}}
	\leqslant \prod_{i=0}^{n-1} \lambda_{\theta^ i\om}
	\leqslant \prod_{i=0}^{n-1}  \deg (f_{\theta^ i\om}) e^{\sup\phi_{\theta^ i\om}}.
\end{equation}
Equation \eqref{eq:lambdas} and the fact that every $P\in \cP_\om^{(n)}$ is a domain of injectivity for $f_\om^n$ 
ensures that 
\begin{align*}
1\geqslant \nu_{\theta^n\om}(f_\om^n P) 
	& = \int_P  \prod_{i=0}^{n-1} \lambda_{\theta^ i\om} \; e^{-\sum_{i=0}^{n-1} \phi_{\theta^i\om} (f_\om^ix)} \, d\nu_\om(x)
	\\ & \geqslant \prod_{i=0}^{n-1} \deg (f_{\theta^ i\om}) e^{\inf\phi_{\theta^ i\om}- \sup \phi_{\theta^ i\om}} \; \nu_\om(P)
\end{align*}
which proves the lemma.
\end{proof}

The estimate in Lemma~\ref{le:UB-Gibbs} will allow us to prove that the measure of cylinder sets in $\cP_\om^{(n)}$ decrease
exponentially fast, a fact that will be used later on. %
Given $\al>0$ and $n\geqslant 1$ consider the time-$n$ exceptional set
$$
E_\om(\al,n):=\Big\{ x\in X_\om \colon \frac1n \sum_{i=0}^{n-1} \log \|Df_{\theta^i\om}(f_\om^i(x))^{-1}\|^{-1} \leqslant  \al\Big\}
$$
of points $x$ whose time averages of the least expansion along the $\om$-random orbit of $x$ is bounded above by $\al$.
One of the hard points in the proof of the hyperbolicity is a counting argument on the number of $n$-cylinder sets
in $\cP_\om^{(n)}$ that present no good average hyperbolicity, that is, that intersect $E_\om(\al,n)$.
Indeed, the combinatorial argument involves counting number of such cylinders when the proportion of
good cylinders in $\cP^0_\om$ may be arbitrarily close to zero (as a function of $\omega$).

\begin{proposition}\label{prop:expanding}
There exists $\al>0$ so that for $\mathbb P$-a.e. $\omega$ and $\nu_\om$ almost every $x\in X_\om$ 
\begin{equation}\label{eq:HT}
\liminf_{n\to\infty} \frac1n \sum_{i=0}^{n-1} \log \|Df_{\theta^i\om}(f_\om^i(x))^{-1}\|^{-1} \geqslant  \al.
\end{equation}
\end{proposition}

\begin{proof}
First we observe that $ \#\{P\in \cP_\om^{(n)} \colon P \cap E_\om(\al,n) \neq \emptyset\}$ coincides with %
$$
 \#\Big\{ (i_0, i_1, \dots, i_{n-1}) \in \prod_{i=0}^{n-1} \{1, \dots, p_{\theta^i\omega}+q_{\theta^i\omega}\} \colon 
	E_\om(\al,n) \cap P(i_0, \dots, i_{n-1}) \neq \emptyset)\Big\},
$$
where $P(i_0, \dots, i_{n-1}) = P_{i_0} \cap f_\om^{-1} P_{i_1} \cap \dots \cap (f_\om^{n-1})^{-1} P_{i_{n-1}}$. Furthermore, 
for each $(i_0, i_1, \dots, i_{n-1}) \in \prod_{i=0}^{n-1} \{1, \dots, p_{\theta^i\omega}+q_{\theta^i\omega}\}$ set 
$$
D_\omega (i_0, i_1, \dots, i_{n-1}) = d_\om(i_0) \cdot d_\om(i_1) \cdot \dots \cdot d_\om(i_{n-1})
$$
where 
$$
d_\om(i_k) =
\begin{cases}
\sigma_{\theta^k\omega}^{-1}, \quad \text{if} \, 1\leqslant i_k \leqslant p_{\theta^k\om} \\
L_{\theta^k\omega}, \quad \text{otherwise}
\end{cases}	
\quad\text{for every} \; 0\leqslant k \leqslant n-1.
$$
If $x\in E_\om(\al,n) \cap P(i_0, \dots, i_{n-1})$ then 
$
-\al \leqslant \frac1n \sum_{i=0}^{n-1} \log \| Df_{\theta^i\om}(f_\om^i(x))^{-1} \| \leqslant \frac1n \sum_{i=0}^{n-1} \log d(i_k)
$
and, consequently, 
\begin{align*}
\Big\{ (i_0, & i_1, \dots, i_{n-1})  \colon 
	E_\om(\al,n) \cap P(i_0, \dots, i_{n-1}) \neq \emptyset)\Big\} \\
	& \subset \Big\{ (i_0,  i_1, \dots, i_{n-1})  \colon 
	\frac1n \log D_\omega (i_0, i_1, \dots, i_{n-1}) \geqslant -\al\Big\}.
\end{align*}
Set $\eta_{\om}=\{1, \dots, p_\om\}$ and  $\zeta_{\om}=\{p_\om+1, \dots, q_\om\}$.
Observe that $D_\om (i_0, i_1, \dots, i_{n-1})=D_\om(\tilde i_0, \tilde i_1, \dots, \tilde i_{n-1})$
if $i_k$ and $\tilde i_k$ belong to the same element in $\{\eta_{\theta^k\om}, \zeta_{\theta^k\om}\}$.
Then
\begin{align}\label{eq:combinat}
\# \Big\{ (i_0,  i_1, \dots, i_{n-1}) &  \colon 
	\frac1n \log D_\omega (i_0, i_1, \dots, i_{n-1}) \geqslant -\al\Big\}  \nonumber \\
	& \leqslant \sum_{\substack{j_0,\dots, j_{n-1} \in \{0,1\} \\ \frac1n \sum_{i=0}^{n-1} \log b_\om(j_i) \geqslant -\al }} a_\om(j_0) \dots a_\om(j_{n-1}),
\end{align}
where
$$
b_\om(j_i)
=\begin{cases}
\sigma^{-1}_{\theta^i\om}, \quad \text{if } \, j_i=1, \\
L_{\theta^i\om},\qquad \text{if}\, j_i=0
\end{cases}
\qquad\text{and}\qquad
a_\om(j_i)
=\begin{cases}
p_{\theta^i\om}, \quad \text{if } \, j_i=1, \\
q_{\theta^i\om},\qquad \text{if}\, j_i=0
\end{cases}
$$	
for $0\leqslant k \leqslant n-1$.
Using assumption (H3) together with the ergodicity of $(\theta,\Omega,\mathbb P)$ we conclude that for any $\vep>0$
$$
\frac1n \sum_{i=0}^{n-1} \log( \sigma_{\theta^i\om}^{-1}p_{\theta^i\om}+L_{\theta^i\om} q_{\theta^i\om}) 
	\leqslant C +\vep
$$
 where $C:=\int_\Om \log( \sigma_\om^{-1}p_\om+L_\om q_\om) \, d\mathbb P< 
\int \log \deg (f_\om) \, d\mathbb P$ provided that $n$ is large enough. Moreover, observe that
\begin{equation}\label{eq:combinat1}
\prod_{i=0}^{n-1} (\sigma_{\theta^i\om}^{-1}p_{\theta^i\om}+L_{\theta^i\om} q_{\theta^i\om})
	=\sum_{j_0,\dots, j_{n-1} \in\{0,1\}^n} \prod_{i=0}^{n-1} a(j_i) b(j_i).
\end{equation}
So, using that 
\begin{align*}
\sum_{j_0,\dots, j_{n-1} \in\{0,1\}^n} \prod_{i=0}^{n-1} a(j_i) b(j_i)
	& \geqslant \sum_{\substack{j_0,\dots, j_{n-1} \in\{0,1\}^n\\ \frac1n \sum_{i=0}^{n-1} \log b_\om(j_i) \geqslant -\al}}
	\prod_{i=0}^{n-1} a(j_i) b(j_i) \\
	& \geqslant e^{-\al n} \sum_{\substack{j_0,\dots, j_{n-1} \in\{0,1\}^n\\ \frac1n \sum_{i=0}^{n-1} \log b_\om(j_i) \geqslant -\al}}
	\prod_{i=0}^{n-1} a(j_i) 
\end{align*}
we conclude that
\begin{equation}\label{eq:combinat2}
\sum_{\substack{j_0,\dots, j_{n-1} \in\{0,1\}^n\\ \frac1n \sum_{i=0}^{n-1} \log b_\om(j_i) \geqslant -\al}}
	\prod_{i=0}^{n-1} a(j_i)
	\leqslant e^{(C+\al+\vep)n}
\end{equation}
for every large $n$.
This, together with Lemma~\ref{le:UB-Gibbs} and equations \eqref{eq:combinat} and \eqref{eq:combinat1}, proves that
\begin{align*}
\nu_\omega (E_\om(\al,n)) 
	 \leqslant & \nu_\om \Big(\bigcup_{P\in \cP_\om^{(n)}} E_\om(\al,n) \cap P\Big) \\
	 \leqslant & \#\{P\in \cP_\om^{(n)} \colon P \cap E_\om(\al,n) \neq \emptyset\} \times \\ 
	 \phantom{\leqslant }  
	& \exp \Big\{ \sum_{i=0}^{n-1} (\sup \phi_{\theta^ i\om}-\inf\phi_{\theta^ i\om} -\log \deg (f_{\theta^ i\om}) )\Big\} \\
 \leqslant & \exp \Big\{ (\int_\Om \log( \sigma_\om^{-1}p_\om+L_\om q_\om) \, d\mathbb P +\al +\vep)n \Big\} \times  \\
   & \exp \Big\{  \sum_{i=0}^{n-1} (\sup \phi_{\theta^ i\om}-\inf\phi_{\theta^ i\om} -\log \deg (f_{\theta^ i\om}) )\Big\}
\end{align*}
for all large $n$. Assumption (P) guarantees the latter is summable provided that $\al,\vep>0$ are small, 
and the result follows from the Borel-Cantelli lemma. 
\end{proof}

\subsection{Weak Gibbs property}

In this subsection we shall prove Theorem~\ref{thm:weak.Gibbs}, that is, that the measure $\nu$ given by 
Theorem~\ref{thm:Ruelle}, satisfies a weak Gibbs property a sequence of instants of hyperbolicity. 

\subsubsection{Hyperbolic times}\label{sec:HT}
We recall briefly the concept of hyperbolic times in random dynamical systems.
Given $\gamma>0$, we say $n$ is a
{\em $\gamma$-hyperbolic time} for  $(\omega, x)\in X$ if there exists $C(\om)>0$ such that
\begin{equation}\label{eq:HTT}
\prod_{j=n-k+1}^{n}\|Df_{\theta^j(\omega)}(f^j_{\omega}(x))^{-1}\| \leqslant e^{-\gamma k}
	\quad \text{for all $1\leqslant k \leqslant n$}.
\end{equation}
The existence of hyperbolic times arises as a consequence of the following lemma. 

\begin{lemma}(Pliss lemma)\label{Pliss} \cite{Pliss}
Let $A\geqslant c_2 > c_1 >0$ and $\xi =(c_2-c_1)/(A-c_1)$. If $(a_i)_{1\leqslant i \leqslant n}$
is a sequence such that $a_i \leqslant A$ for every ${1\leqslant i \leqslant n}$ and $\frac1N\sum_{j=1}^N a_j \geqslant c_2$
then there exist $\ell\geqslant \xi N$ and $1\leqslant n_1 < n_2 < \dots < n_\ell \leqslant N$ so that,
for each $1\leqslant i \leqslant \ell$,
$$
\sum_{j=n+1}^{n_i} a_j \geqslant c_2 \, (n_i - n)
	\qquad \text{for every $0\leqslant n < n_i$.}
$$
\end{lemma}

Given a $\mathbb P$-typical point $\om$ and $\alpha>0$ given by Proposition~\ref{prop:expanding}, 
for each $N\geqslant 1$, consider the sequence $a_j(\om)= \log \|Df_{\theta^i\om}(f_\om^i(x))^{-1}\|^{-1}$,
$0\leqslant n \leqslant N-1$. Taking the constants
$c_2=2c_1=\alpha$ and 
$$
A=A_N(\om):= \max_{0\leqslant n \leqslant N-1} \, \big[\max_{x\in X_{\theta^i\om}}\log \|Df_{\theta^i\om}(x)^{-1}\|^{-1} \big]
$$
in Lemma~\ref{Pliss} we conclude that there exist $\ell \geqslant \frac{\al N}{2A_N(\om) -\al }$ of $\frac{\al}2$-hyperbolic times
in the time interval $[1,N]$. Hypothesis (H0c) together with Birkhoff's ergodic theorem implies that 
$\limsup_{N\to\infty} \frac{A_N(\om)}N=0$. Hence, $\mu$-almost every $(\om,x)$ has infinitely many $\gamma$-hyperbolic times
for $\gamma=\frac\al2$.

\begin{remark}\label{ht-density}
In the context of deterministic dynamics and random perturbations of a certain map (e.g. \cite{AMO})  
the constant $A$ above may be chosen independent of $\om$ and $N$, thus yielding the stronger conclusion that 
hyperbolic times have positive density. 
Hence, assumption (H5') is a sufficient condition for positive density of hyperbolic times
in our random context. 
\color{black}
\end{remark}

We proceed to prove that the latter
ensures the existence of local unstable manifolds for typical points $x\in X_\om$.

\begin{lemma}\label{p.hypreball}
Given $0<\gamma<1$ there exists $\vep_0>0$ (depending only on $\gamma$ and $f=(f_\om)_\om$) such that
the following holds:  if $n$ is a $\gamma$-hyperbolic time for $(\omega, x)\in X$ then for every $0<\vep<\vep_0$
\begin{enumerate}
\item $f^{n}_{\omega}$ maps $B_\om(x,n,\vep)$ diffeomorphically onto $B(f^{n}_{\omega}(x),\vep)$;
\item for every $1\leqslant k\leqslant n$ and $y, z\in B_\om(x,n,\vep)$, $$ \dist(f_{\omega}^{n-k}(y),f_{\omega}^{n-k}(z)) \leqslant
e^{-\gamma k/2}\dist(f_{\omega}^{n}(y),f_{\omega}^{n}(z)),$$
\end{enumerate}
where the dynamic ball  $B_\om(x,n,\vep)\subset X_\om$ is defined by \eqref{eq.dynball}.
\end{lemma}

\begin{proof}

Let $\de_0$ be given by (H0b). By (H4) there exists $0<\vep_0<\de_0$ such that 
$\|Df_\om^{-1}(x)\| \leqslant e^{\frac\gamma2} \|Df_\om^{-1}(y)\|$
for every $y\in B(x,\vep_0)\subset X_\om$ and $\mP$-a.e. $\om$.

Assume that $n$ is a $\gamma$-hyperbolic time for $(\omega, x)\in X$ and $0<\vep<\vep_0$. By construction, the inverse branch of $f_{\theta^{n-1}\om}$ which maps $f_{\om}^n(x)$ to $f_{\om}^{n-1}(x)$ is defined
in the ball of radius $\vep_0$ and $\| (Df_{\theta^{n-1}\om}(y))^{-1}\| \leqslant e^{-\gamma /2}$ for every $y\in B(f_{\om}^n(x),\vep_0)$. This ensures that $f_{\theta^{n-1}\om}^{-1}\mid_{B(f_{\om}^n(x),\vep)}$ is a uniform contraction and
a diffeomorphism onto $U_\om^{n-1}:= f_{\theta^{n-1}\om}^{-1}({B(f_{\om}^n(x),\vep)}) \subset {B(f_{\om}^{n-1}(x),\vep)}$.
Moreover,   
$$\dist(y,z) \leqslant e^{-\gamma/2}\dist(f_{\theta^{n-1}\om}(y),f_{\theta^{n-1}\om}(z))$$
for every $y,z\in U_\om^{n-1}$. Using \eqref{eq:HTT} recursively we obtain that  for every $1\leqslant j \leqslant n$ the map
$(f^j_{\theta^{n-j}\om})^{-1}\mid_{B(f_{\theta^{n}\om}(x),\vep)}$ is a uniform contraction, of contraction rate
$e^{-\gamma j/2}$, onto a open neighborhood $U_\om^{n-j} \subset X_{\theta^{n-j}\om}$ of $f_{\om}^{n-j}(x)$.
By construction, the set $U_\om^0=B_\om(x,n, \vep)$ is mapped diffeomorphically by $f^{n}_{\omega}$ 
onto $B(f^{n}_{\omega}(x),\vep)$. Item (2) is immediate by construction.
\end{proof}

\subsubsection{Gibbs property at hyperbolic times}\label{sec:GHT}

Since $\mu=(\mu_\om)_\om$ is such that $\mu_\om\ll \nu_\om$ for $\mP$-a.e. $\om$ and 
$\nu_\om$ satisfies \eqref{eq:HT} then $\mu$-almost every $(\om,x)\in X$ has infinitely many $\alpha/2$-hyperbolic times.
In what follows we shall write $\alpha/2$-hyperbolic time simply as hyperbolic time.
The following is a standard consequence of the backward contraction at hyperbolic times.

\begin{lemma}\label{p.bdistort}
For every $0<\vep<\vep_0$ and $\mP$-a.e. $\om$, 
there exist $K(\om)\geqslant 0$ and $0< L(\omega,\vep) \leqslant 1$ satisfying    $\int K(\om) d\mathbb P < \infty$ 
and $\int \log L(\om,\vep)\, d\mathbb P>-\infty$
such that, for any  
hyperbolic time  $n$,  

\begin{equation}\label{eq:WG}
\frac{L(\theta^n\omega,\vep) }{e^{\vep K(\theta^n\om) }}
\leqslant 
 \frac{\nu_\om(B_\om(x,n,\vep))}{\big(\prod_{i=0}^{n-1} \la_{\theta^i\om}\big)^{-1} \; \exp \big( S_n\phi_\om(x) \big) } 
 	\leqslant  e^{\vep K(\theta^n\om)}.
\end{equation} 
where $S_n\phi_\om(x):=\sum_{j=0}^{n-1} \phi_{\theta^j\omega}\circ f_\omega^j(x)$. 
\end{lemma}
\color{black}

\begin{proof}
If $A\subset X_\om$ is such that $f_\om\mid_A$ is injective and $(g_n)_n$ is a sequence of continuous functions on 
$X_\om$ such that $g_n\to\chi_A$ at $\nu_\om$-almost every point and $\sup |g_n| \leqslant 2$
for all $n$, then
$
\cL_\om(e^{-\phi_\om}g_n)(x) 
	=\sum_{f_\om(y)=x} g_n(y),
$
where the right hand side converges to $\chi_{f_\om(A)}(x)$ at $\nu_\om$-almost
every point, because $f_\om\mid _A$ is injective. The dominated convergence theorem implies that
$
\int \lambda_\om e^{-\phi_\om} g_n \, d\nu_\om
 = \int e^{-\phi_\om} g_n \, d(\cL_\om^*\nu_{\te\om})
 = \int \cL_\om(e^{-\phi_\om}g_n) \, d\nu_{\te\om}$
 tends to
 $\nu_{\te\om}(f_\om(A))$ as $n\to\infty$.
Since the left hand side converges to $\int_A \lambda e^{-\phi_\om}
d\nu_\om$, we conclude that
$
\nu_{\te\om}(f_\om(A))=\int_A \lambda_\om e^{-\phi_\om} d\nu_\om,
$
which proves that $J_{\nu_\om} f_\om=\lambda_\om e^{-\phi_\om}$ is the Jacobian of $f_\om$ with respect to $\nu_\om$
(hence the Jacobian, which is almost everywhere defined, admits a $C^1$-representative).
Now, item (1) in Lemma~\ref{p.hypreball} guarantees that 
$f^{n}_{\omega}$ maps $B_\om(x,n,\vep)$ diffeomorphically onto $B(f^{n}_{\omega}(x),\vep)$, 
and consequently 
$$
\nu_{\te^n\om}(B(f^{n}_{\omega}(x),\vep) ) 
	= \prod_{i=0}^{n-1} \la_{\theta^i\om}  \int_{B_\om(x,n,\vep)}  e^{-S_n\phi_\om(z)} \, d\nu_\om(z).
$$
Moreover, with $K(\omega) := \sum_{i=0}^{\infty} \|\phi_{\te^{-i}\om}\|_{C^1} \; e^{-\al i/ 2}$, 
\begin{align*}
|S_n\phi_\om(y) - S_n\phi_\om(z) | & \leqslant \sum_{i=0}^{n-1} |\phi_{\te^{n-i}\om}(f_\om^{n-i}(y)) - \phi_{\te^{n-i}\om}(f_\om^{n-i}(z)) | \\
	& \leqslant \sum_{i=0}^{n-1} \|\phi_{\te^{n-i}\om}\|_{C^1} \; e^{-\al i/ 2}\, d(f_\om^{n}(y), f_\om^{n}(z))\\
	& \leqslant \vep \sum_{i=0}^{\infty} \|\phi_{\te^{n-i}\om}\|_{C^1} \; e^{-\al i/ 2} = K(\theta^n \omega) \vep
\end{align*}
for every  $y,z\in B_\om(x,n,\vep)$. Observe that the right hand side is integrable, as $\int \|\phi_{\om}\|_{C^1} dP < \infty$.  
Furthermore, it is not hard to check that %
the $\nu_\om$-measure of every ball of radius $\vep$ centered at some point of 
$\supp \nu_\omega\subset X_\omega$ is bounded from below by some constant $L(\om,\vep)>0$ by compactness, and from above by 1. Hence, for any  $y\in B_\om(x,n,\vep)$,
\begin{align*}
\prod_{i=0}^{n-1} \la_{\theta^i\om}   e^{-S_n\phi_\om(y)}  \nu_{\om} (B_\om(x,n,\vep)) 
  & = e^{\pm K(\theta^n \omega) \vep} \prod_{i=0}^{n-1} \la_{\theta^i\om}  \int_{B_\om(x,n,\vep)}  e^{-S_n\phi_\om(z)} \, d\nu_\om(z)\\
  & =  e^{\pm K(\theta^n \omega) \vep} \nu_{\te^n\om}(B(f^{n}_{\omega}(x),\vep)) \\ 
  &  \in [L(\theta^n \omega,\vep) e^{- K(\theta^n \omega) \vep} , e^{K(\theta^n \omega) \vep)}  ].
\end{align*}
This proves the lemma. 
\end{proof}

\begin{remark}
Since $d \mu_\om/d\nu_\om$ is bounded above and below, for $\mathbb P$-almost every $\om$, 
 the probability $\mu=(\mu_\om)_\om$ also satisfies a weak Gibbs property similar to \eqref{eq:WG}. That is, for $\tilde{K}(\om) :=  \| \log h_\omega \|_\infty$, we have 
 $\int \tilde{K}(\om) d\mathbb P<\infty$ and  
for every $0<\vep<\vep_0$ and $\mP$-a.e. $\om$, if $n$ is a hyperbolic time then  
\[
\frac{L(\theta^n\omega,\vep)}{e^{\vep {K}(\theta^n\om) + \tilde{K}(\om)}} 
\leqslant 
 \frac{\mu_\om(B_\om(x,n,\vep))}{\big(\prod_{i=0}^{n-1} \la_{\theta^i\om}\big)^{-1} \; \exp \big( S_n\phi_\om(x) \big) } 
 	\leqslant  e^{\vep {K}(\theta^n\om) + \tilde{K}(\om)}.
\] 
\end{remark}

\subsubsection{Proof of Theorem~\ref{thm:weak.Gibbs}}
In view of Lemma~\ref{p.bdistort}, and changing the random variable $K$ if necessary,  it is enough to prove that 
$$L(\theta^n\omega,\vep)=\min_{z\in \text{supp}(\nu_{\te^n\om})}\nu_{\te^n\om}(B(z,\vep)) \leqslant 1$$
is either bounded away from zero along some subsequence of hyperbolic times, or that it has sub-exponential convergence
to zero. This can be obtained using either assumptions (H5) or (H5'). 
 
Assume first the hypothesis (H5). For every $\vep>0$ there exists $N(\vep)\geqslant 1$ so that 
$f_\om^{N(\vep)}(B(x,\vep))= X_{\theta^{N(\vep)}\om}$ for every $x \in X_\om$ and $\mathbb P$-almost every $\om$. 
Thus, the existence of a Jacobian ensures that 
$$
1 =  \nu_{\te^{n+N(\vep)}\om} (X_{{\te^{n+N(\vep)}\om}})
	\leqslant e^{ K(\theta^{n+N(\vep)} \omega) \vep} \prod_{i=0}^{N(\vep)-1} \la_{\theta^{n+i}\om}  \; e^{-S_{N(\vep)}\phi_{\te^n\om}(z)} \, \nu_{\te^n\om}(B(z,\vep))
$$
for any $z\in X_{\te^n\om}$. Thus, 
$$
1\geqslant L(\theta^n\omega,\vep)
	\geqslant e^{-K(\theta^{n+N(\vep)} \omega) \vep} \Big(\prod_{i=0}^{N(\vep)-1} \la_{\theta^{n+i}\om}\Big)^{-1}  
		\; e^{S_{N(\vep)}\min \phi_{\te^n\om}}. 
$$
Hence, $L(\om,\vep)$
 satisfies the integrability condition since the integrals $\int K(\om) \, d\mathbb P$, 
$\int \log \lambda_\om \, d\mathbb P$ and $\int \min \phi_\om\, d\mathbb P$ are finite. 
This proves the theorem under this assumption.

Assume, alternatively, that the hypothesis (H5') holds. By Remark~\ref{ht-density}, the sequence of hyperbolic times has positive density $\xi\in (0,1)$ in the positive integers. Choosing $L>0$ small so that 
$\mathbb P(\om\colon L(\om,\vep)\geqslant L) > 1-\frac\xi3$, the ergodic theorem ensures that 
$$
 \#\Big\{ 0\leqslant j \leqslant n-1\colon L(\theta^j\om,\vep) \geqslant L \Big\}  \geqslant 1-\frac{2\xi}{3}
$$
for every large $n\geqslant 1$. Hence, there exists a subsequence of hyperbolic times $(n_k)_{k\geqslant 1}$ 
with positive density in the integers such that $L(\theta^{n_k}\om,\vep) \geqslant L$ for every $k\geqslant 1$.
This proves the theorem.
\color{black}

\section{Hyperbolic potentials and high temperature}\label{s.temperature}

The main goal of this section is to prove Theorem~\ref{thm:HP}, which provides sufficient conditions
for the construction of hyperbolic potentials at high temperature.  
First we define the notion of hyperbolic potentials and establish sufficient conditions, based on 
their average oscillation, for a potential to be hyperbolic. Then we provide general estimates on the 
entropy of set of points with no hyperbolicity for the random dynamical system and use these in the proof
of the theorem.

\subsection{Hyperbolic potentials} 
The notion of hyperbolic potential can be traced back to one-dimensional dynamics, where one requires the potential
$\phi: [0,1] \to \mathbb R$ to satisfy $\sup\phi < P_{top}(f,\phi)$ (a condition which ensures all measures with large topological
pressure have positive topological entropy, see also \cite{DU}). 
An axiomatization of this concept appeared in \cite{RV}. In our context this will be defined as follows.

\begin{definition}\label{def:hp}
A potential $\phi=(\phi_\om)_\om \in L^1_X(\Om,C^1(M))$ is called \emph{hyperbolic} if
there exists $\alpha>0$ so that 
$$
\pi_\phi(f,H^c) < \pi_\phi(f),
$$
where 
$H\subset X$ denotes the set of points with infinitely many $\frac\alpha2$-hyperbolic times.
\end{definition}

\subsection{Entropy of the set of points with no expansion} 

Here we give an upper bound for the relative entropy of the set of points with finitely many hyperbolic times.
We assume that $M$ is a $m$-dimensional compact Riemannian manifold, hence it is a 
Besicovitch metric space (cf. \cite{Fed69}). Thus, we have the following:

\begin{lemma}\cite[Lemma~6.2]{VV10}\label{l.covers}
There exists $C>0$ and a sequence of finite open coverings
$(\cQ_k)_{k\geq 1}$ of $M$ such that $\diam(\cQ_k) \to 0$ as $k \to
\infty$, and every set $E\subset M$ satisfying $\diam(E)\leq D\diam
\cQ_k$ intersects at most $C D^m$ elements of $\cQ_k$.
\end{lemma}

The previous lemma will be instrumental to obtain upper bounds for the topological entropy of the 
random dynamical systems satisfying assumptions (H0)-(H5). Indeed,
since $\diam(\cQ_k) \to 0$ as
$k \to \infty$ then
\begin{equation}\label{eq:partitionQk}
\pi_\phi(\omega,f,H^c)
    =\lim_{k \to \infty} \pi_\phi(\omega,f,H^c,\Omega\times \cQ_k)
\end{equation}
for every potential $\phi\in L^1(\Omega, C^0(M))$. In order to estimate the random entropy at the set $H^c$ 
that has at most finitely many hyperbolic times one needs only to estimate the right hand side of \eqref{eq:partitionQk} 
in the case of the potential $\phi\equiv 0$. Such technical estimates resemble estimates on the Hausdorff dimension of 
fibered dynamics. 

\begin{proposition}\label{l.r.entropy.iteration}
For $\mathbb P$-almost every $\omega\in\Omega$ and every $\ell \geq 1$, the following holds:
$$
\pi_0(f,H^c)  \leqslant \int_\Om \log( \sigma_\om^{-1}p_\om+L_\om q_\om) \, d\mathbb P+ m \int \log L_\om \, d\mathbb P +\alpha.
$$
\end{proposition}

\begin{proof}
Let $\cP_\om$ denote the finite covering of $X_\om$ given by assumption $(H1)$.
Recall that there exists $\alpha>0$ so that the set 
$$
E_\om(\alpha,n):=\Big\{ x\in X_\om \colon \frac1n \sum_{i=0}^{n-1} \log \|Df_{\theta^i\om}(f_\om^i(x))^{-1}\|^{-1} \leqslant  \alpha\Big\}
$$
of points $x$ whose time averages of the least expansion along the $\om$-random orbit of $x$ is bounded above by $\alpha$,
is a $\nu_\om$-measure zero subset of $X_\omega$ (recall Proposition~\ref{prop:expanding}). Actually, using the 
notation used in the proof of Proposition~\ref{prop:expanding}, for each $N \geq 1$ the set $H_\omega^c:=H^c \cap X_\om$ is covered by
$
\bigcup_{n \geq N} 
    \big\{P\in \cP_\om^{(n)} \colon P \cap E_\om(\al,n) \neq \emptyset\big\}.
$
In particular, $H_\omega^c$ is covered by
\begin{align}
\bigcup_{n \geq N} & \bigcup_{(i_0, i_1, \dots, i_{n-1})}  \Big\{ P(i_0, \dots, i_{n-1})  \colon 
	E_\om(\al,n) \cap P(i_0, \dots, i_{n-1}) \neq \emptyset)\Big\}  \nonumber \\
	& \subset
	\bigcup_{n \geq N} \bigcup_{(i_0, i_1, \dots, i_{n-1})} \Big\{ P(i_0, \dots, i_{n-1})  \colon 
	\frac1n \log D_\omega (i_0, i_1, \dots, i_{n-1}) \geqslant -\al \Big\}, \label{eq:inclusion-1}
\end{align}
where the union is taken over $n$-uples 
$(i_0, i_1, \dots, i_{n-1}) \in \prod_{i=0}^{n-1} \{1, \dots, p_{\theta^i\omega}+q_{\theta^i\omega}\}$,
$P(i_0, \dots, i_{n-1}) \in \cP_\om^{(n)}$ and 
$
D_\omega (i_0, i_1, \dots, i_{n-1}) = d_\om(i_0) \cdot d_\om(i_1) \cdot \dots \cdot d_\om(i_{n-1}),
$
with
$$
d_\om(i_k) =
\begin{cases}
\sigma_{\theta^k\omega}^{-1}, \quad \text{if} \, 1\leqslant i_k \leqslant p_{\theta^k\om} \\
L_{\theta^k\omega}, \quad \text{otherwise}
\end{cases}	
\quad\text{for every} \; 0\leqslant k \leqslant n-1.
$$
In order to estimate the entropy with respect to the random family $f^\ell=(f_\om^\ell)_\om$ we shall consider iterations 
by multiples of $\ell$. Fix $\zeta>0$ arbitrary. First, a simple argument involving Euclid's division algorithm ensures that for any the right hand side in 
\eqref{eq:inclusion-1} is contained in the union
\begin{align*}
	\bigcup_{j \geq \frac{N}\ell} \bigcup_{(i_0, i_1, \dots, i_{j\ell-1})} \Big\{ P(i_0, \dots, i_{j\ell-1})  \in \cP_\om^{(j\ell)} \colon 
	\frac1{j\ell} \log D_\omega (i_0, i_1, \dots, i_{j\ell-1}) \geqslant -\al -\zeta\Big\}
\end{align*}
provided that $N\ge 1$ is large enough.
This allow us to write 
\begin{align}\label{eq:estimate-Hc1}
m_\beta(\om,f^\ell,0,H^c,\cQ_k,N) 
	& \leqslant 
		\sum_{j \geq \frac{N}\ell} \sum_{P} \sum_{Q_P} e^{-\beta n(Q_P)}
	 =
	\sum_{j \geq \frac{N}\ell}  \sum_{P} \sum_{Q_P}  e^{-\beta j}
\end{align}
where the second sum concerns 
\begin{equation}\label{eq:first-collection}
\{ P  \in \cP_\om^{(j\ell)} \colon  \frac1{j\ell} \log D_\omega (i_0, i_1, \dots, i_{j\ell-1}) 
\geqslant -\al -\zeta\}
\end{equation} 
and the third one is taken over the set
$$
\Big\{ Q_P=Q_0 \cap f_\om^{-\ell}(Q_1) \cap \dots \cap
  f_\om^{-\ell(j-1)}(Q_{j-1}) : Q_i \in \cQ_k, i=0,\dots,j-1 \; \&\; Q_P \cap P \neq \emptyset \Big\}.
$$
of $j$-cylinders for $f^\ell$ generated by elements of the covering $\cQ_k$
which intersect such an element $P\in \cP^{(j\ell)}_\om$ formed by points with controled expansion rates.
Thus, in order to provide an upper bound for \eqref{eq:estimate-Hc1} it is enough to bound the cardinality of the
two previous sets of cylinders.

Concerning the cardinality of the first set, equations \eqref{eq:combinat1} and \eqref{eq:combinat2} imply that 
\begin{align*}
\# \Big\{ P  \in  \cP_\om^{(j\ell)} & \colon  \frac1{j\ell} \log D_\omega (i_0, i_1, \dots, i_{j\ell-1})  \geqslant -\al -\zeta\Big\} \\
	& \leqslant \exp\Big( \Big[\int_\Om \log( \sigma_\om^{-1}p_\om+L_\om q_\om) \, d\mathbb P +\alpha + 2\zeta\Big] j\ell \Big). 
\end{align*}

We proceed by fixing $P\in \cP_\om^{(j\ell)}$ in ~\eqref{eq:first-collection} and to
cover such element in by $j$-cylinders  relatively to the transformation $f^\ell$ and the covering $\cQ_k$.
Recall that the maps $f_\om$ are $C^1$-local diffeomorphisms and admit uniform domains of invertibility (recall assumption (H0b)), if $k\geqslant 1$ is large then for $\mathbb P$-almost every $\tilde \omega\in \Om$ the inverse branch 
$$f_{\tilde \om}^{-\ell} := (f_{\tilde \om}^{\ell}\mid_{ f_{\tilde \om}^{\ell}(P)})^{-1} : f_{\tilde \om}^{\ell}(P) \to P$$ 
extends to the union of all $Q\in \cQ_k$ so that $Q \cap f_{\tilde \om}^{\ell}(P) \neq \emptyset$.
Now, the estimates on the behavior of backward iterates in assumption (H1) imply that 
$$
\diam(f_{\tilde \om}^{-\ell}(Q)) \leq \Big[\prod_{i=0}^{\ell-1} L_{\theta^i\tilde\om} \Big]\, \diam(Q) \quad\text{ for every} \;Q
\in \cQ_k.
$$
In particular, Lemma~\ref{l.covers} guarantees that $f_{\tilde \om}^{-\ell}(Q)$ intersects at most 
$C  \Big(\prod_{i=0}^{\ell-1} L_{\theta^i\tilde \om} \Big)^m$ elements of the covering $\cQ_k$. 
Therefore, as each $j$-cylinders for $f^\ell$  determined by $\cQ_k$ and intersecting $P$ involves the backward iterates by
$\{f_{\theta^{i\ell} \om}^{-\ell} \colon 0 \leqslant i \leqslant j-1\}$, the number of those needed to cover $P$  is at most
$$
\#\,\cQ_k \, \times C^j  \Big(\prod_{i=0}^{j\ell-1} L_{\theta^i\om} \Big)^m
	\leqslant e^{j \big[2\log C + \ell m \int \log L_\om \, d\mathbb P\big]}
$$ 
for every large $j\geqslant 1$.
Altogether this shows 
\begin{align*}
m_\al(\om,f^\ell,0, & H^c,\cQ_k,N) \\
	& \leqslant \sum_{j \geq \frac{N}\ell} e^{-j \big\{\beta-\ell \big[\int_\Om \log( \sigma_\om^{-1}p_\om+L_\om q_\om) \, d\mathbb P +\alpha + 2\zeta\big] \big\}} 
\cdot e^{j \big\{2\log C + \ell m \int \log L_\om \, d\mathbb P \big\}}
\end{align*}
which tends to zero as $N\to \infty$ (independently of $k$) whenever 
$$
\beta > \ell \Big[\int_\Om \log( \sigma_\om^{-1}p_\om+L_\om q_\om) \, d\mathbb P +\alpha + 2\zeta + m \int \log L_\om \, d\mathbb P + \frac{2\log C}\ell\Big]
$$
This proves that $\pi_0(\om, f^\ell,H^c,\cQ_k)$ is bounded by the right hand side above
for every $k\ge 1$ and, consequently,
\begin{align*}
\pi_0(f,H^c) & = \frac1\ell \pi_0(f^\ell,H^c) =  \frac1\ell \int \pi_0(\om, f^\ell,H^c) \, d\mathbb P \\
		& \leqslant \int_\Om \log( \sigma_\om^{-1}p_\om+L_\om q_\om) \, d\mathbb P +\alpha + 2\zeta + m \int \log L_\om \, d\mathbb P + \frac{2\log C}\ell
\end{align*}
for every $\ell\geqslant 1$.
Since $\zeta>0$ is arbitrary we get
\begin{align*}
\pi_0(f,H^c)  \leqslant \int_\Om \log( \sigma_\om^{-1}p_\om+L_\om q_\om) \, d\mathbb P+ m \int \log L_\om \, d\mathbb P +\alpha.
\end{align*}
This proves the proposition.
\end{proof}

\subsection{Proof of Theorem~\ref{thm:HP}}
Assume that the random dynamical system $f=(f_\om)_\om$ satisfies (H0)-(H6). 
In our setting the random entropy coincides with 
$$
\pi_0(f) = %
	\int \log \deg f_\om \, d\mathbb P.
$$
This, together with, assumption (H6), guarantees that 
$$
\alpha:= \frac12 \Big[\pi_0(f) - \int_\Om \log  ({\sigma_\om^{-1}p_\om+L_\om q_\om}) \, d\mathbb P + \dim M \cdot \int \log L_\om \, d\mathbb P \Big] > 0.  
$$
By Proposition~\ref{l.r.entropy.iteration}, the complement $H^c$ of the set of points with $\frac\al2$-hyperbolic times
verifies
$$
\pi_0(f,H^c)  \leqslant \int_\Om \log( \sigma_\om^{-1}p_\om+L_\om q_\om) \, d\mathbb P+ m \int \log L_\om \, d\mathbb P +\alpha < \pi_0(f)-\alpha,
$$
which proves the first part of the theorem.

The second part is a direct consequence of the first one. Any 
$\phi\in L^1(\Om, C^1(M))$ with small average oscillation, meaning $\int \sup \phi_\om \,d\mathbb P < \int \inf \phi_\om \,d\mathbb P+\alpha$,
verifies 
$$
\pi_\phi(f,H^c) \leqslant \pi_0(f,H^c) + \int \sup \phi_\om \,d\mathbb P
	< \pi_0(f) + \int \inf \phi_\om \,d\mathbb P
	\leqslant \pi_\phi(f)
$$
hence it is hyperbolic potential
(the first and third inequalities follow directly from Appendix~A).
Clearly, for any other potential $\phi\in L^1(\Om, C^1(M))$ with average oscillation larger than $\alpha$ just take
$T_0= \frac1\alpha \cdot \big[ \int \sup \phi_\om \,d\mathbb P - \int \inf \phi_\om \,d\mathbb P \big]>0$,
and observe that $\frac1T \phi$ is a hyperbolic potential, for every $|T|>T_0$. This proves the theorem.
\hfill $\square$

\section{Existence and uniqueness of equilibrium states for hyperbolic potentials}\label{s.existence}

In this section we prove that hyperbolic potentials always have a unique equilibrium state. 
First, consider the invariant subset 
$$
H:=\Big\{ (\om,x) \in X \colon \liminf_{n\to\infty} \frac1n \sum_{i=0}^{n-1} \log \|Df_{\theta^i\om}(f_\om^i(x))^{-1}\|^{-1} >0 \Big\}.
$$
Clearly, every invariant probability measure $\eta\in \cM(X,f)$ satisfying $\eta(H)=1$ has only 
positive Lyapunov exponents along the fiber direction (we refer the reader to \cite{Arn98b} for Oseledet's theorem 
and definition of Lyapunov exponents). We write $H_\om:= H \cap X_\om$ and $H_\om(n)$ as the set of points $x\in X_\om$
such that $n$ is a hyperbolic time for $(\om,x)$. 
In what follows we need to consider a notion of topological pressure (and variational principle) 
for invariant sets that are not necessarily compact.  The remaining of this section uses the concept of relative pressure 
$\pi_\phi(f,\La)$ for subsets $\La\subset X$ developed 
in Appendices A and B (see Sections~\ref{AppendixA} and \ref{AppendixB}), 
we refer the reader there for the definitions and proofs. 
The following will be instrumental.

\begin{lemma}\label{le:Hm}
Let $\phi=(\phi_\om)_\om \in L^1_X(\Om,C^1(M))$ be a {hyperbolic} potential.
Any ergodic probability $\eta \in \cM(X,f)$  such that $h_\eta(f) +\int\phi \,d\eta > \pi_\phi(f,H^c)$
satisfies $\eta(H)=1$. 
\end{lemma}

\begin{proof}
The variational principle for the relative pressure (Theorem~\ref{thmVP1}) together with Proposition~\ref{prop:III} imply that 
\begin{align*}
\pi_\phi(f)= & \sup_{\eta \in \cM(X, f)} \Big\{h_{\eta}(f)+\int\phi \,d\eta \Big\} 
		=  \sup_{\eta \in \cM_e(X, f)} \Big\{h_{\eta}(f)+\int\phi \,d\eta \Big\} \\
		& >  \pi_\phi(f,H^c)  \geqslant \sup_{\eta (H^c)=1} \Big\{h_{\eta}(f)+\int\phi \,d\eta  \Big\}.
\end{align*}
\end{proof}

Roughly speaking, the previous lemma guarantees that invariant probabilities 
with large pressure give zero weight to $H^c$ (although the set $H^c$ might be topologically large, e.g. dense).
The following is a direct consequence of the previous result together with a version of the Brin-Katok formula for random
dynamical systems. 

\begin{corollary}\label{cH1}
If $\mu\in \mathcal M(X,f)$ is given by Theorem~\ref{thm:Ruelle} then
$$\pi_\phi(f, H) \geqslant h_\mu(f)  + \int \phi \, d\mu \geqslant \int \log \lambda_\om \, d\mP.$$
\end{corollary}

\begin{proof}
By construction $\mu_\om\ll \nu_\om$ for $\mP$-a.e. $\om$ and, consequently, $\mu(H)=1$.
For $\mP$-a.e. $\om$ the density $h_\om$ is bounded away from zero and infinity. Let $0<c_\om<C_\om$ be so that
$c_\om\leqslant h_\om(\cdot) \leqslant C_\om$. 
This, together with the Brin-Katok formula for random dynamical systems \cite{Zhu} and the integrability of $K$ in Lemma~\ref{p.bdistort} implies
\begin{align*}
h_\mu(f) & = \int \lim_{\vep\to 0} \limsup_{n\to\infty} -\frac1n \log \mu_\om(B_\om(x,n,\vep)) \, d\mu(\om,x) \\
	& \geqslant  \int \lim_{\vep\to 0} \limsup_{n\to\infty} \Big[\frac1n \sum_{i=0}^{n-1} \log \la_{\te^i \om}  - \frac1n S_n \phi(\om,x) \Big]\, d\mu(\om,x)\\
	& = \int \log \lambda_\om \, d\mP - \int \phi \, d\mu
\end{align*}
(here we also used $\pi_*\mu=\mP$), and the lemma follows.
\end{proof}

In view of the previous corollary, the existence of equilibrium states for hyperbolic potentials 
will follow from the following:

\begin{proposition}\label{cH2}
$\pi_\phi(f) = \pi_\phi(f, H) =\int \log \lambda_\om \, d\mP.$
\end{proposition}

\begin{proof}
As $\pi_\phi(f) = \max\{ \pi_\phi(f, H), \pi_\phi(f, H^c)\}$ the first equality follows from the definition of hyperbolic potential.
Moreover, by Corollary~\ref{cH1}, we remain to prove that $\pi_\phi(f, H) \leqslant \int \log \lambda_\om \, d\mP.$

Throughout the proof fix $\om\in \Om$ and $\alpha> \int \log \lambda_\om \, d\mP$.
Recall that every $(\om,x)\in H$ has infinitely many hyperbolic times. In particular,
given $N \geqslant 1$ and $0<\vep \leqslant \vep_0$ %
$$
H_\om \subset
    \bigcup_{n \geqslant N} \bigcup_{x\in H_\om(n)} B_\om(x,n,\vep).
$$

We claim that there exists $D>0$ %
so that for every $n\geqslant N$ there exists a family $\cG_\om(n) \subset H_\om(n)$ so that
every point in $H_\om(n)$ is covered by at most $D$ dynamic balls $B_\om(x,n,\vep)$ with $x \in \cG_\om(n)$.
Indeed, each dynamic ball $B_\om(x,n,\vep)$ associated to a point $x\in H_\om(n)$ is mapped
diffeomorphically onto $B(f_\om^n(x),\vep) \subset X_{\te^n\om}$. 
Since $X_{\te^n\om}\subset M$, Besicovitch's covering lemma implies that there exists $D>0$
(depending on $m=\dim M$) and an at most countable family $\cG_\om(n) \subset
H_\om(n)$ such that every point of $f_\om^n(H_\om(n))\subset X_{\te^n\om}$ is contained in at most $D$
elements of the family $\{B(f_\om^n(x),\vep): x \in \cG_\om(n)\}$.
It follows that every point in $H_\om(n)$ is contained in at most $D$ dynamic balls $B_\om(x,n,\vep)$
with $x\in \cG_n$, proving the claim.

Now we need to bound the complexity of the set $H$ using the Carath\'eodory structures 
described in Appendix A. Take $\mathcal G_\om:=\cup_{n\geqslant N} \, \cG_\om(n) \times \{\om\} \subset X$. 
Given $\al\in \mathbb R$, $N\geqslant 1$ and $\vep>0$ small,  equation~\eqref{eq.alpham} together with the Gibbs property of $\nu$ at hyperbolic times imply that
\begin{align*}
m_\al(\om,f,\phi,\Lambda,\vep,N) 
	& \leqslant 
	\sum_{{(x,n)} \in \mathcal{G}_\om}
	  e^{-\alpha n(x)+ S_{n(x)} \phi_\om({\mathrm B_\om(x,n(x),\vep)})} \\
  	& =  \sum_{n\geqslant N}
		\sum_{x \in \mathcal{G}_\om(n)}
	  e^{-\alpha n+ S_{n} \phi_\om({\mathrm B_\om(x,n,\vep)})}  \\
  	& \leqslant  \sum_{n\geqslant N} 
		e^{K(\theta^n\om) \vep}  \,  e^{-(\alpha- \frac1n \sum_{i=0}^{n-1} \log \la_{\te^i \om}) n} \,
		\sum_{x \in \mathcal{G}_\om(n)}
	 \nu_\om(B_\om(x,n,\vep))\\
	 &\leqslant D \sum_{n\geqslant N} 
		e^{K(\theta^n\om) \vep}  \,  e^{-(\alpha- \frac1n \sum_{i=0}^{n-1} \log \la_{\te^i \om}) n}
\end{align*}
Since $\lim_{n\to\infty}\frac1n \sum_{i=0}^{n-1} \log \la_{\te^i \om} =  \int \log \lambda_\om \, d\mP<\al$
for $\mP$-a.e. $\om$ and $e^{K(\theta^n\om) \vep}$ has sub-exponential growth, the right-hand side above is summable. In particular we get 
$m_\al(\om,f,\phi,\Lambda,\vep,N)\to 0$ as $N \to \infty$ (independently of $\vep$). 
This shows that $\pi_\phi(f,H) \leqslant \int \log \la_\om \, d\mP$, completing the proof of the proposition.
\end{proof}

We can now deduce the main result of this section.

\begin{corollary}
Assume the random dynamical system $f=(f_\om)_\om$ satisfies assumptions (H0)-(H5) and that
the potential $\phi=(\phi_\om)_\om \in L_X^1(\Om, C^1(M))$ is a hyperbolic potential that satisfies (P). 
Then the probability $\mu$ is the 
unique equilibrium state for $f$ with respect to $\phi$. 

\end{corollary}

\begin{proof}
Corollary~\ref{cH1} and Proposition~\ref{cH2} yield that
\begin{equation}\label{eq:Tpressure}
\pi_\phi(f)=\pi_\phi(f,H) = \int \log \lambda_\om \, d\mP = h_\mu(f) +\int\phi \,d\mu
\end{equation}
(here we used item (3) in Proposition~\ref{prop:III}). This proves that $\mu$ is an equilibrium state for
$f$ with respect to $\phi$.

Now, Lemma~\ref{le:Hm}  ensures that any other equilibrium state  $\eta\in \cM(X,f)$ for $f$ with respect to $\phi$
satisfied $\eta(H)=1$. In particular $\eta$ has only positive Lyapunov exponents.
Then Theorem~\ref{thm.equilibrium-acim} in Appendix B implies that $\eta$ is absolutely continuous with respect
to $\nu$. Finally, since the densities $d\mu_\om/d\nu_\om$ are bounded away from zero and infinity for $\mP$-a.e. $\om$
we conclude that $\eta=\mu$, which proves the uniqueness of equilibrium states. 
\end{proof}


\vfill \pagebreak
\section{Appendix A: Relative pressure for random dynamical systems}\label{AppendixA}

In this appendix we introduce a Carath\'eodory structure for random dynamical systems which allow us 
to define topological pressure for invariant sets that are not necessarily compact on the fibers %
and establish a variational principle for this pressure function, which is of independent interest. 
This result extends the previous variational principle for random bundle transformations, 
where compactness is assumed, proved by Kifer~\cite{Ki01}.

\subsection*{Relative pressure in random dynamical systems}\label{s.relative.pressure}

Here we introduce the notion of relative pressure for random dynamical systems
inspired by the work of Pesin and Pitskel in the deterministic
setting. 
The Carath\'eodory structure introduced below are fiberwise versions of the ones 
present at the Appendix II of \cite{Pe97}.

\subsubsection*{Relative pressure using coverings}\label{ssec:relative-covering}
Let $\Lambda\subset X$ be a measurable set which is positively $F$-invariant (i.e., $F(\Lambda)\subset\Lambda$) and 
set $\Lambda_\om= \Lambda \cap X_\om$. Let $\mathcal{U}$ be a finite open cover of $M$ and set $\mathcal{U_{\om}}:=\{U\cap X_{\om}\ :\ U\in\mathcal{U}$\}. 
Denote by $\cI_\om^n$ the space of all $n$-strings $\underbar{i}(\om)=(U_0, U_{1}, \cdots, U_{n-1})$ with $U_{i}\in\mathcal{U}_{\theta^i\om}$ and set $n(\underbar{i})=n$. For a given string denote by $\underbar{U}(\underbar{i})=
\{x\in X_\om\ :\ f_{\om}^j(x)\in U_j\ \text{ for } j=0, \cdots, n-1\}$.
Furthermore, for every integer $N \geqslant 1$, let $\cS_ N \cU(\om)$
be the space of all cylinders of depth at least $N$. Given $\alpha
\in \R$ define
%
\begin{equation}\label{eq. alpha measure}
m_\al(\om, f, \phi,\Lambda,\cU,N)
	= \inf_{\mathcal{G}} \Big\{\sum_{{\mathrm{\underline U}} \in \mathcal{G}_\om}
 e^{-\alpha n({\mathrm{\underline U}})+ S_{n({\mathrm{\underline U}})} \phi_\om({\mathrm{\underline U}})} \Big\}
\end{equation}
where the infimum is taken over a class of families $\cG$ of the form  
$\cG_\omega \subset \cS_ N \cU(\om)$ that cover $\Lambda_\omega$ such that \eqref{eq. alpha measure} is measurable in $\omega$. In order to describe the class of families $\cG$ first recall that $\Omega$ is a Lebesque space which implies that there exists a refining sequence of finite partitions $(\mathfrak{p}_k)$ which generates $\mathcal{F}$ up to measure zero. In \eqref{eq. alpha measure}, we now take the infimum over all finite families of the following form: for any such 
$\mathcal{G}$ there exists $k\geq 1$
such that 
$\mathcal{G} \subset \{ V \times \{\underbar{i} \} : V \in \mathfrak{p}_k,
n(\underbar{i}) \geq N\}$, and $\mathcal{G}_\omega 
:= \{\underbar{U}(\underbar{i}) : \text{for some } V \times \{\underbar{i}\} \in \mathcal{G} \hbox{ s.t. } \omega \in V\}$ is an open cover of $\Lambda_\omega$
for $\mathbb P$-almost every $\omega\in \Omega$.
This construction is crucial in here as the cardinality  of these class is bounded by the cardinality of finite subsets of a countable set, hence  there are only countably many $\mathcal{G}$ of this type. In particular, $\Omega\ni \omega\mapsto m_\al(\om, f, \phi,\Lambda,\cU,N)$ is measurable as it is the infimum of measurable functions over a countable family. 

For fixed $\cG$, it then immediately follows that 
$(\om,\al)\mapsto \sum_{{\mathrm{\underline U}} \in \mathcal{G}_\om}
 e^{-\alpha n({\mathrm{\underline U}})+ S_{n({\mathrm{\underline U}})} \phi_\om({\mathrm{\underline U}})} $ 
 is measurable in $\Om$ 
 and continuous in $\al$, %
 a condition that ensures jointly measurability 
(see e.g. ~\cite[Lemma~1.1]{Cra02}). Hence, by countability of $\{\cG\}$,  the infimum $(\om,\al)\mapsto 
m_\al(\om, f, \phi,\Lambda,\cU,N)$ is jointly measurable and this guarantees the measurability of all the following quantities as a function of $\Om$. 
%
The limit
$$
m_\al (\om, f, \phi, \Lambda, \cU)=\lim_{N \to \infty} m_\al(\om, f,\phi,\Lambda,\cU,N)
$$
exists by monotonicity. Set also $\pi_\phi(\om, f,\Lambda,\cU)= \inf{\{\alpha : m_\al(\om, f,\phi,\Lambda,\cU)=0\}}$
and
%
$$
\pi_\phi(\om, f,\Lambda)=\lim_{\diam(\mathcal{U})\to0} \pi_\phi(\om, f,\Lambda,\cU).
$$
A simple adaptation of the proof \cite[Theorem 11.1]{Bo75} shows that the limit does exist and does not depend on the choice of coverings with diameter going to zero. We define the \emph{relative pressure} $\pi_\phi(f,\Lambda)$ by 
$$
\pi_\phi(f,\Lambda)=\int_{\Om}\pi_\phi(\om, f,\Lambda)\, d\mP.
$$
The \emph{relative entropy} $h_\La(f)$ is defined as the relative pressure $\pi_\phi(f,\Lambda)$ for $\phi\equiv 0$, and it can be related to topological pressure of a continuous potential $\phi$ by
\begin{equation}\label{eq:estimate-entropy}
\pi_\phi(f,\Lambda) \leqslant h_\Lambda(f) + \int \sup \phi_\om \;d\mathbb P.
\end{equation}
Indeed, let $\om\in \Om$ be a $\mathbb P$-typical point. Then, for every open covering  $\cU_\om$ of $X_\om$, any $\vep>0$ and any large $N \geqslant 1$ (depending on $\om$)
\begin{align*}
m_\al(\om, f,\phi,\Lambda,\cU,N)
     & = 
     \inf_{\mathcal{G}_\om} \Big\{\sum_{{\mathrm{\underline U}} \in \,\mathcal{G}_\om}
      e^{S_{n({\mathrm{\underline U}})}
      (-\alpha+\phi_\om({\mathrm{\underline U}}))
	} \Big\} \\
      & \leqslant 
     \inf_{\mathcal{G}_\om} \Big\{\sum_{{\mathrm{\underline U}} \in \,\mathcal{G}_\om}
      e^{S_{n({\mathrm{\underline U}})}
      (-\alpha+\sup \phi_\om)
       +S_{n({\mathrm{\underline U}})} k_{\cU,\om}} \Big\} 
\\
      & \leqslant 
     \inf_{\mathcal{G}_\om} \Big\{\sum_{{\mathrm{\underline U}} \in \,\mathcal{G}_\om}
      e^{- n({\mathrm{\underline U}}) [\alpha -\int \sup \phi_\om \,d\mP - \int k_{\cU,\om} \, d\mP -\vep]
 	} \Big\} 
      \\
     & = m_{\alpha -\int \sup \phi_\om \,d\mP - \int k_{\cU,\om} \, d\mP -\vep}\,(f,0,\Lambda,\cU,N)
\end{align*}
where $k_{\cU,\om}:=\sup\{|\varphi_{\om}(x)-\varphi_{\om}(y)| : y \in \cU_\om(x)\}$ tends to zero as $\diam\cU \to 0$
and the infimum is taken over all families $\mathcal{G}_\om \subset
\cS_N \cU(\om)$ that cover $\Lambda_\om$. 
Since $\vep>0$ was chosen arbitrary, taking the limit as the diameter of $\cU$ tends to zero we conclude that 
$\pi_\phi(f,\Lambda) \leqslant h_\Lambda(f) + \int \sup \phi_\om \;d\mathbb P$, as claimed.

\subsubsection*{Relative pressure using dynamic balls}\label{relative-db}

Here we give a dual definition of relative pressure using dynamic balls. The proof that it coincides with the 
notion of Subsection~\ref{ssec:relative-covering} can follow a standard route, and is left as an exercise to the reader.
Fix $\vep>0$ and $\om\in \Omega$ and consider the sets
$\cI_\om^n= X_\om\times\{n\}$ and $\cI_\om =X_\om \times \N$. For
every $\alpha \in \R$ and $N\geqslant 1$, define
\begin{equation}\label{eq.alpham}
m_\al(\om,f,\phi,\Lambda,\vep,N) 
	=  \inf_{\mathcal{G}_\om} \Big\{ 
	\sum_{{(x,n)} \in \mathcal{G}_\om}
  e^{-\alpha n(x)+ S_{n(x)} \phi_\om({\mathrm B_\om(x,n(x),\vep)})} \Big\}, 
\end{equation}
where the infimum is taken over all finite or countable families
$\mathcal{G}_\om\subset \cup_{n \geqslant N}\cI_\om^n$ such that the collection
of sets $\{B_\om(x,n,\vep): (x,n)\in \mathcal G_\om\}$ cover $\Lambda_\om$ and, as before, 
$
S_{n} \phi_\om({\mathrm B_\om(x,n,\vep)})
	=\sup\{ \sum_{j=0}^{n-1}\phi_{\theta^j(\om)} (f_\om^j(y)): y \in B_\om(x,n,\vep)\}.
$
Using that the previous sequence is increasing on $N$, the limit
\begin{equation}\label{eq.press1}
m_\alpha(\om, f,\phi,\Lambda, \vep)
  = \lim_{N\to+\infty}m_\al(\om,f,\phi,\Lambda,\vep,N) 
\end{equation}
does exist. If $\alpha$ is such that \eqref{eq.press1} is finite and $\beta>\al$ then   
$m_\beta(f,\phi,\Lambda,\vep,N)  \leqslant e^{-(\beta-\al)N} m_\al(f,\phi,\Lambda,\vep,N) $ tends to zero as $N\to\infty$, hence
$m_\beta(f,\phi,\Lambda, \vep)=0$. For that reason we define
$$
\pi_\phi(\om,f,\Lambda,\vep) 
	= \inf{\{\alpha: m_\al(\om,f,\phi,\Lambda, \vep)=0\}},
$$
which is decreasing on $\vep$. 
The (random) \emph{relative pressure} of the set $\La$ with respect to 
$(f,\phi)$ is defined by
\begin{equation}\label{eq.random.pressure}
 \pi_\phi(\om,f,\Lambda)
 	= \lim_{\vep \to 0} \pi_\phi(\om,f,\Lambda,\vep) .
\end{equation}

\begin{remark}\label{r.average}
In the case that $\Lambda$ is compact and invariant, it follows from the general results on Carath\'eodory structures
in \cite{Pe97} that the limit \eqref{eq.alpham} can be computed using dynamic balls of the same length as
$$
 \pi_\phi(\om,f,\Lambda)=\limsup_{n\to\infty}\frac1n 
 	\log 
	\Big[ 
	\inf_{\mathcal{G}_\om} \Big\{ \sum_{{(x,n)} \in \mathcal{G}_\om}  e^{S_{n} \phi_\om({\mathrm B_\om(x,n,\vep)})} \Big\} 
	\Big] %
$$
Under this assumption and ergodicity of the measure $\mathbb P$, Bogensch\"utz and Ochs \cite{BO} 
proved that 
$
 \pi_\phi(f,\Lambda)= \pi_\phi(\om,f,\Lambda)
$
for $\mathbb P$-almost every $\om$.
\end{remark}

\begin{proposition}\label{p.pressure.iteration}
Let $(X_\om)_\om$ be a family of compact sets,  $f=(f_\om)_\om$ be a family of continuous maps $f_\om : X_\om \to X_{\te \om}$,
 $\phi=(\phi_\om)_\om$ be a continuous potential and
$\Lambda\subset X$ be an $F$-invariant set. Then
$\pi_{S_\ell \phi}(f^\ell,\La)=\ell \pi_\phi(f,\La)$ for every $\ell \geqslant 1$, where
$f^\ell=(f_\om^\ell)_\om$.
\end{proposition}

\begin{proof}
Fix $\ell\geqslant 1$. By continuity of each $f_\om$ and compactness of the fibers $X_\om$, 
given any $\rho>0$
there exists $\vep(\om,\ell,\rho)>0$ (depend) such that $d(x,y) < \vep$ implies
$d(f_\om^j(x),f_\om^j(y))<\rho$ for every $0 \leqslant j < \ell$ and every $0<\vep \leqslant \vep(\om,\ell,\rho)$. 
It follows that
\begin{equation}\label{eq.balls}
B_{\om,f}(x,\ell n,\vep)
 \subset B_{\om, f^\ell}(x, n,\vep)
 \subset B_{\om,f}(x,\ell n, \rho)
\end{equation}
for each $0<\vep \leqslant \vep(\om,\ell,\rho)$ and $x\in X_\om$, where $B_{\om,g}(x,n,\vep)$ denotes the dynamic ball centered at $x$, of radius $\vep$ and size $n$ 
under the iteration of the map $g_{\te^{n-1} \om} \dots g_{\te \om} g_\om$.

\medskip
First we prove the $\geqslant$ inequality. Fix $\om\in \Om$. Given $N\geqslant 1$ and any family
$\cG_\om(\ell) \subset \cup_{n \geqslant N} \cI_n$ such that the balls
$B_{\om,f^\ell}(x,j,\vep)$ with $(x,j)\in\cG_\om(\ell)$ cover $\Lambda_\om$,
denote
$$
\cG_\om:=\{(x,j\ell) : (x,j)\in\cG_\om(\ell)\}.
$$
The second inclusion in \eqref{eq.balls} ensures that the dynamic balls
$B_{\om,f}(x,k,\rho)$ with $(x,k)\in\cG_\om$ cover $\Lambda_\om$. 
In particular
\begin{align*}
\sum_{(x,j)\in\cG_\om(\ell)}
 & e^{-\alpha \ell j +  \sum_{i=0}^{j-1} S_\ell\phi_\om(f_\om^{i\ell}(\,B_{\om,f^\ell}(x,j,\vep)\,))}\\
 & \qquad \geqslant \sum_{(x,k)\in\cG_\om} e^{-\alpha k + \sum_{i=0}^{k-1} \phi_{\te^i\om}(f_\om^i(\,B_{\om,f}(x,k,\rho)\,))
 	-  \sum_{i=0}^{k-1} \kappa_\rho(\te^i\om)}
\end{align*}
where $\kappa_\rho(\om)=\sup\{|\varphi_{\om}(x)-\varphi_{\om}(y)| : d(x,y)<\rho\}$. 
Ergodicity and Birkhoff's ergodic theorem ensures that $\frac1k\sum_{i=0}^{k-1} \kappa_\rho\circ \te^i$ is almost everywhere convergent to
the constant $\kappa_{\rho}^*=\int \kappa_\rho(\om)\,d\mP$. It is not hard to check that $\lim_{\rho \to 0}  \kappa^*_\rho=0$.
Hence, given $\zeta>0$ it holds that 
\begin{align*}
\sum_{(x,j)\in\cG_\om(\ell)}
 & e^{-\alpha \ell j +  \sum_{i=0}^{j-1} S_\ell\phi_\om(f_\om^{i\ell}(\,B_{\om,f^\ell}(x,j,\vep)\,))}\\
 & \qquad \geqslant \sum_{(x,k)\in\cG_\om} e^{-(\alpha+\kappa^*_\rho -\zeta) k + \sum_{i=0}^{k-1} \phi_{\te^i\om}(f_\om^i(\,B_{\om,f}(x,k,\rho)\,))}
\end{align*}
provided that $N\geqslant 1$ is large enough. In other words, as $\cG_\om$ is arbitrary %
this proves
$$
m_{\alpha\ell}(\om, f^\ell,S_\ell\phi,\Lambda,\vep,N)
 \geqslant m_{\alpha+\kappa^*_\rho -\zeta}(\om, f,\phi,\Lambda,\rho,N\ell),
$$
for every large $N\geqslant 1$. Since $\vep(\om,\ell,\rho) \to 0$ as $\rho\to 0$ and $\zeta$ was taken arbitrary the latter proves the desired inequality.

\medskip
For the $\leqslant$ inequality, we need to show the
relative pressure is not affected if one restricts the infimum
to families $\cG_\om$ of pairs
$(x,k)$ such that $k$ is always a multiple of $\ell$. More
precisely, let $m_\alpha^\ell(\om, f,\phi,\Lambda,\vep, N)$ be the
infimum over this subclass of families, and let
$m_\alpha^\ell(f,\phi,\Lambda,\vep)$ be its limit as $N\to\infty$.

\begin{lemma}\label{l.multiplosdeell}
We have $m_\alpha^\ell(\om, f,\phi,\Lambda,\vep)
 \leqslant m_{\alpha-\rho}(\om, f,\phi,\Lambda,\vep)$ for every $\rho>0$.
\end{lemma}

\begin{proof}
Fix $\rho>0$. It is enough to prove that
$m_\alpha^\ell(f,\phi,\Lambda,\vep,N) \leqslant
m_{\alpha-\rho}(f,\phi,\Lambda,\vep,N)$
for every large $N$. Let $N\geqslant 1$ be large so
that $N\rho > \ell(\alpha + \sup|\phi_\om|)$. Given any
$\cG_\om\subset\cup_{n \geqslant N} \cI_n$ such that the balls
$B_{\om,f}(x,k,\vep)$ with $(x,k)\in\cG_\om$ cover $\Lambda_\om$, define $\cG'_\om$ to
be the family of all $(x,k')$, $k'=\ell[k/\ell]$ such that
$(x,k)\in\cG_\om$. Notice that
$$
-\alpha k' + S_{k'}\phi_\om(x) \leqslant -\alpha k + \alpha \ell + S_k\phi_\om(x)
+ \ell \sup|\phi_\om| \leqslant (-\alpha+\rho) k + S_k\phi_\om(x)
$$
given that $k\geqslant N$. The lemma follows immediately.
\end{proof}

Let us resume the proof of the proposition. Take an arbitrary $0<\vep \leqslant \vep(\om,\ell,\rho)$.
Let $\cG'_\om$ be any family of pairs $(x,k)$ with $k\geqslant N\ell$ and such
that every $k$ is a multiple of $\ell$. Define $\cG_\om(\ell)$ to be the
family of pairs $(x,j)$ such that $(x,j\ell)\in\cG'_\om$. The first
inclusion in \eqref{eq.balls} ensures that if the balls
$B_{\om,f}(x,k,\vep)$ with $(x,k)\in\cG'_\om$ cover $\Lambda_\om$ then so do the
balls $B_{\om,f^\ell}(x,j,\vep)$ with $(x,j)\in\cG_\om(\ell)$. Then
\begin{align*}
\sum_{(x,k)\in\cG'_\om} & e^{-\alpha k + \sum_{i=0}^{k-1} \phi_{\te^i\om}(f_\om^i(\, B_{\om,f}(x,k,\vep)\,))} \\
	& \geqslant
\qquad \sum_{(x,j)\in\cG_\om(\ell)}
 e^{-\alpha \ell  j + \sum_{i=0}^{j-1} S_\ell\phi_\om(f_\om^{i\ell}(\, B_{\om,f^\ell}(x,j,\vep)\,))- \sum_{i=0}^{\ell j-1} \kappa_\rho(\te^i\om)}.
\end{align*}
Since $\cG_\om(\ell)$ is arbitrary, an argument identical to the previous one ensures that 
$$
m_\alpha^\ell(\om, f,\phi,\Lambda,\vep,N\ell)
 \geqslant m_{\alpha\ell}(\om,f^\ell,S_\ell\phi,\Lambda,\vep,N).
$$
Taking the limit when $N\to\infty$ and using
Lemma~\ref{l.multiplosdeell},
$$
m_{\alpha-\rho}(\om, f,\phi,\Lambda,\vep) \geqslant
m_\alpha^\ell(\om, f,\phi,\Lambda,\vep)
 \geqslant m_{\alpha\ell}(\om, f^\ell,S_\ell\phi,\Lambda,\vep).
$$
It follows that $\ell\big(\pi_\phi(\om, f,\Lambda,\vep) + \rho\big) \geqslant
\pi_{S_\ell\phi}(f^\ell, \Lambda,\vep)$. Since $\rho$ is arbitrary and $\vep(\om,\ell,\rho)$ tends to zero as $\rho \to 0$ we
conclude that 
$\pi_{S_\ell\phi}(f^\ell,\Lambda) \geqslant \ell \,\pi_\phi(f, \Lambda)$.
\end{proof}

\subsection*{A variational principle for relative pressure}

The main results of this appendix are Theorems~\ref{thmVP1} and ~\ref{thmVP2} which, combined, provide a variational principle
for the relative pressure of random dynamical systems.

\begin{theorem}\label{thmVP1}
Let $\phi\in L_{X}^1(\Om, C(M))$ and let $\Lambda\subset X$ be a measurable set so that $F(\Lambda)\subseteq \Lambda$. Then 
\[
\pi_\phi(f,\Lambda)\geqslant \sup_{\mu(\Lambda)=1} \Big\{h_{\mu}(f)+\int_\Lambda\phi \, d\mu\Big\}.
\]
\end{theorem}

\begin{proof}
Let $\mu\in \cM(X,f)$ with $\mu(\Lambda)=1$ and denote by $\mu=(\mu_\om)_{\om}$ the disintegration of $\mu$. 
Assume, without loss of generality, that $\mu$ is ergodic. 
In fact, Lemma 6.19 in \cite{Cra02} ensures that
$\zeta\in \cM(X,f)$ is an extreme point of $\cM(X,f)$ if and only if $\zeta$ is ergodic. This, together with Choquet representation theorem, provides an ergodic decomposition theorem for all measures in $\zeta\in \cM(X,f)$. Since $\mu(\Lambda)=1$, we can obtain the ergodic decomposition on some partition of $\Lambda$.
Let $\eta$ be the partition of $\Lambda$ which induces ergodic components $\Lambda_s$ for $s\in S$, where $S$ denotes some index set. Let us denote by $\mu_s$ the probability measure on $\Lambda_s$ and by $\nu$ the probability measure on the quotient space $\tilde \Lambda=\Lambda/\eta$. By convexity of the maps $\mu\mapsto h_\mu(f)$ and $\mu\mapsto \int\phi \,d\mu$,
\[
h_\mu(f)=\int_{\tilde \Lambda}h_{\mu_s}(f)d\nu(s),\ \ \ 
\int_{\Lambda}\phi \,d\mu=\int_{\tilde \Lambda}\Bigl(\int_{\Lambda_s}\phi d\mu_s\Bigr) \,d\nu(s).
\]
In consequence there is an ergodic component $\mu_s$ such that 
$h_{\mu_s}(f)+\int_{\Lambda_s}\phi \, d\mu_s\geqslant h_\mu(f)+\int_{\Lambda}\phi \, d\mu$, proving our claim.

Throughout, assume that $\mu \in \mathcal M(X,f)$ is $F$-invariant and ergodic.  
Denote by $\mu_{\Om}(\cdot)=\mu(\Om\times\cdot)$ the probability measure obtained as marginal of $\mu$ on $M$.
Since the measure $\mu_{\Om}$ is regular, we can verify the following lemma as in the deterministic case (see \cite[Lemma~1]{PP}). 

\begin{lemma}\label{le:auxi}
For every $\varepsilon>0$, there are a positive number $0<\delta<\varepsilon$, a Borel partition $\xi=\{C_1, \cdots, C_m\}$ of $M$, and a finite open covering $\cU=\{U_1,\cdots, U_k\}$ of M with $k\geqslant m$ such that 
\begin{itemize}
\item[(1)] $\diam(U_i), \diam(C_j)\leqslant \vep$ for $i=1,\cdots,k$ and $j=1,\cdots,m$.
\item[(2)] $\bar{U}_i\subset C_i$ for $i=1, \cdots, m$.
\item[(3)] $\mu_{\Om}(C_i\setminus U_i)<\delta$ for $i=1,\cdots,m$
\item[(4)] $\mu_{\Om}(\bigcup_{i=m+1}^{k}U_i)<\delta$.
\item[(5)] $2\delta \log m <\vep$.
\end{itemize} 
\end{lemma}

Now, fix $\vep>0$ and let $\delta>0$, the partition $\xi$ and the covering $\cU$ be given by Lemma~\ref{le:auxi}.
By Birkhoff's ergodic theorem, for $\mP$-a.e. $\om\in\Om$, there are $N_1(\om)\geqslant 1$ and $A_1(\om)\subset X_\om$ 
with $\mu_\om(A_1(\om))>1-\delta$ such that 
\begin{equation}\label{var1}
\frac{1}{n}\#\Big\{  0\leqslant l \leqslant n-1 \colon f_\om^l(y)\in \bigcup_{i=m+1}^{k}U_i \Big\}  <2\delta
\end{equation}
for every $x\in A_1(\om)$ and every $n\geqslant N_1(\om)$.

Given $\om\in \Om$, $n\geqslant 1$ and $x\in M$, consider the partition 
$$
\xi_{n}(\om)=\xi_\om\bigvee f_{\om}^{-1}\xi_{\theta\om}\bigvee\cdots\bigvee(f_{\om}^{n-1})^{-1}\xi_{\theta^{n-1}\om},
$$ 
and let $\xi_{n}(\om)(x)$ denote the element of the partition $\xi_{n}(\om)$
which contains $x$. 
By the Shannon-Macmillan-Breiman theorem for random dynamical systems (cf. Proposition 2.1 in \cite{Zhu}), for $\mP$-a.e. $\om\in\Om$, there are $N_2(\om)\geqslant 1$ and $A_2(\om)\subset X_\om$ with $\mu_\om(A_2(\om))>1-\delta$ so that 
\begin{equation}\label{var2}
\mu_\om(\xi_{n}(\om)(x)) \leqslant \exp\Bigl(-(h_\mu(f,\xi)-\delta)n\Bigr)
\end{equation}
for every $x\in A_2(\om)$  and $n\geqslant N_2(\om)$.
Using once more Birkhoff's ergodic theorem, for $\mP$-a.e. $\om\in\Om$ there are $N_3(\om)\geqslant 1$ and $A_3(\om)\subset X_\om$ with $\mu_\om(A_3(\om))>1-\delta$ such that if $x\in A_3(\om)$ then 
\begin{equation}\label{var3}
\Bigl|\frac{1}{n}\sum_{i=0}^{n-1}\phi_{\theta^n\om}(f_{\om}^i(x))-\int_\Lambda \phi \,d\mu\Bigr|<\delta
	\quad\text{and}\quad
	\Bigl|\frac{1}{n}\sum_{i=0}^{n-1}\kappa_{\vep}(\theta^i\omega)-\int_\Om\kappa_{\vep}^*(\om)\mP\Bigr|<\delta
\end{equation}
for every $n\geqslant N_3(\om)$, where $\kappa_\vep(\om)=\sup\{|\varphi_{\om}(x)-\varphi_{\om}(y)| : d(x,y)<\vep\}$
and the function $\kappa_{\vep}^*\in L^1(\mP)$ is $\theta$-invariant function and satisfies $\int_\Om\kappa_{\vep}^*(\om)d\mP=\int_\Om\kappa_{\vep}(\om)d\mP$.

\begin{remark}\label{rmk:oscillation}
We note that $0\leqslant\kappa_{\vep}\leqslant\kappa_{\vep'}$ whenever $0<\vep\leqslant\vep'$. This implies that $0\leqslant\kappa_{\vep}^*\leqslant\kappa_{\vep'}^*$ if $0<\vep\leqslant\vep'$. It is not hard to check that $\lim_{n\to\infty}\int_{\Om}\kappa_{\vep}^*(\om)d\mP=0$, since $\lim_{\vep\to0}\kappa_{\vep}(\om)=0$ for $\mP$-a.e. $\om$ and $\int_\Om\kappa_{\vep}^*(\om)d\mP=\int_\Om\kappa_{\vep}(\om) d\mP$.
\end{remark}

Now, if $N(\om)=\max_{1\leqslant i \leqslant 3}\{N_i(\om)\}$ and $A_\om=\bigcap_{i=1}^{3}A_i(\om)$, then $\mu_\om (A_\om)\geqslant 1-3\delta$. 
Moreover, taking $\alpha <h_\mu(f,\xi)+\int_{\Lambda}\phi \,d\mu-\int_\Om\kappa_{\vep}^* d\mP-3\delta-\vep$ and $N\geqslant N(\om)$,
by ~\eqref{eq. alpha measure} there exists a finite covering $\cG_\om\subset\cS_N \cU(\om)$ of $\Lambda_\om$ so that $n(\underbar{U})\geqslant N$ for $\underbar{U}\in\cG_\om$ and 
\[
\Bigl|\sum_{\underbar{U}\in\cG_\om}\exp\Bigl(-\al n(\underbar{U})+S_{n(\underbar{U})}\phi_\om(\underbar{U})\Bigr)-m_{\al}(f,\phi,\Lambda,\cU,\om)\Bigl|<\delta.
\]
Let $\cG_\om(l)$ be the subset of $\cG_\om$ each of whose elements $\underbar{U}$ satisfies $n(\underbar{U})=l$ and $\underbar{U}\cap A_\om \neq\emptyset$, and set $Y_\om(l)=\bigcup_{\underbar{U}\in\cG_\om}\underbar{U} \subset X_\om$.
A standard application of Shannon-Macmillan-Breiman's theorem, entirely analogous to Lemma 2 in  \cite[Appendix II]{Pe97}, guarantees that 
\begin{equation}\label{eq:count}
\sharp \cG_\om(l)\geqslant \mu_{\om}(Y_\om(l)\cap A_\om) \cdot \exp\Bigl((h_\mu(f,\xi)-\delta-2\delta\log \#\xi)l\Bigr)
\end{equation}
for every $l\geqslant N$.
Together with the estimates in $(\ref{var3})$, the fact that every $\un U\cap A_\om\neq\emptyset$ for every $\un U\in \cG_\om$ and that all elements in $\cU$ have diameter smaller or equal to $\vep$, this yields

\begin{equation*}
\begin{split}
m_\al(\om, f,\phi,\Lambda,\cU,N) +\delta 
	& >\sum_{\underbar{U}\in\cG_\om}\exp\Bigl(-\al n(\underbar{U})+S_{n(\underbar{U})}\phi_\om(\underbar{U})\Bigr)\\
	&= \sum_{l=N}^{\infty}\sum_{\underbar{U}\in\cG_\om(l)}\exp\Bigl(-\al l+S_{l}\phi_\om(\underbar{U}) \Bigr) \\
	&\geqslant \sum_{l=N}^{\infty}\sum_{\underbar{U}\in\cG_\om(l)}
	\exp\Bigl\{\Bigl(-\al+ \int_{\Lambda}\phi \, d\mu-\int_{\Om}\kappa_\vep^*(\om)d\mP-2\delta\Bigr)l\Bigr\}.
\end{split}
\end{equation*}
Using \eqref{eq:count}, item (5) in Lemma~\ref{le:auxi}, that $0<\de<\vep$ and the choice of $\alpha$ we get
\begin{equation*}
\begin{split}
m_\al & (\om, f,\phi,\Lambda,\cU,N) +\delta \\
	& \geqslant \sum_{l=N}^{\infty} 
	\mu_\om(Y_\om(l)\cap A_\om) \exp\Bigl\{\Bigl(-\al+ h_\mu(f,\xi)+\int_{\Lambda}\phi \,d\mu-\int_{\Om}\kappa_\vep^*(\om)d\mP-3\delta-\vep\Bigr)l\Bigr\} \\
&\geqslant\sum_{l=N}^{\infty}\mu_\om(Y_\om(l)\cap A_\om)\geqslant\mu_\om(A_\om)
	\geqslant1-3\delta.
\end{split}
\end{equation*}
Since $m_\al (\om, f,\phi,\Lambda,\cU,N) \geqslant1-4\delta$ for every large $N$ it follows that
$$
\pi_{\phi}(f,\Lambda,\cU,\om)\geqslant h_\mu(f,\xi)+\int_{\Lambda}\phi \,d\mu-\int_{\Om}\kappa_{\vep}^*(\om)d\mP-4\vep.
$$ 
Taking the limit as $\vep$ tends to zero (choosing appropriate coverings $\cU_\vep$ and recalling Remark~\ref{rmk:oscillation}) 
we get $\pi_{\phi}(\om, f,\Lambda)\geqslant h_\mu(f)+\int_{\Lambda}\phi \,d\mu$, which proves the theorem. 
\end{proof}

Observe that it may occur that  
$
\pi_\phi(f,\Lambda) > \sup_{\mu(\Lambda)=1} \big\{h_{\mu}(f)+\int_\Lambda\phi \, d\mu\big\}
$
(we refer the reader to \cite{Pe97} for such examples in the deterministic context of subshifts of finite type).
Hence, the converse inequality in Theorem~\ref{thmVP1} may fail.  Nevertheless, the previous partial variational principle
is sufficient to prove that topological pressure can be computed as the maximum of the relative pressure of an invariant set and its complement. More precisely:

\begin{proposition}\label{prop:III}
Assume that $\Lambda\subset  \Gamma$ are $F$-invariant measurable subsets of $X$. Then
the following properties hold:
\begin{itemize}
\item[(1)] $\pi_{\phi}(\om, f,\Lambda, \cU) \leqslant \pi_{\phi}(\om, f,\Gamma,\cU)$ for every finite covering $\cU$ of $M$, 
\item[(2)] 
$\pi_{\phi}( f,\Lambda) \leqslant \pi_{\phi}(f,\Gamma)$, 
\item[(3)] $\pi_\phi(f,X)=\max\{\pi_\phi(f,\Lambda), \pi_\phi(f,\Lambda^c) \}$. 
\end{itemize}
\end{proposition}

\begin{proof}
Assume that $\Lambda\subset\Gamma\subset X$ are measurable sets and $\cU$ is a  finite covering of $M$.
Item (1) is a direct consequence of the monotonicity condition of the exterior measures
$m_\al(\om, f, \phi,\Lambda,\cU,N) \leqslant m_\al(\om, f, \phi,\Gamma,\cU,N)$. Item (2) is a direct consequence of item (1).
Finally, on the one hand item (2) implies $\pi_\phi(f,X) \geqslant \max\{\pi_\phi(f,\Lambda), \pi_\phi(f,\Lambda^c) \}$.
On the other hand, since each $X_\om$ is compact, the variational principle in \eqref{e.pressure} together with Theorem~\ref{thmVP1}
ensures that
\begin{align*}
\pi_\phi(f)= & \sup_{\mu \in \cM(X, f)} \Big\{h_{\mu}(f)+\int\phi \,d\mu \Big\} 
		=  \sup_{\mu \in \cM_e(X, f)} \Big\{h_{\mu}(f)+\int\phi \,d\mu \Big\} \\
		& = \max\Big\{  \sup_{\mu (\Lambda)=1} \Big\{h_{\mu}(f)+\int\phi \,d\mu \Big\}, \; 
			\sup_{\mu (\Lambda^c)=1} \Big\{h_{\mu}(f)+\int\phi \,d\mu  \Big\} \Big\}\\
		& \leqslant \max\{\pi_\phi(f,\Lambda), \pi_\phi(f,\Lambda^c) \}.
\end{align*}
This proves item (3) and finishes the proof of the proposition.
\end{proof}

Finally, although not needed here, 
we finish this section with the following converse of Theorem~\ref{thmVP1} for product measures, 
yielding a variational principle for the relative entropy of random dynamical systems. 

Let $\Lambda\subset \Om \times X$ be a measurable set so that $F(\Lambda)\subseteq \Lambda$. 
For each $(\om,x)\in \Lambda$ and $n\geqslant 1$ consider the empirical measures 
$
\mu_n(\om,x)=\frac1n\sum_{j=0}^{n-1} \delta_{F^j (\om,x)},
$
and denote by $V(\om,x) \subset M(X,f)$ the set of $F$-invariant probability measures obtained
as weak$^*$ accumulation points of $(\mu_n(\om,x))_{n\geqslant 1}$. As these measures may not
be supported on $\Lambda$, consider the set 
$$
\cE(\Lambda)=\{(\om,x)\in\Lambda \colon V(\om,x) \cap \mathcal M(\Lambda,f)\neq \emptyset\}
$$
of points whose empirical measures have some accumulation measure supported on $\Lambda$.  
We have the following variational principle:

\begin{theorem}\label{thmVP2}
Let $\phi\in L_{X}^1(\Om, C(M))$ and let 
$\mu\in M(X,f)$ be a probability measure with $\mu(\La)=1$. 
Then
\[
\pi_{\phi}(f, \Lambda) 
	\geqslant \pi_{\phi}(f, \cE(\Lambda)) 
	= \sup\Big\{ h_\mu(f) + \int \phi\, d\mu \colon \mu(\Lambda)=1\Big\}.
\]
\end{theorem}

\begin{proof}
The inequalities
$
\pi_{\phi}(f, \Lambda) 
	\geqslant \pi_{\phi}(f, \cE(\Lambda)) 
	\geqslant \sup\Big\{ h_\mu(f) + \int \phi\, d\mu \colon \mu(\Lambda)=1\Big\}
$
follow immediately from Proposition ~\ref{prop:III} (2) and Theorem~\ref{thmVP1}. Thus we 
are left to prove that 
$
\pi_{\phi}(f, \cE(\Lambda)) 
	\leqslant \sup\Big\{ h_\mu(f) + \int \phi\, d\mu \colon \mu(\Lambda)=1\Big\}.
$
Actually, we will prove that there exists a $\mP$-full measure subset $\Om_0$ 
so that
\begin{align}\label{eq:claim}
\pi_{\phi}(f, \cE(\La)\cap(\Om_0\times M))\leqslant h_{\mu}(f)+\int_{\Om}\phi\ d\mP.
\end{align}
We note that, by the definition of topological pressure, $\pi_{\phi}(f,Y)=\int_{\Om}\pi_{\phi}(f, Y,\om) d\mP$. Hence, ~\eqref{eq:claim} is sufficient to prove the theorem as removing a set of the form $Z\times M$ from $Y$, where $Z$ is a zero set with respect to $\mP$, will not change its value. 

The strategy is a modification of the arguments in \cite[Theorem~2]{PP}. For the proof, we need two auxiliary results.
Let $E$ be a finite set and $\underbar{i}=(i_0, \dots, i_{k-1})\in E^k$. Denote by $\mu_{\ub{i}}$ the measure on $E$ given by

\begin{align*}
\mu_{\underbar{i}}(e)=\frac{1}{k}\#\{0\leqslant j \leqslant k-1\colon i_j=e\}.
\end{align*}
Put 
\begin{align*}
H(\mu_{\ub{i}})=-\sum_{e\in E}\mu_{\ub{i}}(e)\log\mu_{\ub{i}}(e)
\quad\text{and}\quad
R(k,h,E)=\{\ub{i}\in E^k:\ H(\mu_{\ub{i}}) \leqslant h\}.
\end{align*}
Some combinatorial arguments yield the following lemma:

\begin{lemma}\label{number of cover} \cite[Lemma~2.16]{Bo75}
We have
\begin{align*}
\limsup_{k\to\infty}\frac{1}{k}\log\# R(k,h,E)\leqslant h.
\end{align*}

\end{lemma}

Let $\cU=\{U_1,\dots, U_r\}$ be an open cover of $M$ with $\de:=\diam\cU>0$ and fix $\ep>0$. 
As before, for each $\om\in \Om$ and every string $\underbar{a}=(U_0, U_{1}, \cdots, U_{N-1})$ with $U_{i}\in\mathcal{U}$, we set $\underbar{U}(\underbar{a})=
\{x\in X_\om : f_{\om}^j(x)\in U_j\ \text{for}\, j=0, \cdots, N-1\}$ and 
$n(\ub{U}(\ub{a}))=N$ (we omit the dependence on $\om$ for notational simplicity),
and recall that $\kappa_\delta(\om)=\sup\{|\varphi_{\om}(x)-\varphi_{\om}(y)| : d(x,y)<\delta\}$.
Endow the space $(\Om,\cF,\mP)$ with the topology via the definition of the Lebesgue space.

\begin{proposition}\label{goodcover0}
There exists a full $\mP$-measure subset $\Om_1\subset\Om$ such that 
for each $(\om,x)\in \cE(\Lambda) \cap (\Om_1\times M$) and $\mu\in V(\om,x)\cap M(\Lambda,f)$ there exists an integer $m=m(\om,x,\mu)\geqslant1$ satisfying that: for every positive integer $n \geqslant 1$ there are $N\geqslant n$ and 
a string \ub{a} with $n(\ub{U}(\ub{a}))=N$ such that:
\begin{enumerate}
\item[(1)] $x\in\ub{U}(\ub{a})$,
\item[(2)] $\sup_{z\in\ub{U}(\ub{a})}\sum_{k=0}^{N-1}\phi_{\theta^k\om}(f_{\om}^k(z))
		\leqslant N\Bigl(\int_{\Om \times M} \phi \,d\mu+\frac{1}{N}\sum_{i=0}^{N-1}\kappa_\delta(\theta^i\om)
+\vep\Bigr)$
\item[(3)] the string $\ub{a}$ contains a substring $\ub{a}'$ with $n(\ub{U}(\ub{a}'))=km\geqslant N-m$ such that 
\begin{align*}
\frac{1}{m}H(\ub{a}')\leqslant h_{\mu}(f)+\vep.
\end{align*}
\end{enumerate}
\end{proposition}

\begin{proof}
Take a Borel partition $\zeta=\{C_1,\dots, C_r\}$ of $M$ such that $\overline{C_i}\subset U_i$
and consider the partition of $\Om\times M$ given by $\Om\times\zeta=\{\Omega\times C_1,\dots,\Om\times C_r\}$. 
Notice that the supremum in the definition of the random measure theoretic entropy 
$h_{\mu}(f)=\sup_{\xi}h_{\mu}(f,\xi)$ can be taken over a finite partition whose elements are of the form $\Om\times\xi$, where $\xi$ is a partition of $M$. This, together with \eqref{e.entropy}, 
ergodicity and invertibility of $(\theta,\mP)$ and Kingman's subadditive ergodic theorem, 
implies that there exists such a partition $\zeta$ satisfying 
$$
\lim_{n\to +\infty}\frac{1}{n}
H_{\mu_\om} \Big(\bigvee_{k=0}^{n-1}(f^k_\om)^{-1}\zeta \Big) = h_{\mu}(f,\zeta) 
	\leqslant h_\mu(f)+\frac\vep2
\qquad 
\mP-\text{a.e.} \, \om.
$$
(see e.g. proof of Proposition~2.2 in \cite{Zhu}).
 In particular, there is a measurable set $\Om_1\subset \Omega$ with $\mP(\Om_1)=1$ so that, for every $\om\in\Om_1$, there exists $m=m(\om)\geqslant1$ satisfying
\begin{align}\label{eq:part-Om1}
\frac{1}{m}H_{\mu_\om}\bigl(\zeta^{(m)}(\om)\bigr)
	\leqslant h_\mu(f)+\frac{\vep}{2}, 
\end{align}
where $\zeta^{(m)}(\om)=\bigvee_{i=0}^{m-1}(f_\om^i)^{-1}\zeta$.
Now, fix  $(\om,x)\in \Lambda \cap (\Om_1\times M$) and $\mu\in V(\om,x)\cap M(\Lambda,f)$. 
Take a positive integer sequence $n_i\to\infty$ such that 
$
\mu_{n_i}(\om,x)\to \mu
$
in the weak * topology.
We may assume without loss of generality that $n_j=mk_j$, where $m=m(\om)$ and $k_j\geqslant 1$.

Now, fix $\beta>0$ arbitrary. By regularity of $\mu$ (hence of $\mu_\Om(\cdot):=\mu(\Om \times\cdot))$, there exists a compact set 
$K_i\subset C_i$ with $\mu(\Omega\times (C_i\setminus K_i)) < \beta \mu(\Omega\times C_i)$.
Choosing an element $B_i$ of the open cover $\bigvee_{i=0}^{m-1}(f_\om^i)^{-1}\cU$ which contains $K_i$, one can find a Borel set $V_i^*$ 
such that 
$K_i\subset V_i^*\subset B_i$ and $\{V_i^*\}$ is a Borel partition of $M$. 
For each $1\leqslant i \leqslant r$, every $n_j=mk_j$ and $0\leqslant q\leqslant m-1$ we set
\begin{align*}
M_i^{(j)}=\# \{0\leqslant s \leqslant n_{j}-1\colon f_{\om}^s(x)\in V_{i}^*\}
\end{align*}
and
\begin{align*}
M_{i,q}^{(j)}=\#\{0\leqslant s\leqslant n_{j}-1\colon f_{\om}^s(x)\in 
V_{i}^*, s=q\ (\text{mod}\ m)\}.
\end{align*}
The latter describe the frequency of visits of the random orbit to a pre-determined 
sequence of Borel sets considered above.
Set also
$$
p_{i}^{(j)}=M_{(i_1,i_2, \dots, i_{n_j})}^{(j)}/n_j
\quad\text{and}\quad 
p_{i,q}^{(j)}=M_{(i_1,i_2, \dots, i_{n_j}),q}^{(j)}/k_j.
$$
As $\mu_{n_j}(\om,x)\to \mu$ as $j\to\infty$ we deduce that, taking continuous bump functions
supported on $V_i^*$,
 \begin{align*}
\liminf_{j\to\infty} p_{i}^{(j)} 
	& = 
	\liminf_{j\to\infty} \frac{1}{n_j}\sum_{k=0}^{n_j-1} \chi_{V_i^*} (F^k(\om,x))
	\geqslant\mu(\Om\times K_i) \\
	& 
	\geqslant (1-\beta) \,  \mu(\Om\times C_i).
\end{align*}
Similarly,
\begin{align*}
\limsup_{j\to\infty} p_{i}^{(j)}
\leqslant \int_{\Om}\mu_{\om}(K_i)d\mP +\sum_{i\neq j}\int_\Om \mu_{\om}(C_j\setminus K_j)
\leqslant \mu(\Om\times C_i)+r \beta.
\end{align*}
Therefore, taking $j$ sufficiently large and $\beta>0$ sufficiently small, 
we have
\begin{align*}
-\frac{1}{m}\sum_{i}p_i^{(j)}\log p_i^{(j)} \leqslant -\frac{1}{m}\sum_{i}\mu_{\Om}(C_i)\log \mu_{\Om}(C_i)+\frac{\vep}{2}\leqslant h_{\mu}(f)+\vep.
\end{align*}
The rest part of the proof of the first and the third statement is the same as in the proof of Lemma 2.15 in \cite{Bo75}.
By convexity of the function $c(x)=-x\log x$, we have
\begin{align*}
c(p_i^{(j)})\geqslant \frac{1}{m}\sum_{q=0}^{m-1} c(p_{i,q}^{(j)})
\end{align*}
and hence
$
\sum_{i=1}^{t}c(p_i^{(j)})\geqslant \frac{1}{m}\sum_{q=0}^{m-1} \sum_{i=1}^{t}c(p_{i,q}^{(j)}).
$
This ensures that there exists $0\leqslant q\leqslant m-1$ such that
\begin{align*}
\frac{1}{m}\sum_{i=1}^{t}c(p_{i,q}^{(j)})\leqslant \frac{1}{m}\sum_{i=1}^{t}c(p_i^{(j)})\leqslant h_\mu(f)+\vep.
\end{align*}

Set $N=n_j+q$. For $s<q$ we chose $U_s\in\cU$ so that $f_\om^s(x)\in U_s$. 
For every $V_{i}^{*}$, we take a string $\ub{$a$}_i$ so that 
$V^{*}_{i} \subset B_i=\ub{U}(\ub{$a$}_i)$. 
For $s\geqslant q$ we write $s=q+mp+e$ with $p\geqslant 0$ and $0\leqslant e<m$ and set $U_s=U_{e,i}$, where $i$ is chosen such that $f_\om^{q+mp}(x)\in V_i^*$. Set $\ub{$a$}_p=U_{0,i}\dots U_{m-1, i}$ and take the substring $\ub{$a$}$ as $U_0\dots U_{q-1}\ub{$a$}_0\cdots\ub{$a$}_{k_j-1}$. 
Then for $\ub{$a$}'=\ub{$a$}_0\dots\ub{$a$}_{k_j-1}$ the measure $\mu_{\ub{$a$}'}$ is given by the probabilities $p_{i,q}^{(j)}$ and it satisfies 
\begin{align*}
\frac{1}{m} H(\ub{$a$}')=\frac{1}{m} \sum_{i=0}^{t} c(p_{i,q}^{(j)})\leqslant h_{\mu}(f)+\vep,
\end{align*}
which ensures the first and the third statements.
The second statement follows from the weak$^*$ convergence 
$\mu_{n_j}(\om,x)\to\mu$ as $j\to\infty$ and
the definition of the function $\kappa_\delta(\om)$.

\end{proof}

We are now in a position to complete the proof of Theorem~\ref{thmVP2}, which was reduced to the proof of the claim in \eqref{eq:claim}. 
In what follows, take  
$$P:=\sup\Big\{ h_\mu(f) + \int \phi\, d\mu \colon \mu(\Lambda)=1\Big\}.$$
By ergodicity of $(\theta, \mP)$, we can take a measurable set $\Om_2\in\cF$ with $\mP (\Om_2)=1$ satisfying that for $\vep>0$ and $\om\in\Om_2$ there is an integer $n(\vep,\om)\geqslant 1$
such that

\begin{align*}
\Big|\frac{1}{n}\sum_{i=0}^{n-1}\kappa_{\delta}(\theta^{i} \om)-\int_\Om \kappa_{\delta}(\om)\ d\mP \Big|\leqslant \vep
\end{align*}
for $n\geqslant n(\vep,\om)$.  
Set also $\Om_0=\Om_1\cap\Om_2$ (recall ~\eqref{eq:part-Om1}). 

For each $m\geqslant 1$, denote by $G_m$ the set of all points {$(\om,x)\in \cE(\La) \cap(\Om_0\times M)$} for which  
there exists $\mu\in V(\om,x)\cap M(\Lambda,f)$ and the statement in Proposition~\ref{goodcover0} holds for the positive integer $m$. 
Moreover, taking $u\in\R$, put 
$$
G_{m,u} =\Big\{ (\om,x)\in G_m \colon \exists \mu\in V(\om,x)\cap M(\Lambda,f) \;\&\; \int_{\La}\phi\ d\mu\in [u-\vep,u+\vep] \Big\}.
$$
Observe that for each $(\om,x)\in G_{m,u}$ and probability $\mu\in V(\om,x)\cap M(\Lambda,f)$ satisfying 
$\int_{\La}\phi\ d\mu\in [u-\vep,u+\vep]$ one has that
$h_\mu(f) \le P - u + \vep$. Moreover, taking $c=\int_\Om |\phi_\om|_{\infty}\ d\mP$ and an $\vep$-dense subset $\{u_1,\dots,u_s\}$
of the interval $[-c,c]$ it follows, by construction, that 
\begin{equation}\label{eq:countable-G}
\cE(\La) \cap(\Om_0\times M)=\bigcup_{m\geqslant1}\bigcup_{i=1}^{s}G_{m,u_i}
\end{equation}
Since the quantity $\pi_\phi(\om,f,\cE(\La) \cap(\Om_0\times M))$ is generated by some Carath\'eodory structure then 
\begin{align}\label{eq:reduce-press}
\pi_\phi(\om,f,\cE(\La) \cap(\Om_0\times M),\cU)=\sup_{i,m} \pi_\phi(\om,f,G_{m,u_i},\cU)
\end{align}
(see e.g.~\cite{Pe97}).

We now fix $m\geqslant 1$ and $u\in \mathbb R$ and proceed to give an upper bound for the pressure of any non-empty set $G_{m,u}$. 
For each $N\geqslant 1$, denote by $\cG_{m,u}(\om,N)$ the set of all strings $\underline a'$ given in Proposition~\ref{goodcover0} for the pairs $(\om,x) \in G_{m,u}$ and $\mu\in V(\om,x) \cap M(\La, f)$ having length $km \in [N-m, N]$.
Item (3) in Proposition~\ref{goodcover0} ensures that
\begin{align*}
\#\cG_{m,u}(\om,N)
	& \leqslant (\#\cU)^m \; \# \{ \underline a' \in (\#\cU)^{km} \colon H(\underline a')
		 \leqslant m(h_\mu(f)+\vep) \} \\
	& \leqslant (\#\cU)^m \; \#R(km, m(h+\vep),\#\cU),
\end{align*}
where $h= P- u + \vep$. 
As $km\leqslant N < (k+1)m$, by Lemma \ref{number of cover} we obtain
\begin{align*}
\limsup_{N\to\infty}\frac1N \log \#\cG_{m,u}(\om,N) \leqslant h+\vep.
\end{align*}
Altogether we deduce that
\begin{align*}
m_\alpha (\om,f,\phi,G_{m,u},\cU ,N)
&\leqslant \sum_{k=N}^{\infty} \sum_{\underline U \in \cG_{m,u}(\om,k)} \# \cG_{m,u}(\om,k) \exp(-\alpha k +S_{k}\phi_\om(\ub{U})) \\
&\leqslant \sum_{k=N}^{\infty}\exp\Big(-\alpha k +k(h+2\vep)
	+k\big[u+2\vep+ \frac{1}{k}\sum_{i=0}^{k-1}\kappa_{\delta}(\te^i\om)\big]\Big) \\
&\leqslant 
\sum_{k=N}^{\infty}
	\exp\Big(-k\Big(\alpha-\Big[ P+ \int_\Om\kappa_{\delta}\ d\mP+5\vep
	\Big]\Big)\Big)
\end{align*}
for every sufficiently large $N$. Therefore, if 
$\alpha\geqslant P+\int_\Om\kappa_{\delta}(\om)\ d\mP+5\vep$, then $m_\alpha (\om,f,\phi,G_{m,u},\cU)=\lim_{N\to\infty}m_\alpha (\om,f,\phi,G_{m,u},\cU ,N)=0$, which implies that 
\begin{align*}
\pi_\phi(\om,f,G_{m,u},\mathcal U)
&\leqslant P+\int_\Om\kappa_{\delta}(\om)\ d\mP+5\vep.
\end{align*}
This, together with ~\eqref{eq:reduce-press}, ensures that 
\begin{align*}
\pi_\phi(\om,f,\cE(\La) \cap(\Om_0\times M),\cU)
	\leqslant P+\int_\Om\kappa_{\delta}(\om)\ d\mP+5\vep.
\end{align*}
Taking $\delta=\diam\ \cU\to 0$ and $\vep\to 0$ we prove that 
\eqref{eq:claim} holds, thus completing the proof of the theorem. 
\end{proof}

\vfill \pagebreak
\section{Appendix B: Ledrappier-Young formula for 
random dynamics}\label{AppendixB} 

The main goal of this Appendix is to present a mainly self-contained characterization of non-uniformly expanding equilibrium states (i.e. equilibrium states having only positive Lyapunov exponents) as those invariant probabilities which are absolutely continuous with
respect to an expanding conformal measure.   Let us first set the abstract framework.

\subsection*{Setting and main statement}

Let $(M,d)$ be a compact metric space and let $\cB$ be the Borel $\sigma$-algebra. 
Assume that $(\Om, \cF, \mP)$ is a Lebesgue space and that $\te: \Om \to \Om$ is an invertible $\mP$-preserving measurable  transformation and let $X\subset\Om\times M$ be a measurable compact subset. Consider a family $f=(f_\om)_\om$ of continuous maps $f_\om:X_\om\to X_{\te\om}$ and consider the skew-product 
$F: X \to X$ defined by $F(\om,x)=(\theta(\omega), f_\omega(x)).$
The main assumption is the existence of an expanding conformal measure. More precisely, assume there exists
a probability $\nu=(\nu_\om)_\om$ on $X$ such that $\cL_\om^*\nu_\om =\la_\om \nu_{\te\om}$ for $\mP$-a.e. $\om$,  
an $F$-invariant subset $H\subset X$ satisfying $\nu(H)=1$ and $\al>0$ so that
\begin{equation}\label{eq:HT-app}
\liminf_{n\to\infty} \frac1n \sum_{i=0}^{n-1} \log \|Df_{\theta^i\om}(f_\om^i(x))^{-1}\|^{-1} \geqslant  \al.
\quad \text{ for every $(\om,x)\in H$}.
\end{equation}

\begin{remark}\label{rmk-no-fibers}
Khanin and Kifer \cite{KhKi96,Ki08} considered a class of random average expanding maps where for $\mP$-a.e. $\om$
there exist a sequence $(n_k(\om))_k$ of instants at which all points in $X_\om$ have a hyperbolic time property. In particular, any
partition of small diameter (in the fibers) is generating for the random dynamics. This is useful as the Kolmogorov-Sinai theorem allows
to reduce the computation of the measure theoretical entropy to the entropy of a partition.  Under the weaker assumption \eqref{eq:HT-app}
it may occur that for every $\om\in \Om$ there are points with no expanding behavior, hence natural generating partitions may fail to exist.
\end{remark}

The previous remark justifies the need to construct suitable partitions adapted to the non-uniformly expanding nature of the random
dynamics, a main step to obtain a Rohklin formula for entropy hereafter. Most known results deal with the case of SRB measures
(i.e., absolutely continuous along the unstable manifolds) in which case Rohklin's formula is known as Pesin's formula. 
In the deterministic context of $C^2$-diffeomorphisms, the construction of suitable partitions and  Rohklin's formula for the entropy 
were achieved by Ledrappier and Young \cite{LY85}. 
An adaptation of their construction in the case of non-uniformly expanding maps appeared in \cite{VV10}.
In the random dynamical systems context there have been several contributions to the theory characterizing SRB measures as those
satisfying  Pesin's formula, justified by the need of dealing with invertible or noninvertible randomness and fiber dynamics (we provide more details after the statement of the main result below).
Unfortunately, most of these results seem not to 
adapt to our invertible base, non-invertible fiber and mild hyperbolicity assumptions, which reinforces the need of the following main result of this Appendix.

\begin{maintheorem}\label{thm.equilibrium-acim}
Let $(\Om,\mP)$ be a probability space, $f=(f_\om)_\om$ be a family of $C^1$-local diffeomorphisms 
$f_\om:X_\om\to X_{\theta(\om)}$, and $\phi=(\phi_\om)_\om \in L^1_X(\Om,C^\beta(M))$, for some $\beta>0$. Assume 
that $\nu=(\nu_\om)_\om$ is a conformal measure satisfying \eqref{eq:HT-app} 
 with 
Jacobians  $J_{\nu_\om} f_\om=\la_\om e^{-\phi_\om}$
for $\mP$-a.e $\om$. If, in addition, $\pi_\phi(f,X)=\int \log \la_\om\, d\mP$ then every equilibrium state $\eta$ for $f$ with respect
to $\phi$ such that $\eta(H\cap \supp\nu )=1$  
is absolutely continuous with respect to $\nu$, meaning that $\eta=(\eta_\om)_\om$ and $\eta_\om \ll \nu_\om$ for $\mP$-a.e. $\om$.
In particular, if the random dynamical system is topologically transitive on $\supp\nu$ then there exists at most one expanding equilibrium state for $f$ with respect $\phi$ on $\supp\nu$.
\end{maintheorem}

Some remarks are in order. In the special case that $\nu_\om$ denotes the Lebesgue measure on $X_\om$, the characterization 
of expanding equilibrium states as probabilities absolutely continuous with respect to the Lebesgue measure
corresponds to Pesin's formula for random dynamical systems and the family of potentials $\phi_\om=-\log |\det Df_\om|$, 
$\om\in \Om$.  
This topic has been extensively studied and several contributions were given which considered, separately,
the cases where the randomness is independent and identically distributed, both invertible and non-invertible map $\theta$ and both the case of $C^2$ diffeomorphisms 
or endomorphisms (cf. 
\cite{BL98, 
	LiuMZ, 
	LY98, 
	LQ, 
	LQZ} 
and references therein). In comparison to the deterministic context  \cite{Pe77}, it is worth mentioning 
that, 
in a random dynamical systems setting, the $C^2$-differentiability assumption may not be dropped to $C^{1+\alpha}$ 
in general (cf. Remark 2.4 in \cite{LQZ}). Our requirements on the $C^1$-smoothness of the maps in 
Theorem~\ref{thm.equilibrium-acim} are enough because we deal with the case that all Lyapunov exponents have positive 
sign (hence we need no absolutely continuous stable foliation nor to control angles between stable and unstable Pesin 
subbundles). Furthermore, it seems that Theorem~\ref{thm.equilibrium-acim} produces new results even in this 
context of SRB measures, as we could not find any such work dealing with invertible base and non-invertible and non-uniformly
expanding fiber dynamics. 

In the context of general equilibrium states, other versions of Theorem~\ref{thm.equilibrium-acim} for random 
dynamical systems appeared in 
\cite{
	BG95, 
	KhKi96, 
	Ki92, Ki08, %
	MSU08}, %
most of the times associated to random dynamical systems associated to $C^2$-expanding maps 
and making use of some Markov structure.   
Our approach is inspired by the arguments in \cite[Chapter VI]{LQ}, from which we borrow some
information concerning Lyapunov exponents and the theory of invariant manifolds.   

\subsection*{Natural extensions and lifts}

Some of the reasons to consider natural extensions are that the dynamics generating the random dynamics are not invertible
and that, as mentioned in Remark~\ref{rmk-no-fibers}, there may be points $(\om,x), (\om,y) \in X$ so that $f_\om(x)=f_\om(y)$
but which behave differently, meaning that the  refined partitions 
show a different behavior (in terms of contraction and non-contraction) 
in neighborhoods of $(\om,x)$ and $(\om,y)$. 

The \emph{natural extension} of the skew-product $F$ is the map $\tilde F : \tilde X  \to \tilde X$ 
defined on
$$
\tilde X
    = \Big\{ (\dots, (\om_{-1},x_{-1}), (\om_0,x_0)) \in X^\mathbb N : F(\om_i,x_i)=(\om_{i+1},x_{i+1}) , \;\forall i <
    0\Big\},
$$
and given as usual by 
\begin{equation}\label{eq:NE1}
\tilde F (\,(\dots,  (\om_{-1},x_{-1}), (\om_0,x_0))\,) = (\dots,  (\om_{-1},x_{-1}), (\om_0,x_0), F(\om_0,x_0))
\end{equation}
Notice that $\tilde X\subset X^\mathbb N \subset (\Om\times M)^{\mathbb N}$. Based on 
the invertibility of $\theta$, we may consider a simplification of the natural extension. Indeed, consider alternatively
the map
$$
\tilde F \colon \bigcup_{\om \in \Om} \{\om\} \times \tilde X_\om \to \{\om\} \times  \bigcup_{\om \in \Om} \tilde X_\om
$$
where
\begin{equation}\label{eq:NE2}
\tilde F (\om, \, (\dots, x_{-2}, x_{-1}, x_0)) = (\theta \om,\, (\dots, x_{-2}, x_{-1}, x_0, f_\om(x_0)))
\end{equation}
and
$$
\tilde X_\om = \Big\{ (\dots, x_{-2}, x_{-1}, x_0) \in \prod_{i=-\infty}^0  X_{\theta^i\om} \colon f_{\theta^{-i}\om}(x_{-i})=x_{-i+1}, \; \forall i\geqslant 0 \Big\}.
$$
The space $\tilde X_\om$ can be thought as the natural extension for a single realization of the random dynamics and, clearly,
$\tilde X:=\{\om\} \times  \bigcup_{\om \in \Om} \tilde X_\om \subset \Om \times M^{\mathbb N}$. Moreover, the maps defined by \eqref{eq:NE1}
and \eqref{eq:NE2} are measurable, invariant and make the following diagrams commute. 
\[
\begin{array}{cccc}
((\om_{-i},x_{-i}))_{i\geqslant 0}  & \xrightarrow[]{\quad\tilde F\quad} &  (\,\dots, (\om_{-1},x_{-1}), (\om_0,x_0), (\theta \om_0, f_{\om_0}(x_0))\,)  \medskip \\ 
\downarrow {\tilde \pi_0} & & \downarrow {\tilde \pi_0} \medskip \\
(\om_0,x_0)  & \xrightarrow[]{\quad F\quad} &  (\, \theta \om_0, f_{\om_0}(x_0)\,) \\ 
\end{array}
\]
and
\[
\begin{array}{cccc}
((\om_{-i},x_{-i}))_{i\geqslant 0}  & \xrightarrow[]{\quad\tilde F\quad} &  (\,\dots, (\om_{-1},x_{-1}), (\om_0,x_0), (\theta \om_0, f_{\om_0}(x_0))\,)  \medskip \\ 
\downarrow {\pi_0^1} & & \downarrow {\pi_0^1} \medskip \\
(\om_0, (x_{-i})_{i\geqslant 0})  & \xrightarrow[]{\quad\tilde F\quad} &  (\,\te{\om_0}\, (\dots, x_{-2}, x_{-1}, x_0, f_\om(x_0)) \,)  \medskip \\
\downarrow {\tilde \pi_0} & &  \downarrow {\tilde \pi_0} \medskip \\
(\om_0,x_0)  & \xrightarrow[]{\quad F\quad} &  (\, \theta \om_0, f_{\om_0}(x_0)\,) \\ 
\end{array}
\]
where $\pi_0^1 : (\Om\times M)^{\mathbb N} \to \Om \times M^{\mathbb N}$, 
$\tilde \pi_0 : \Om \times M^{\mathbb N} \to \Om \times M$ and $\tilde \pi_0 : (\Om \times M)^{\mathbb N} \to \Om \times M$ 
are the natural projections on the corresponding first coordinates.
Since $\pi_0^1$ is a bijection, the map $\tilde F$ seems a more suitable concept than the natural extension $\tilde F$. By a slight abuse of notation 
we shall refer to the latter when mentioning the natural extension of $F$.  For each $\om\in \Om$ we denote by 
$\tilde f_\om: \tilde X_\om \to \tilde X_{\te\om}$ the map 
$$
(\dots, x_{-2}, x_{-1}, x_0) \mapsto (\dots, x_{-2}, x_{-1}, x_0, f_\om(x_0)),
$$
where each fiber $\tilde X_\om$ is endowed with the metric $\tilde d(\un x,\un y)= \sum_{i \geqslant 0} 2^{-i}d(x_{-i},y_{-i})$ 
and the natural sigma-algebra inherited from $M^{\mathbb N}$ .

The non-uniform expanding structure of the fiber dynamics can now be lifted to the natural extension. This can be made both for the relevant sets as well as for the invariant probability measures. 

\begin{remark}\label{rmk:lifts}
It is well known that for every $F$-invariant probability $\eta$ there exists a unique $\tilde F$-invariant probability $\tilde\eta$ 
so that $(\tilde\pi_0)_*\tilde\eta =\eta$ and, moreover, $h_{\tilde\eta}(\tilde F)=h_\eta(F)$ (cf. \cite{Ro61}).
By construction, the probability $\tilde\eta := (\pi_0^1)_*\tilde\eta$ is $\tilde F$-invariant, $(\tilde \pi_0)_*\tilde \eta=\eta$ 
and $(\tilde \pi_\Om)_*\tilde \eta=\mathbb P$, where $\tilde \pi_\Om: \Om \times M^{\mathbb N} \to \Om$ is the natural projection on 
$\Om$. In particular, by the semiconjugacy it follows that $h_{\tilde\eta}(\tilde F)=h_\eta(F)$. 
Furthermore, if $\eta(H)=1$ then the $\tilde F$-invariant set $\tilde H=(\tilde\pi_0)^{-1}(H)$ has full $\tilde \eta$-measure.
\end{remark}

\begin{remark}\label{rmk:lift-nu}
Although the conformal measure $\nu=(\nu_\om)_\om$ is not necessarily invariant, one can also construct a natural family of
lifts associated to it.   
Indeed, if $(\om,x)\in H$ is an Oseledets typical point 
then the map 
$$
\tilde \pi_0 \mid_{\widetilde W^u_\text{loc}(\om, \un x)} :{\widetilde W^u_\text{loc}(\om, \un x)} \to W^u_\text{loc}(\om, x) 
$$
is a bijection and the measure $\nu_\om$ (restricted to any subset of $W^u_\text{loc}(\om, x)$)
can be lifted to $\tilde X$. This will be important later on in the proof that all expanding equilibrium states satisfy the
Rohklin property.

\end{remark}

\subsection*{Oseledets's theorem and local unstable manifolds}

The following is a version of the Oseledets theorem for random dynamical systems with not necessarily invertible $C^1$-fiber maps.
We refer the reader to \cite{Arn98b} and \cite[Chapter III]{LQ} for more details and proofs.

\begin{theorem}\label{thm-Oseledets}
Let $\te$ be an ergodic map on a probability space $(\Om, \mP)$, $M$ be a compact Riemannian manifold, 
$f=(f_\om)_\om$ be a family of $C^1$-maps
$f_\om:X_\om \subset M\to X_{\theta(\om)} \subset M$ and let $F$ be the natural skew-product. Suppose that $\eta$ is an $F$-invariant 
probability such that $\pi_*\eta=\mP$ and that $\int \log^+ \| Df_\om(x)\| \, d\eta(\om,x)<+\infty$.
There exists an $F$-invariant subset $\Sigma \subset X$ satisfying $\eta(\Sigma)=1$ and so that
for every $(\om,x)\in \Sigma$ there exists $1\leqslant k(x) \leqslant \dim M$, a filtration of linear subspaces of $T_x M$
$$
\{0\}=V^0_{\om,x} \subset V^1_{\om,x} \subset \dots \subset V^{k(x)}_{\om,x} = T_x M
$$
and numbers 
$\lambda^1(x) < \lambda^2(x) < \dots < \lambda^{k(x)}(x)$ (depending only on $x$) called Lyapunov exponents defined by
$$
\lim_{n\to\infty} \frac1n \log \|Df_\om^n(x) v\| =\la^i(x), \quad \forall v\in V^i_{\om,x} \setminus V^{i-1}_{\om,x}, \; 
\forall 1
\leqslant i \leqslant k(x). 
$$
Moreover, $Df_\om(x) V^i_{\om,x} = V^i_{\te\om,f_\om(x)}$ for every $(\om,x)\in \Sigma$ and 
the functions $k(x)$, $\lambda^i(x)$ and $V^i_{\om,x}$ vary measurably with $(\om,x)$, for every $1\leqslant i \leqslant k(x)$.
\end{theorem}

We will make use of the unstable manifold theorem for invariant probabilities having only positive Lyapunov exponents.

\begin{theorem}\label{p.Pesinpointwise}
Assume the hypothesis of Theorem~\ref{thm-Oseledets} and that  all Lyapunov exponents of $\eta$ are bounded below by $\la>0$.
For any $\vep>0$ small, for $\tilde \eta$-almost every $(\om,\un x)$ there are $\de_\vep(\om,\un x)>0$ and $\ga (\om,\un x) >0$ 
and an embedded $C^1$-disk $W^u_{\text{loc}}(\om,\un x) \subset  X_\om$, varying measurably with $(\om,\un x)$,
so that:
\begin{enumerate}
\item For every $y \in W^u_{\text{loc}}(\om,\un x)$ there is a unique $\un y \in \tilde X_\om$
    such that $\tilde \pi_0(\un y)=y_0$ and
    $
    d(x_{-n},y_{-n}) \leqslant \ga(\om,\un x) \, e^{-(\la-\vep) n}
    \; \forall n \geqslant 0;
    $
\item If a point $\un z \in \tilde X_\om$ satisfies
    \[ d(\un x,\un z) \leqslant \de_\vep(\om,\un x) \hbox{ and }
    d(x_{-n},z_{-n}) \leqslant \de_\vep(\om,\un x) e^{-(\la-\vep) n}
    \]
    for every $n \geqslant 0$
    then $z_0$ belongs to $W^u_{\text{loc}}(\om,\un x)$;
\item If $\tilde W^u_{\text{loc}}(\om,\un x)$ is the set of points $\un y \in \tilde X_\om$
    given by (2) then 
        $$
        d(y_{-n},z_{-n}) \leqslant \ga(\om,\un x) \, e^{-(\la-\vep) n} d(y,z)
        $$
    for every $\un y,\un z \in \tilde W^u_{\text{loc}}(\om,\un x)$ and every $n \geqslant 0$.
\end{enumerate}
\end{theorem}

Since local unstable leaves vary
measurably with the point, there are compact sets of arbitrary
large measure, referred to as \emph{hyperbolic blocks}, restricted to
which the local unstable leaves passing through those points vary
continuously as follows (see e.g. \cite[pp96--97]{LQ}).

\begin{corollary}\label{c.Pesinblocks}
There are countably many compact sets $(\tilde \La_i)_{i
\in \N}$ whose union 
is a $\tilde\eta$-full measure set such that the following holds.
For every $i \geqslant 1$ there are positive numbers $\vep_i\ll 1$,
$\la_i$, $r_i$, , $\ga_i$ and $R_i$ such that for every $(\om,\un x) \in
\tilde\La_i$ there exists an embedded submanifold $W^u_{\text{loc}}(\om,\un x) \subset X_\om\subset M$ of
dimension $\dim M$, and:
\begin{enumerate}
\item if $y_0 \in W^u_{\text{loc}}(\om,\un x)$ then there is a unique $\un y \in \tilde
      X_\om$ such that for every $n \geqslant 1$
        \begin{equation*}
        d(x_{-n},y_{-n}) \leqslant r_i e^{-\vep_i n} \quad\text{and}\quad
        d(x_{-n},y_{-n}) \leqslant \ga_i e^{-\la_i n};
        \end{equation*}
\item for every $0<r \leqslant r_i$ the set $W^u_{\text{loc}}(\om, \un y) \cap B(x_0,r)$ is
        connected and the map
        $$
        B((\om,\un x), \vep_i r) \cap \tilde\La_i\ni \un y \mapsto W^u_{\text{loc}}(\om,\un y) \cap B(x_0,r)
        $$
        is continuous (in the Hausdorff topology);
\item if $\un y$ and $\un z$ belong to $B((\om,\un x), \vep_i r) \cap \tilde\La_i$
        then either $W^u_{\text{loc}}(\om,\un y) \cap B(x_0,r)$ and $W^u_{\text{loc}}(\om,\un z) \cap B(x_0,r)$
        coincide or are disjoint; in the later case, if  $\un y \in {\tilde W}^u(\om, \un z)$ then
        $d(y_0,z_0)>2r_i$;
\item if $(\om_1, \un y) \in \tilde\La_i \cap B((\om,\un x), \vep_i r)$ then $W^u_{\text{loc}}(\om_1,\un y)$
        contains the ball of radius $R_i$ around $W^u_{\text{loc}}(\om_1,\un y) \cap
        B(x_0,r)$.
\end{enumerate}
\end{corollary}

\color{black}

\subsection*{Measurable generating partition}

We proceed with the construction of a special partition in $\tilde X$, 
adapted to an invariant and expanding measure, that is closely related with Ledrappier's geometric
construction in  \cite[Proposition~3.1]{Le84a}. In fact this exposition is inspired by a dual argument in 
\cite[Chapter~4,~$\mathsection$2]{LQ}  on the construction of measurable partitions adapted
to stable foliations.

While $\tilde X$ is not a manifold in general, its fibered structure 
makes reasonable to code inverse branches according to random orbits by $\te^{-1}$.   
Given a partition $\tilde\cQ$ denote by $\tilde\cQ(\om,\un  x)$ the element
of the partition containing $(\om,\un x) \in \tilde X$. We say that $\tilde Q$
is an \emph{increasing partition} if $(\tilde F^{-1}\tilde \cQ)(\om, \un x)
\subset \tilde\cQ(\om,\un x)$ for $\tilde\eta$-almost every $(\om, \un x)$, in
which case we write $\tilde F^{-1}\tilde \cQ\succ \tilde\cQ$.

All partitions considered throughout are fibered, meaning that
$\tilde Q$ is a refinement of the partition in fibers $\tilde \cF_0:=\{\{\om\} \times \tilde X_\om \colon \om\in \Om \}$. 
Making this implicit requirement simplifies the notation as it avoids using an extra conditional entropy term as in \cite{LQ}.
Assume, for the time being, the following two instrumental results.

\begin{proposition}\label{p.generating.partition}
Let $\eta$ be an $F$-invariant probability such that $\pi_*\eta=\mathbb P$ and $\eta(H)=1$.
There exists an $\tilde F$-invariant and full $\tilde \eta$-measure subset $\tilde S
\subset \tilde H$, and a measurable partition $\tilde\cQ$ of $\tilde S$
such that:
\begin{enumerate}
\item\label{ep1} $\tilde F^{-1}\tilde \cQ\succ \tilde\cQ$,
\item\label{ep2} $\bigvee_{j=0}^{+ \infty} \tilde F^{-j} \tilde\cQ$ is the
        partition of $\tilde S$ into points 
\item\label{ep4} The sigma-algebras $\cM_n$ generated by the partitions
        $\tilde F^{-n} \tilde\cQ$, $n\geqslant 1$, generate the
        $\si$-algebra in $\tilde S$, 
\item\label{ep3} For $\tilde \eta$-almost every $(\om, \un x)$ the element
        $\tilde\cQ(\om,\un x)\subset \tilde W^u(\om, \un x)$ contains a neighborhood of $\un x$ in
        $\tilde W^u(\om,\un x) \subset \tilde X_\om$ and the  projection $\pi(\tilde\cQ(\om,\un x))$ contains
        a neighborhood of $x_0$ in $X_\om$.
\end{enumerate}
\end{proposition}

Due to the non-expanding nature of the dynamics we cannot ensure in Proposition~\ref{p.generating.partition} 
that the refined partition $\bigvee_{j=0}^{+\infty} (\tilde F^j_\om)^{-1} \tilde Q_{\te^j\om}$ is a partition into points on 
$\{\om\} \times \tilde X_\om$ (compare  item (2) above). 

\begin{proposition}\label{p.entropygenerating}
$h_{\tilde\eta}(\tilde F)
    =H_{\tilde\eta}( \tilde F^{-1} \tilde\cQ \mid \tilde\cQ)$.
\end{proposition}

The previous propositions, whose proofs are postponed to the end of the appendix, justify the construction of the 
previous family of measurable partitions.

\subsection*{Rokhlin-type formula}

A main step in the proof of Theorem~\ref{thm.equilibrium-acim} is to prove that every expanding 
equilibrium state $\eta$ for $f$ with respect to $\phi$ satisfies 
a Rokhlin-like formula involving the Jacobian with respect to the conformal measure. 
\smallskip

Let $\eta$ be an equilibrium state for $f$ with respect to $\phi$ and assume, without loss of generality, that $\eta$
is ergodic. Let $\tilde \eta=(\tilde\eta_\om)_\om$ be the lift of $\eta$ to $\tilde X$, and denote by  
$(\tilde\eta_{\om,\un x})_{(\om, \un x)}$ be the disintegration of the measure $\tilde\eta_\om$ with respect to the 
partition $\tilde \cQ_\om$ on $\{\om\} \times \tilde X_\om$. 
We also consider the lift of the conformal measure $\nu$ which is adapted to $\tilde \eta$ which is constructed as follows. Since $\tilde \eta$-almost every $(\om,\un x)$
is Oseledets typical, we define the measure $\tilde \nu_{\om,\underline x}$ as the pull back 
$\big(\tilde \pi_0 \mid_{\widetilde W^u_\text{loc}(\om, \un x)}\big)^* \nu_\om$ (recall Remark~\ref{rmk:lift-nu}).
Namely, the lift $\tilde \nu = (\tilde \nu_{\om,\underline x})_{\om,\underline x}$ is defined by 
\begin{equation}\label{eq-nu-lift}
	\tilde \nu(\tilde E)
    = \int \tilde\nu_{\om,\underline x} (\tilde E) \, d\tilde\eta (\om,\un x)
\end{equation}
for every measurable subset $\tilde E \subset \tilde X$. 

\begin{remark}
By construction,  
$
\tilde\nu_{\om,\un x}(\tilde\cQ(\om,\un x))
    =\nu_\om\big( \pi(\tilde\cQ(\om,\un x))\cap W^u_{\text{loc}}(\om,\un x) \big)
$
for $\tilde\eta$-almost every $(\om,\un x)$.
Since $\tilde\eta$ is an expanding measure then $\widetilde W^u_{\text{loc}}(\om,\un x)$ forms an open neighborhood 
of $\un x \in \tilde X_\om$ and $\widetilde W^u_{\text{loc}}(\om,\un x) \cap \pi(\tilde\cQ(\om,\un x))$ contains a neighborhood of 
$x$ in $X_\om$. Using that $\eta(\supp \nu)=1$ we get $x\in \supp(\nu)$ 
and, we hereafter conclude that $0<\tilde\nu_{\om,\un x}(\tilde\cQ(\om,\un x))\leqslant 1$ for $\tilde\eta$-almost
every $(\om,\un x)$.
\end{remark}

We are in a position to show that the equilibrium state satisfies a second  Rokhlin-type formula.

\begin{lemma}\label{l.Pesin.extension}
The measure $\tilde\nu_{\om,\un x}$ has a Jacobian $J_{\tilde \nu_{\om,\un x}} \, \tilde F_\om = J_{\nu_\om}
f_\om \circ (\pi \circ \tilde \pi_0)$ with respect to $\tilde F_\om$, for $\mathbb P$-a.e. $\om$. In addition,
\begin{equation}\label{eq-Roh}
h_{\tilde\eta}(\tilde F) = \int \log J_{\tilde \nu_{\om,\un x}} \, \tilde F_\om \;d\tilde\eta(\om,\un x).
\end{equation}
Furthermore, for $\tilde\eta$-almost every $(\om,\un x)$ and every $\un y
\in \tilde\cQ(\om, \un x)$ the product
\begin{equation}\label{eq-densi}
\De_\om(\un x,\un y) = \prod_{j=1}^{\infty }
        \frac{J_{\tilde \nu_{\te^{-j}\om},\tilde F_\om^{-j}(\un x)} f_{\te^{-j}\om} (\tilde F_\om^{-j}(\un x))}
        {J_{\tilde \nu_{\te^{-j}\om},\tilde F_\om^{-j}(\un x)} f_{\te^{-j}\om} (\tilde F_\om^{-j}(\un y))}
\end{equation}
is positive and finite.
\end{lemma}

\begin{proof}
By definition $\tilde\nu_{\te\om,F_\om \un x}=(\tilde \pi_0 \mid_{W^u_{\text{loc}}(\te\om,F_\om(\un x))})^*(\nu_{\te\om}\mid_{W^u_{\text{loc}}(\te\om,f_\om(x))})$. Thus,
if $E_\om\subset X_\om$ is measurable, $f_\om \mid_{E_\om}$ is injective and $\tilde E_\om = \tilde \pi_0^{-1}(E_\om)$
is a cylinder then
\begin{align*}\label{eq.extension.conformal}
\tilde\nu_{\te\om, F_\om(\un x)} (\tilde F_\om (\tilde E))
    & = \tilde\nu_{\te\om, F_\om(\un x)} ( F_\om (\tilde E_\om \cap (F_\om^{-1} \tilde\cQ)(\om,\un x)) ) \\
    & = \nu_{\te\om} ( W^u_{\text{loc}}(\te\om,f_\om(x)) \cap f_\om (E_\om \cap \tilde \pi_0( (\tilde F^{-1} \tilde\cQ)(\om,\un x)))) \\
    & = \int_{E_\om \cap \tilde \pi_0( (\tilde F^{-1} \tilde\cQ)(\om,\un x))) \cap f_\om^{-1}(W^u_{\text{loc}}(\te\om,f_\om(x)))} 
    	\, J_{\nu_\om} f_\om \, d\nu_\om \\
   & = \int_{\tilde E_\om \cap (\tilde F^{-1} \tilde\cQ)(\om,\un x)} \, J_{\nu_\om} f_\om \circ (\pi\circ \tilde \pi_0) \, d\tilde \nu_{\om,\un x}.
\end{align*}
Since the sigma-algebra $\tilde\cB$ is the completion of the sigma-algebra generated by the cylinders then the first statement in the proposition holds. 
The Rokhlin formula ~\eqref{eq-Roh} follows, by simple algebraic manipulations, using that 
$h_{\eta}(f)  = h_{\tilde\eta}(\tilde F)$ and 
$$
\pi_\phi(f,X)=\int \log \la_\om \, d\mP(\om) = h_\eta(f) + \int \!\!\!\int \phi_\om(\cdot) \, d\eta_\om \, d\mP(\om). 
$$
Finally, the convergence of the product ~\eqref{eq-densi} follows by the 
H\"older continuity of the Jacobian $J_{\tilde \nu_{\om,\un x}} \, \tilde F_\om = \la_\om e^{-\phi_\om \circ (\pi \circ \tilde \pi_0)}$, 
the fact that for $\mP$-a.e. $\om$ the partition $\tilde\cQ(\om,\cdot)$ is subordinated to unstable
leaves, and the backward distance contraction for points in the same
unstable leaf. This completes the proof of the lemma.
\end{proof}

\subsection*{Absolutely continuous disintegration}

In this subsection we complete the proof of Theorem~\ref{thm.equilibrium-acim} using 
the main tools provided by Propositions ~\ref{p.generating.partition} and ~\ref{p.entropygenerating}. Namely,
we obtained that 
\begin{equation}\label{eq.gPesin.formula}
H_{\tilde\eta}( \tilde F^{-1} \tilde\cQ \mid \tilde\cQ)
        = \int \log J_{\tilde \nu_{\om,\un x}} \, \tilde F_\om \;d\tilde\eta(\om,\un x).
\end{equation}
The statement in the theorem is a direct consequence of the following:
\begin{proposition} For $\mathbb P$-a.e. $\om$ the following properties hold: 
\begin{enumerate}
\item $\tilde\eta_{\om,\un x}$ is absolutely continuous with
respect to $\tilde\nu_{\om,\un x}$, for $\tilde\eta$-almost every $(\om,\un x)$;
\item $\eta_\om$ is absolutely continuous with respect to $\nu_\om$.
\end{enumerate}
\end{proposition}

\begin{proof}

By construction, the product $\De_\om(\un x, \un y)$ (recall \eqref{eq-densi}) is bounded away from zero and infinity for $\mP$-a.e. $\om$. Thus, the term 
$
    Z(\om,\un x)=\int_{\tilde\cQ(\om,\un x)} \De_\om(\un x,\un y) \,d\tilde\nu_{\om, \un x}(\un y)
$
satisfies
$0<Z(\om,\un x)<\infty$ for $\tilde \eta$-almost every $(\om,\un x)\in \tilde X$.
Consider the probability $\tilde \zeta_{\om,\un x}$ given by
$$
\tilde \zeta_{\om,\un x} (B):= \frac{1}{Z(\om,\un x)} \int_{\tilde\cQ(\om,\un x) \cap B} \De_\om(\un x,\un y)
    \;d\tilde\nu_{\om,\un x}(\un y)
$$
for every measurable $B\subset \tilde X_\om$,
and we proceed to prove that $\tilde \eta_{\om,\un x} =\tilde \zeta_{\om,\un x}$.
First, the property $\tilde F^{-1} \tilde \cQ \succ \tilde\cQ$ ensures that 
\begin{align}
\nonumber 
& \tilde \zeta_{\om,\un x}((\tilde F^{-1}\tilde\cQ)(\om,\un x)) \\
    & = \frac{1}{Z(\om,\un x)} 
    	\int_{(\tilde F^{-1}\tilde\cQ)(\om,\un x)} \De_\om(\un x,\un y) \,d\tilde\nu_{\om,\un x}(\un y) \nonumber \\
    & = \frac{1}{Z(\om,\un x)} 
    	\int_{\tilde\cQ(\te\om,F_\om(\un x))} J_{\nu_\om} f_\om \circ (\pi\circ \tilde\pi_0 )(\om,\un y) \;\De_\om(\un x,\un y) \,d\tilde\nu_{\te\om,F_\om(\un x)}(\un y) \nonumber  \\
     & = \frac{Z(\te\om, F_\om(\un x))}{Z(\om,\un x)}  \cdot\frac1{\;J_{\tilde\nu_{\om,\un x}} \tilde F_\om(\un x)}.
     	\label{eq:last}
\end{align}
An argument identical to \cite[Lemma~VI.8.1]{LQ} ensures that for $\mP$-almost every $\om$ the map $\cQ(\om,\un x) \ni \un y \mapsto \De_\om(\un x,\un y)$ is H\"older continuous and uniformly bounded away from zero and infinity 
(by a constant depending only on $\om$). In particular
$$
Z(\om,\un x))= \lim_{n\to\infty} \int_{\tilde\cQ(\om,\un x)} 
\prod_{j=1}^{n}
        \frac{J_{\tilde \nu_{\te^{-j}\om},\tilde F_\om^{-j}(\un x)} f_{\te^{-j}\om} (\tilde F_\om^{-j}(\un x))}
        {J_{\tilde \nu_{\te^{-j}\om},\tilde F_\om^{-j}(\un x)} f_{\te^{-j}\om} (\tilde F_\om^{-j}(\un y))}
		\,d\tilde\nu_{\om, \un x}(\un y)
$$
and $\log^+ \frac{Z(\te\om, F_\om(\un x))}{Z(\om,\un x)} \leqslant \log^+ J_{\tilde \nu_{\om,\un x}} \tilde F_\om \in L^1(\tilde\eta)$.
Together with \eqref{eq:last}, this ensures that 

\begin{equation}\label{eq-Roh-2}
\int - \log \tilde \zeta_{\om,\un x}((\tilde F^{-1}\tilde\cQ)(\om,\un x)) \,d\tilde\eta(\om,\un x)
        = \int \log J_{\tilde\nu_{\om,\un x}} \tilde F_\om(\un x) \,d\tilde\eta(\om,\un x).
\end{equation}
Equations \eqref{eq.gPesin.formula} and \eqref{eq-Roh-2} together with
$$
H_{\tilde\eta}(\tilde F^{-1}\tilde\cQ \mid \tilde\cQ)
    =  \int - \log \tilde\eta_{\om,\un x}((\tilde F^{-1}\tilde\cQ)(\om,\un x))
    \;d\tilde\eta(\om,\un x)
$$
imply that
$$
\int - \log \tilde \zeta_{\om,\un x}((\tilde F^{-1}\tilde\cQ)(\om,\un x)) \,d\tilde\eta(\om,\un x)
	= \int - \log \tilde\eta_{\om,\un x}((\tilde F^{-1}\tilde\cQ)(\om,\un x))  \;d\tilde\eta(\om,\un x).
$$
This ensures
$$
0=\int  \log \frac{\tilde \zeta_{\om,\un x}((\tilde F^{-1}\tilde\cQ)(\om,\un x))}
	{\tilde\eta_{\om,\un x}((\tilde F^{-1}\tilde\cQ)(\om,\un x))}  \;d\tilde\eta(\om,\un x)
	\leqslant 
	\log \int   \frac{\tilde \zeta_{\om,\un x}((\tilde F^{-1}\tilde\cQ)(\om,\un x))}
	{\tilde\eta_{\om,\un x}((\tilde F^{-1}\tilde\cQ)(\om,\un x))}  \;d\tilde\eta(\om,\un x) =0
$$
which together with the strict convexity of the logarithm ensures that 
the measures $\tilde \eta_{\om,\un x}$ and $\tilde \zeta_{\om,\un x}$
coincide on the sigma-algebra generated by $\tilde F^{-1}\tilde \cQ$.
Replacing $\tilde F$ by any
power $\tilde F^n$ in the previous computations it is not difficult to
check that $\tilde\eta_{\om,\un x}$ and $\tilde\zeta_{\om,\un x}$ coincide in
the increasing family of sigma-algebras generated by 
$(\tilde F^{-n}(\tilde\cQ))_{n \geqslant 1}$, proving Item (1) in the proposition.

Now, just observe that Item (1) and the definition of the lifted measures $\tilde\nu_{\om,\un x}$
implies $(\tilde\pi_0)_*\tilde\eta_{\om, \un x} \ll \nu$ for
$\tilde\eta$-almost every $\un x$. Since $(\tilde\eta_{\om,\un x})$ is a
disintegration of $\tilde\eta$ and $(\tilde \pi_0)_*\tilde\eta=\eta$ then $\eta\ll\nu$, proving Item (2). 
\end{proof}

\subsection*{Main estimates}
The present subsection is devoted to the proof of the two main propositions used in the proof of Theorem~\ref{thm.equilibrium-acim}.
We combine modification of arguments in \cite[pp 96--103]{LQ}. 
While the latter considers random positive iterations to consider partitions subordinated to stable manifolds,
we need to consider inverse interations to capture backward contraction of unstable manifolds. Moreover, we avoid the use of stationary measures.  
 
\subsubsection*{Proof of Proposition~\ref{p.generating.partition}}

Take an $F$-invariant probability  $\eta$ such that $\pi_*\eta=\mathbb P$ and $\eta(H)=1$
and let $\tilde\eta$ be the unique $\tilde F$-invariant probability projecting on it.
We proceed to construct a $\tilde \eta$-almost everywhere generating partition $\tilde Q$.

Since $\tilde\eta$ is an expanding measure,
Proposition~\ref{p.Pesinpointwise} guarantees the existence of local
unstable manifolds at $\tilde\eta$-almost every point. 
Take $i \geqslant 1$ such that the Pesin block $\tilde \Lambda_i$ satisfies $\tilde\eta(\tilde\La_i)>0$, 
and let $r_i$, $\vep_i$, $\ga_i$ and $R_i$ be given by
Corollary~\ref{c.Pesinblocks}. 

Fix $0<r\leqslant r_i$ and $(\om_0, \un x)
\in \supp(\tilde\eta\mid_{\tilde\La_i})$. By construction $W^u_{\text{loc}}(\om,\un y)  \cap B(x_0,r)$ is connected and the map 
$(\om,\un y) \mapsto W^u_{\text{loc}}(\om,\un y)  \cap B(x_0,r)$ is a continuous function on $B((\om_0,\un x), \vep_i r) \cap \tilde\La_i$.
Consider the sets
$$
\tilde V_\om(\un y,r)=\{\un z \in W^u_{\text{loc}}(\om,\un y) : z_0 \in B(x_0,r)\}	
	\subset \tilde X_\om
$$
defined for any $(\om,\un y) \in B((\om_0,\un x), \vep_i r) \cap \tilde\La_i$, and consider the
subset of $\tilde X$ given by
$$
\tilde S_{\om_0} (\un x,r)
    =\bigcup \Big\{\{\om\} \times \tilde V_\om (\un y, r) : (\om,\un y) \in B((\om_0,\un x),\vep_i r) \cap \tilde\La_i\Big\}.
$$

Consider an initial partition defined as follows.  
Take the partition $\tilde\cQ_0(\om)$ of $\tilde X_\om$ (depending on $r$)
whose elements are the
connected components $\tilde V_\om(\un y,r)$ of the unstable manifolds and their complement 
$\tilde X_\om \setminus \tilde S_\om (\tilde
x,r)$ for every $\om\in \Om$ satisfying $(\{\om\} \times \tilde X_\om) \cap \tilde S_{\om_0} (\un x,r) \neq \emptyset$,
and take $\tilde\cQ_0(\om)=\tilde X_\om$ otherwise.

Now, the partition $\tilde Q$ is obtained by the refinements by the dynamics, exploring the
ergodicity of $\tilde \eta$. Indeed, on the one hand Poincar\'e's recurrence theorem implies that 
$$
\tilde S = \tilde S_r := \bigcap_{j=0}^{+\infty} \bigcup_{n=j}^{+\infty} \tilde F^n(\tilde S_{\om_0}(\un x,r)) 
$$
is a $\tilde\eta$-full measure set. On the other hand, taking the partition
$$
\tilde Q(\om) = \bigvee_{n=0}^{+\infty} \tilde F_{\te^{-n}\om}^n(\tilde Q_0(\te^{-n}\om))
		= \bigvee_{n=0}^{+\infty}  f_{\te^{-n}\om}^n(\tilde Q_0(\te^{-n}\om)),
$$
on $\tilde X_\om$ and the partition $\tilde Q$ formed by the previous ones on each fiber, by construction one has that 
 $\tilde F^{-1}\tilde \cQ\succ \tilde\cQ$.
Indeed, 
\begin{align*}
\tilde F^{-1}(\{\om\} \times \tilde Q(\om)) 
	&= \{\te^{-1}\om\} \times f_{\te^{-1}\om}^{-1} 
	\Big( \bigvee_{n=0}^{+\infty} f_{\te^{-n}\om}^n(\tilde Q_0(\te^{-n}\om)) \Big) \\
	&= \{\te^{-1}\om\} \times \bigvee_{n=-1}^{+\infty} f_{\te^{-n}(\te^{-1}\om)}^{n}(\tilde Q_0(\te^{-n}(\te^{-1}\om))) \\
	& \succ \{\te^{-1}\om\} \times \tilde Q(\te^{-1}\om).
\end{align*}
Hence Item (1) holds.
Moreover, since $\tilde\eta$-almost every $(\om,\un y)$ belongs to $\tilde S$,
and these points have infinitely many returns to the set $\tilde S_{\om_0} (\un x,r)$ under iteration by $\tilde F$,
and $\tilde S_{\om_0} (\un x,r)$ is formed by pieces of unstable manifolds, the backward contraction 
along unstable leaves guarantee not only that the diameter of the partition $\bigvee_{j=0}^{n}
\tilde F^{-j}\tilde\cQ$ tend to zero as $n \to\infty$, as the sigma-algebras $\cM_n$ generated by the partitions
        $\tilde F^{-n} \tilde\cQ$, $n\geqslant 1$, generate the
        $\si$-algebra in $\tilde S$. This proves Items (2) and (3).

In what follows we prove that, diminishing $r$ if needed, the resulting partition $\tilde Q=\tilde Q_r$ 
satisfies Item (4): 
for $\tilde \eta$-almost every $(\om, \un y)$ the element
        $\tilde\cQ(\om)(\un y)\subset \tilde W^u(\om, \un y)$ contains a neighborhood of $\un y$ in
        $\tilde W^u(\om,\un y) \subset \tilde X_\om$ and the  projection $\pi(\tilde\cQ(\om,\un y))$ contains
        a neighborhood of $y_0$ in $X_\om$.
We proceed to show that the partition $\tilde\cQ(r)$ satisfies
\eqref{ep3} for Lebesgue almost every parameter $r$. Given  $0<r\leqslant
r_i$ and $(\om,\un y) \in \tilde S_r$ define
$$
\beta_r(\om,\un y)
    =\inf_{n \geqslant 0}\Big\{R_i, \frac{r}{\ga_i}, \frac{1}{2\ga_i} e^{\la_i n} d(y_{-n}, \partial
    B(x_0,r))\Big\},
$$
that it clearly non-negative. First we observe the following:
\begin{itemize}
\item[(a)] If $(\om,\un y) \in \tilde S_{\om_0}(x_0,r)$, $z_0 \in  W^u(\om,\un y) \subset X_\om$ 
	and $d(y_0,z_0)<\beta_r(\om,\un y)$ then there exists $\un z \in \tilde\cQ(\om,\un y)$ such that
            $\pi(\un z)=z_0$;
\item[(b)] There exists a full Lebesgue measure set of parameters $0<r\leqslant r_i$
            such that the function $\beta_r(\cdot)$ is
            strictly positive almost everywhere and $\tilde\eta(\partial \tilde\cQ_r)=0$.
\end{itemize}
Indeed, any point $(\om,\un y) \in \tilde S_{\om_0}(\un x,r)$ 
belongs to the local unstable manifold of some element  
$(\om,\un t) \in B((\om_0,\un x),\vep_i r) \cap \tilde \La_i$. 
If $z_0 \in W^u(\om,\un y)$ 
and $d(y_0,z_0)<\beta_r(\om,\un y)< R_i$
then there exists $\un z \in \tilde W^u(\om,\un y)$
 such that $\pi(\un z)=z_0$ (cf. Corollary~\ref{c.Pesinblocks}). In particular
$$
d(y_{-n},z_{-n}) \leqslant \ga_i e^{-n \la_i} d(y_0,z_0), \quad \;\forall n \in
\N
$$
ensuring that $d(y_{-n},z_{-n}) \leqslant r$ and $d(y_{-n},z_{-n})
\leqslant 1/2 \,d(y_{-n},\partial B(x_0,r))$ for every $n \in \N$.
Altogether, this ensures that the iterates $\tilde F_\om^{-n}(\un y)$
and $\tilde F_\om^{-n}(\un z)$   
belong to the same element of the partition $\tilde\cQ_0$ for every $n \geqslant 1$, 
and conclude that  $\un z \in \tilde\cQ(\om,\un y)$. 
This proves Item (a).

The argument in the proof of (b) uses the
following remark from measure theory (see e.g. \cite[Chapter 4, Lemma ~2.1]{LQ}):
if $r_0>0$, $\vartheta$ is a Borel measure in $[0,r_0]$ and $0<a<1$
then Lebesgue almost every $r \in[0,r_0]$ satisfies
\begin{equation}\label{eq.sum}
\sum_{k=0}^{\infty} \vartheta \big( [r-a^k, r+a^k] \big)< \infty.
\end{equation}
Let $\tilde \eta_M$ denote the marginal of $\tilde \eta$ on $M^{\mathbb N}$.
If $\vartheta$ is the measure on the
interval $[r_i/2,r_i]$ given by 
$$
\vartheta( E ) = \tilde \eta_M \big(\, \un y \in M^{\mathbb N}: d(\pi_0(\un y), x_0) \in E \,\big)
$$
the previous assertion
guarantees that 
\begin{equation}\label{eq.sum2}
\sum_{k=0}^{\infty} \tilde \eta_M \big(\,  \un y \in M^{\mathbb N} : |d(x_0,\pi_0(\un y))-r|< e^{-\la_i k} \,\big) < \infty
\end{equation}
for Lebesgue almost every $r \in[r_i/2,r_i]$.
On the other hand, there exists $D>0$ such that
$|d(z_0,x_0)-r|<D\tau$ whenever $d(z_0,\partial B(x_0,r))<\tau$ and
$0<\tau<r\leqslant r_i$. Therefore, by $\tilde F$-invariance of $\tilde \eta$ and \eqref{eq.sum2},
\begin{align*}
\sum_{k=0}^{\infty}  \tilde\eta \Big(\, (\om,\un y) & \in \tilde X \colon  d(y_{-n}, \partial B(x_0,r)) < D^{-1}e^{-\la_i k} \,\Big) \\
	& = \sum_{k=0}^{\infty} \int_\Om \tilde\eta_\om \Big(\, \un y  \in \tilde X_\om 
				\colon  d(y_{-n}, \partial B(x_0,r)) < D^{-1}e^{-\la_i k} \,\Big) \, d\mathbb P(\om)\\
	& = \sum_{k=0}^{\infty} \int_\Om (f_\om^n)_* \tilde\eta_\om \Big(\, \un y  \in \tilde X_{\te^n\om} 
				\colon  d(\pi_0(\un y), \partial B(x_0,r)) < D^{-1}e^{-\la_i k} \,\Big) \, d\mathbb P(\om)\\
	& \leqslant  \sum_{k=0}^{\infty}  \tilde\eta_M \Big(\, \un y  \in M^{\mathbb N} 
				\colon  d(\pi_0(\un y), \partial B(x_0,r)) < D^{-1}e^{-\la_i k} \,\Big)<\infty.
\end{align*}
Then the Borel-Cantelli lemma assures that for $\tilde\eta$-almost every $(\om,\un y)\in \tilde X$ the condition
$
|d(y_{-n},\partial B(x_0,r))| < D^{-1}e^{-\la_i k}
$
holds for at most finitely many positive integers $k$, proving that
$\beta_r(\om, \un y)>0$. 
Finally, since $\tilde\eta(\cup_{n\geqslant 0} \tilde F^n
(\Om\times \partial B(x_0,r)^{\N}))=0$ for all but a countable set of parameters
$0<r \leqslant r_i$ then $\tilde\cQ(\om,\un y)=\tilde\cQ_r(\om,\un y)$ contains a neighborhood of
$\un y$ in $\tilde W^u(\om,\un y)$ for $\tilde\eta$-almost every $(\om,\un y)\in \tilde X$. 
This proves claim (b) above and finishes the proof of the proposition.
\hfill $\square$

\medskip

\subsubsection*{Proof of Proposition~\ref{p.entropygenerating}}

Given $i \geqslant 1$, let $\tilde\La_i$ and $r_i$ be as 
in the proof of Proposition~\ref{p.generating.partition}. 
The proof involves a preliminary lemma, adapted from \cite{Man81}, which ensures the construction of measurable 
and finite entropy partitions with arbitrarily small diameter.

\begin{lemma}\label{l.D.ep} \cite[Chapter VI, Lemma 5.2]{LQ}
If $\De: \tilde X \to (0,1)$ is a measurable and log-integrable function with respect to $\tilde \eta$
then there exists a measurable partition $\cP_0$
of $\tilde X$ such that $\cP_0\succ \tilde \cF_0$ and $\diam \cP_0(\om, \un x) \leqslant \De(\om,\un x)$ for $\tilde\eta$-a.e. 
$(\om,\un x) \in \tilde X$.
\end{lemma}

This lemma  can be used to construct measurable partitions whose refinements are finer than $\tilde Q$.
More precisely: 

\begin{lemma}\label{l.partition.Mane}
For any $0<\delta<1$ there exists a measurable partition $\tilde \cP$ of $\tilde S$ such that
$H_{\tilde\eta}(\tilde\cP)<\infty$, $\diam(\tilde\cP(\om,\un x)) \leqslant \delta$ for 
$\tilde\eta$-almost every $(\om,\un x)$, and so that the
partition
$$
\tilde \cP^{(\infty)}
    =\bigvee_{n=0}^{+\infty} \tilde F^n \tilde\cP
$$
is finer than $\tilde Q$.
\end{lemma}

\begin{proof}
The argument is identical to the proof of \cite[Chapter VI, Proposition~5.1]{LQ}. 

\end{proof}

We are now in a position to complete the proof of Proposition~\ref{p.entropygenerating}.
Let $\tilde\cP$ be a measurable finite entropy partition so that $\tilde\cP^{(\infty)} \succ \tilde \cQ$, given
by Lemma~\ref{l.partition.Mane}. 
Hence
\begin{align*}
h_{\tilde\eta}(\tilde F,\tilde\cP)
         = h_{\tilde\eta}(\tilde F,\tilde\cP^{(\infty)})
        & = h_{\tilde\eta}(\tilde F,\tilde\cP^{(\infty)} \vee \tilde \cQ) \\
       &  = h_{\tilde\eta}(\tilde F,\tilde F^n \tilde\cP^{(\infty)} \vee \tilde \cQ) \\
      & = H_{\tilde\eta}( \tilde F^n \tilde\cP^{(\infty)} \vee \tilde \cQ \mid \tilde F^{n+1} \tilde\cP^{(\infty)} \vee \tilde F\tilde \cQ)
\end{align*}
for every $n \geqslant 1$ (here we used that $h_{\tilde\eta}(\tilde F,\tilde \zeta)=H_{\tilde\eta}(\tilde F^{-1} \tilde \zeta,\tilde \zeta)$ whenever the partition $\tilde\zeta$ satisfies $\tilde F^{-1}\tilde\zeta \succ \tilde\zeta$). 
Consequently,
\begin{equation}\label{eq:decompose}
h_{\tilde\eta}(\tilde F, \tilde\cP)
 =
 H_{\tilde\eta}(\tilde \cQ \mid \tilde F\tilde \cQ \vee \tilde F^n \tilde\cP^{(\infty)})
    +
 H_{\tilde\eta}(\tilde\cP^{(\infty)} \mid \tilde F^{-n} \tilde \cQ \vee \tilde F \tilde\cP^{(\infty)}).
\end{equation}
The second term in the right hand side above is bounded by
$H_{\tilde\eta}(\tilde\cP)$, which is finite.
Then Item (3) in Proposition~\ref{p.generating.partition} implies
that it tends to zero as $n\to\infty$.  It remains to estimate the first term above, which is given by
$$
 H_{\tilde\eta}(\tilde \cQ \mid \tilde F\tilde \cQ \vee \tilde F^n \tilde\cP^{(\infty)})
        = \int - \log \tilde\eta_{(\tilde F\tilde \cQ \vee \tilde F^n \tilde\cP^{(\infty)})(\om,\un x)} (\tilde \cQ(\om,\un x))\, d\tilde\eta(\om,\un x), 
$$
where the measure $\tilde\eta_{\tilde F\tilde \cQ \vee \tilde F^n
\tilde\cP^{(\infty)}}$  denotes the conditional measure of $\tilde \eta$ with respect to the
partition $\tilde F\tilde \cQ \vee \tilde F^n \tilde\cP^{(\infty)}$. 
Notice that since the diameter of almost every element in $\tilde F^{-n+1}\tilde \cQ$ tends to
zero as $n \to\infty$, there exists a sequence of sets
$(\tilde \Gamma_n)_{n\geqslant 1}$ in $\tilde X$ so that $\lim_n \tilde\eta(\tilde \Gamma_n)=1$ and $ \tilde F\cQ(\om,\un x) \subset \tilde F^n \tilde\cP^{(\infty)}(\om,\un x) \;\text{for every}\; (\om,\un x) \in \tilde \Gamma_n$. Then
\begin{align*}
 H_{\tilde\eta}(\tilde \cQ \mid \tilde F\tilde \cQ \vee \tilde F^n \tilde\cP^{(\infty)})
        = \lim_{n \to\infty} 
        \int_{\tilde \Gamma_n} - \log \tilde\eta_{(\tilde F\tilde \cQ)(\om, \un x)} (\tilde \cQ(\om,\un x)) \, d\tilde\eta(\om,\un x),
\end{align*}
where the measure $\tilde\eta_{\tilde F\tilde \cQ}$ is the conditional measure of $\tilde \eta$ with respect to the
partition $\tilde F\tilde \cQ$. This proves that, taking the limit as $n\to\infty$ in \eqref{eq:decompose},
\begin{equation*}
h_{\tilde\eta}(\tilde F, \tilde\cP)
=
 \lim_{n\to\infty} \int_{\tilde \Gamma_n} - \log \tilde\eta_{(\tilde F\tilde \cQ)(\om, \un x)} (\tilde \cQ(\om,\un x)) \, d\tilde\eta(\om,\un x)
	=  H_{\tilde\eta}(\tilde \cQ \mid \tilde F\tilde \cQ).
\end{equation*}
We conclude that $h_{\tilde\eta}(\tilde F)=H_{\tilde\eta}(\tilde \cQ \mid \tilde F\cQ)=H_{\tilde\eta}(\tilde F^{-1}\tilde \cQ \mid \cQ)$, proving the proposition.

\bigskip

\subsection*{Acknowledgements}
MS was partially supported by CNPq - Brazil, SS
was partially supported by a PNPD-CAPES Postdoctoral Fellowship at UFBA, and PV was partially
supported by CNPq-Brazil and by Funda\c c\~ao para a Ci\^encia e Tecnologia (FCT) - Portugal, through the grant CEECIND/03721/2017 of the Stimulus of Scientific Employment, Individual Support 2017 Call. We are indebted to the three anonymous referees for the careful reading of the manuscript and number of suggestions that helped to improve the manuscript.

%


\begin{thebibliography}{Bibliography}


\bibitem{AMO}
A.~Arbieto, C.~Matheus, and K.~Oliveira.
\newblock Equilibrium states for random non-uniformly expanding maps.
\newblock {\em Nonlinearity}, 17:581--593, 2004.

\bibitem{Arn98b}
L.~Arnold.
\newblock {\em Random dynamical systems}.
\newblock Springer-Verlag, 1998.

\bibitem{Atnip}
J. Atnip, G. Froyland, C. Gonz\'alez-Tokman and S. Vaienti.
\newblock Thermodynamic formalism for random weighted covering systems.
\newblock Preprint arXiv:2002.11421v1

\bibitem{BRS}
W. Bahsoun, M. Ruziboev and B. Saussol, 
Linear response for random dynamical systems, 
\emph{Adv. Math} 364 (2020) 107011.



\bibitem{BaY93}
V.~Baladi and L.-S. Young.
\newblock On the spectra of randomly perturbed expanding maps.
\newblock {\em Comm. Math. Phys.}, 156:355--385, 1993.


\bibitem{BO17}
R. Bilbao and K. Oliveira,
Maximizing entropy measures for random dynamical systems,
 \emph{Stoch. $\&$ Dynam.} 17:05 (2017)  1750032 


\bibitem{Bo92}
T.~Bogensch{\"u}tz.
\newblock Entropy, pressure, and a variational principle for random dynamical
  systems.
\newblock {\em Random Comput. Dynam.}, 1(1):99--116, 1992/93.


\bibitem{BG95}
T. Bogensch\"utz and M. Gundlach. 
Ruelle's transfer operator for random subshifts of finite type. 
\newblock {\em Ergod.Th. \& Dynam. Sys.} 15 (1995), 413--447.

\bibitem{BL98} T. Bogensch\"utz and P.-D. Liu
Characterization of measures satisfying the Pesin entropy formula for random dynamical systems.
\newblock {\em J. Dynam. Diff. Eq.},  10:3 (1998), 425--448.

\bibitem{BO}
T.~Bogensch{\"u}tz and G. Ochs.
The Hausdorff dimension of conformal repellers under random perturbation,
\emph{Nonlinearity} 12 (1999) 1323--1338. 

\bibitem{BCV}
T. Bomfim, A. Castro and P. Varandas, 
\newblock Differentiability of thermodynamical quantities, 
\newblock {\em Adv. Math.}, 292 (2016), 478--528.


\bibitem{Bo75}
R.~Bowen.
\newblock {\em Equilibrium states and the ergodic theory of {A}nosov
  diffeomorphisms}, volume 470 of {\em Lect. Notes in Math.}
\newblock Springer Verlag, 1975.

		
\bibitem{CLR}
Y. Cao, S. Luzzatto and I. Rios, 
Uniform hyperbolicity for random maps with positive Lyapunov exponents. 
\emph{Proc. Amer. Math. Soc.} 136 (2008) 3591--3600.
		
\bibitem{CV13}
A. Castro and P. Varandas, Equilibrium states for non-uniformly expanding maps: decay of correlations and strong stability, 
\emph{Ann. I. H. Poincar\'e - AN}, 30:2 (2013) 225-249.

		
\bibitem{Cra02}
H.~Crauel.
\newblock {\em Random probability measures on {P}olish spaces}, volume~11 of
  {\em Stochastics Monographs}.
\newblock Taylor \& Francis, London, 2002.
				

\bibitem{DKS08}
M. Denker, Yu. Kifer and M. Stadlbauer.
Thermodynamic formalism for random countable Markov shifts.
\emph{Disc. Cont. Dyn. Sys.} 22 (2008), 131--164.


\bibitem{DU}
M.~Denker and M.~Urba{\'n}ski.
\newblock Ergodic theory of equilibrium states for rational maps.
\newblock {\em Nonlinearity}, 4:103--134, 1991.

\bibitem{DUb}
M.~Denker and M.~Urba{\'n}ski.
\newblock Hausdorff and conformal measures on {J}ulia sets with a rationally
  indifferent periodic point.
\newblock {\em J. London Math. Soc.}, 43:107--118, 1991.

\bibitem{Ha}
Y. Hafouta.
Limit theorems for some skew products with mixing base maps.
\emph{Ergod. Th. Dynam. Sys.}  https://doi.org/10.1017/etds.2019.48

\bibitem{Fed69}
H.~Federer.
\newblock {\em Geometric measure theory}.
\newblock Die Grundlehren der Mathematischen Wissenschaften, Band 153.
  Springer-Verlag, 1969.

\bibitem{HV05}
V. Horita and M. Viana. 
Hausdorff dimension for non-hyperbolic repellers. II. DA diffeomorphisms. 
\emph{Discrete Contin. Dyn. Syst.}, 13 (2005) 1125--1152.

\bibitem{KhKi96}
K.~Khanin and Y.~Kifer.
\newblock Thermodynamic formalism for random transformations and statistical
  mechanics.
\newblock In {\em Sina\u\i's {M}oscow {S}eminar on {D}ynamical {S}ystems},
  volume 171 of {\em Amer. Math. Soc. Transl. Ser. 2}, pages 107--140. Amer.
  Math. Soc., Providence, RI, 1996.

\bibitem{Ki91}
Y. Kifer.
 Large deviations for random expanding maps. In: Arnold L., Crauel H., Eckmann JP. (eds) 
 Lyapunov Exponents. Lecture Notes in Mathematics, vol 1486. Springer, Berlin, Heidelberg.

\bibitem{Ki92}
Yuri Kifer.
\newblock Equilibrium states for random expanding transformations.
\newblock {\em Random Comput. Dynam.}, 1(1):1--31, 1992/93.


\bibitem{Ki01}
Y.~Kifer,
On the topological pressure for random bundle transformations, 
Topology, ergodic theory, real algebraic geometry, 
Amer. Math. Soc. Transl. Ser. 2, vol. 202, Amer. Math. Soc., Providence, RI, 2001, pp. 197--214.


\bibitem{Ki08}
Y.~Kifer.
\newblock Thermodynamic formalism for random transformations revisited.
\newblock {\em Stoch. Dyn.}, 8(1):77--102, 2008.


\bibitem{Le84a}
F.~Ledrappier.
\newblock Propri{\'e}t{\'e}s ergodiques des mesures de {S}ina{\"\i}.
\newblock {\em Publ. Math. I.H.E.S.}, 59:163--188, 1984.

\bibitem{LY85}
F. Ledrappier and L.-S. Young, 
The metric entropy of diffeomorphisms: Part II: Relations between entropy, exponents and dimension,
Annals Math. 122:3. (1985), pp. 540--574.

\bibitem{LY98}
F. Ledrappier and L.-S. Young, 
Entropy formula for random transformations. \emph{Prob. Th. Rel. Fields} 80 (1988)  217--240.


\bibitem{Li01}
P.-D. Liu.
\newblock Dynamics of random transformations smooth ergodic-theory.
\newblock \emph{Ergod. th. \& Dynam. Sys.}, 21: 1279--1319, 2001.

\bibitem{LiuMZ}
P.-D. Liu.
Entropy formula of Pesin type for noninvertible random dynamical systems. 
\newblock \emph{Math. Z.} 230 (1999), 201--239.

\bibitem{LQ}
P.-D. Liu and M. Qian. Smooth Ergodic Theory of Random Dynamical Systems (Lecture Notes in
Mathematics, 1606). Springer, 1995.

\bibitem{LQZ}
P.-D. Liu, M. Qian and F.-X.  Zhang,
\newblock Entropy formula of Pesin type for one-sided stationary random maps.
\newblock \emph{Ergod. th. \& Dynam. Sys.}, 22 (2002) 1831--1844.

\bibitem{Liv}
C. Liverani,
\newblock Decay of correlations,
\newblock \emph{Annals of Math.}, 142 (1995) 239--301.


\bibitem{Man81}
R.~Ma{\~{n}}{\'{e}}.
\newblock A proof of {P}esin's formula.
\newblock {\em Ergod. Th. {\&} Dynam. Sys.}, 1:95--101, 1981.

\bibitem{MSU08}
V. Mayer and B. Skorulski and M. Urba{\'n}ski,
\newblock Random Distance Expanding Mappings, Thermodynamic Formalism, Gibbs Measures, and Fractal Geometry
\newblock  Lecture Notes in Math 2036, Springer (2011).

\bibitem{MU15}
V. Mayer and M. Urba{\'n}ski. 
\newblock Countable alphabet random subhifts of finite type with weakly positive transfer operator. 
\newblock \emph{J. Stat. Phys.}, 160:5 (2015)1405--1431.

\bibitem{MU18}
V. Mayer and M. Urba{\'n}ski,
Random Dynamics of Transcendental Functions,	
\emph{J. d'Analyse Math.}, 134 (2018), 201-235

\bibitem{MiU18}
E. Mihailescu and M. Urba{\'n}ski,	
Random countable iterated function systems with overlaps and applications,	
\emph{Adv. Math.} 298 (2016), 726--758.


\bibitem{Pe77}
Ya. Pesin. Lyapunov characteristic exponents and smooth ergodic theory. 
\emph{Russian Math. Surveys}, 32:4 (1977), 55--114.

\bibitem{Pe97}
Ya. Pesin.
\newblock {\em Dimension theory in dynamical systems}.
\newblock University of Chicago Press, 1997.
\newblock Contemporary views and applications.

\bibitem{PP}
Ya. B. Pesin, B. S. Pitskel, 
Topological pressure and the variational principle for noncompact sets, 
\emph{Funct. Anal. Appl.,} 18:4 (1984), 307--318.

\bibitem{Pliss}
V. Pliss, A hypothesis due to Smale, \emph{Diff. Eq.} 8 (1972), 203--214.


\bibitem{RV}
V. Ramos and M. Viana.
Equilibrium states for hyperbolic potentials
\emph{Nonlinearity} 30 (2017) 825.


\bibitem{Ro61}
V.~A. Rohklin.
\newblock Exact endomorphisms of a {L}ebesgue space.
\newblock {\em Izv. Akad. Nauk SSSR Ser. Mat.}, 25 (1961) 499--530.


\bibitem{Stadl}
M. Stadlbauer,
On random topological Markov chains with big images and preimages. 
\emph{Stoch. $\&$ Dynam.} 10 (2010), 77-95.

\bibitem{Stadl2}
M. Stadlbauer,
Coupling methods for random topological Markov chains
On random topological Markov chains with big images and preimages. 
\emph{Ergod. Th.  Dynam. Sys.} 37 (2017), 971--994.

\bibitem{SVZ}
M. Stadlbauer, P. Varandas and X. Zhang,
Quenched and annealed equilibrium states for random Ruelle expanding maps and applications,
Preprint 2020.


\bibitem{VV10}
P.~Varandas and M.~Viana.
\newblock Existence, uniqueness and stability of equilibrium states for non-uniformly expanding maps.
\newblock {\em Ann. I. H. Poincar\'e - AN}, 27:555--593, 2010.

\bibitem{Zhu}
Y. Zhu,
\newblock On local entropy of random transformations,
 \emph{Stoch. $\&$ Dynam.},   8:2 (2008) 197--207.


\end{thebibliography}
\end{document}